\documentclass[11pt]{amsart}

\usepackage[left=3cm, right=3cm, bottom=3cm, top=3cm]{geometry}

\usepackage{tikz}
\usepackage{graphicx}
\usepackage{tikz-cd}

\newtheorem{theorem}{Theorem}[section]
\newtheorem{lemma}[theorem]{Lemma}
\newtheorem{proposition}[theorem]{Proposition}
\newtheorem{corollary}[theorem]{Corollary}

\theoremstyle{definition}
\newtheorem{definition}[theorem]{Definition}
\newtheorem{condition}[theorem]{Condition}
\newtheorem{question}[theorem]{Question}

\theoremstyle{remark}
\newtheorem{remark}[theorem]{Remark}

\numberwithin{equation}{section}

\newcommand{\abs}[1]{\lvert#1\rvert}
\newcommand{\norm}[1]{\lvert\lvert#1\rvert\rvert}

\begin{document}

\title{Displacement energy of Lagrangian 3-spheres}

\author{Yuhan Sun}
\address{Department of Mathematics, Rutgers University, 110 Frelinghuysen Rd, Piscataway NJ 08854.}
\address{Department of Mathematics, Stony Brook University, Stony Brook, New York, 11794.}

\email{sun.yuhan@rutgers.edu}

\begin{abstract}
We estimate the displacement energy of Lagrangian 3-spheres in a symplectic 6-manifold $X$, by estimating the displacement energy of a one-parameter family $L_{\lambda}$ of Lagrangian tori near the sphere. The proof establishes a new version of Lagrangian Floer theory with cylinder corrections, which is motivated by the change of open Gromov-Witten invariants under the conifold transition. We also make observations and computations on the classical Floer theory by using the symplectic sum formula and Welschinger invariants.
\end{abstract}

\date{\today}

\maketitle


\tableofcontents

\section{Introduction}
Let $X$ be a closed symplectic manifold and $L$ be a closed Lagrangian submanifold. A classical problem in symplectic topology cares about the dynamic of $L$ under Hamiltonian isotopies. In particular $L$ is called \textit{nondisplaceable} if it cannot be separated from itself by any Hamiltonian diffeomorphism. That is,
$$
L\cap \phi\left(L\right) \neq \emptyset, \quad \forall \phi\in Ham\left(X, \omega\right).
$$
Otherwise $L$ is called \textit{displaceable}. For a displaceable Lagrangian submanifold, there is a notion of displacement energy to characterize how much effort one needs to displace it away. Let $H_{t}$ be a time-dependent Hamiltonian function on $X$ for $t\in[0, 1]$ and $\phi_{t}$ be the corresponding Hamiltonian isotopy. The Hofer length of $H_{t}$ is defined as
$$
\norm{H_{t}}_{X}= \int_{0}^{1} (\max_{X} H_{t}- \min_{X} H_{t}) dt
$$
and the displacement energy of $L$ is defined as
$$
\mathcal{E}_{L}= \inf\lbrace \norm{H_{t}}_{X}\mid L\cap \phi_{1}\left(L\right)= \emptyset\rbrace.
$$
If $L$ is nondisplaceable then $\mathcal{E}_{L}$ is defined to be infinity.

By the work of Gromov \cite{G} and Chekanov \cite{Ch1,Ch2,O}, the displacement energy is closely related to the least energy of a holomorphic disk with boundary on $L$. Later this relation has been extended by Fukaya-Oh-Ohta-Ono to the torsion part \cite{FOOO, FOOO4} of the Lagrangian Floer cohomology of $L$, possibly equipped with bounding cochains and deformed by bulk cycles. Hence it gives us finer estimates on the displacement energy. In this note we establish a new version of Lagrangian Floer cohomology counting more general bordered Riemann surfaces and study its torsion part. As an application we obtain some new estimates of the displacement energy of Lagrangian 3-spheres in a symplectic 6-manifold.

More precisely, this new version of Lagrangian Floer theory not only counts holomorphic strips with Lagrangian boundary conditions, but also counts holomorphic strips with one interior hole, where the interior hole is mapped to another reference Lagrangian submanifold. Usually this reference Lagrangian is a chosen Lagrangian 3-sphere. This extra counting of holomorphic cylinders enables us to use certain \textit{chains} to deform the Floer theory, generalizing the notion of bulk cycles.

Counting holomorphic cylinders between two non-intersecting Lagrangian submanifolds provides us a map between some quantum invariants of these two Lagrangian submanifolds. For an incomplete list, see \cite{BC2} and \cite{K} for some geometric applications. In our current setting, this Floer theory is motivated by various works around the \textit{conifold transition}, a surgery that replaces a Lagrangian 3-sphere by a holomorphic $\mathbb{C}P^{1}$, introduced by Smith-Thomas-Yau \cite{STY}. How geometric invariants change under this transition is an important question in the fields of symplectic topology and enumerative geometry. Particularly, some closed Gromov-Witten invariants with point-wise constraints are not preserved under this transition, unless one also takes the open Gromov-Witten invariants on $S^{3}$ into account. From this point of view, to compare the Lagrangian Floer theory of a Lagrangian away from the holomorphic $\mathbb{C}P^{1}$ in the resolved side, it is natural to consider the contributions of bordered curves with disconnected boundaries on both of the sphere $S^{3}$ and the Lagrangian. So here we realize this idea in a simple version, where both holomorphic strips and holomorphic strips with one interior hole attached on $S^{3}$ are counted. Similar philosophy already started to play an important role in the mirror symmetry ground. This article can be regarded as an application, maybe the first one, to symplectic topology.

However, the above philosophy often expects that the data of all genera should be considered, otherwise what one obtained is not an invariant. Therefore our baby theory only works modulo some energy. Recently, the open Gromov-Witten theory in $T^{*}S^{3}$ with all genera has been successfully related to knot-theoretic invariants by Ekholm-Shende \cite{ES}. It would be interesting to try to apply the techniques therein to define a full genus Floer theory, starting with the monotone Lagrangian torus in $T^{*}S^{3}$. Hopefully there will be a correspondence between open Gromov-Witten invariants with coefficients in skein modules and bulk-deformed open Gromov-Witten invariants. Then one may move further to toric compactifications or other general cases, to see how Lagrangian Floer theory and even Fukaya category change under the conifold transition.

\subsection{Main results}
Let $X$ be a closed symplectic 6-manifold with a Lagrangian 3-sphere $S$. We say $S$ is integrally homologically trivial if the inclusion map $i:H_{3}(S;\mathbb{Z})\rightarrow H_{3}(X;\mathbb{Z})$ is a trivial map. For an oriented Lagrangian submanifold $L$ in $X$, we will study holomorphic disks with boundary on $L$ and holomorphic cylinders with one boundary on $L$ and the other on $S$. The following condition is designed to make various moduli spaces behave nicely.

\begin{condition}\label{condition}
There exists a compatible almost complex structure $J$ and two positive numbers $E_{+}\geq E_{S}$ such that
\begin{enumerate}
\item all non-constant $J$-holomorphic disks on $L$, with energy less than $E_{+}$, have positive Maslov indices;
\item all Maslov two $J$-holomorphic disks on $L$, with energy less than $E_{+}$, are regular;
\item all non-constant $J$-holomorphic spheres, with energy less than $E_{+}$, have positive first Chern numbers.
\item all $J$-holomorphic disks on $S$, with energy less than $E_{S}$, are constant (this includes the case of a $J$-holomorphic sphere with one point attached on $S$);
\item all $J$-holomorphic cylinders with one boundary on $L$ and the other on $S$, with energy less than $E_{S}$ have positive Maslov index.
\end{enumerate}
\end{condition}

This condition is an open condition except item $(2)$. In this article, we expect that the transversality results of holomorphic curves can be obtained by using a specific $J$ satisfying Condition \ref{condition} via virtual perturbation \cite{FOOO, FOOO6}. Or we fix an open neighborhood of a $J$ satisfying $(1), (3), (4), (5)$ then use families of almost complex structures in this neighborhood via classical means, also see Remark \ref{div} and Remark \ref{analytic}.

Now for a triple $(X, S, L)$ indicated before, moreover suppose that there is a 4-chain $K$ in $X$ such that $\partial K=S$ and $K\cap L=\emptyset$. Note that $L$ is an orientable 3-manifold hence spin. Then by the work of Fukaya-Oh-Ohta-Ono \cite{FOOO}, there is an $A_{\infty}$-structure $\lbrace \mathfrak{m}_{k}\rbrace$ on the singular cohomology of $L$, modulo $T^{E_{+}}$. The main point of this article is to use the chain $K$ as a bulk deformation to deform $\lbrace \mathfrak{m}_{k}\rbrace$.

\begin{theorem}\label{construction}
Let $X$ be a closed symplectic 6-manifold with an integrally homologically trivial Lagrangian 3-sphere $S$, and let $L$ be an oriented Lagrangian submanifold of $X$ satisfying Condition \ref{condition}. For a 4-chain $K$ such that $\partial K=S$ and $K\cap L=\emptyset$, let
$$
\mathfrak{b}=w\cdot K, \quad w\in\Lambda_{0}, \quad v(w)>0
$$
and $E:=\min\lbrace E_{S}+ v(w), 2v(w), E_{+}\rbrace$. Then there is an $A_{\infty}$-structure $\lbrace \mathfrak{m}^{cy, \mathfrak{b}}_{k}\rbrace$ on the singular cohomology of $L$, with coefficients in $\Lambda_{0}\big/ T^{E}\cdot\Lambda_{0}$.
\end{theorem}

Here $\Lambda_{0}$ is the universal Novikov ring with a valuation $v$, see Section 2 for our conventions. This $A_{\infty}$-structure $\lbrace \mathfrak{m}^{cy, \mathfrak{b}}_{k}\rbrace$ is a deformation of the $A_{\infty}$-structure $\lbrace \mathfrak{m}_{k}\rbrace$ in \cite{FOOO}, deformed by the chain $K$ with Novikov ring coefficients. In particular, we count not only holomorphic disks but also holomorphic cylinders with boundaries on $L$ and $S$.

An example is that when both $X$ and $L$ are monotone, then the conditions can be achieved by taking $E_{+}=+\infty$ and $E_{S}$ as the minimal area of a symplectic disk with boundary on $S$.

Similar to the case of the usual $A_{\infty}$-structure, we can also use $\lbrace \mathfrak{m}^{cy, \mathfrak{b}}_{k}\rbrace$ to construct a Lagrangian Floer cohomology of $L$, deformed by the chain $K$. In Theorem \ref{construction}, the Lagrangian submanifold $L$ can be any closed oriented 3-manifold with a spin structure. Next we focus on the case where $L$ is a Lagrangian torus close to $S$, to get some geometric applications. We start with the local geometry near a Lagrangian 3-sphere. Let $S^{3}$ be a 3-sphere and $(T^{*}S^{3}, \omega_{0})$ be the total space of its cotangent bundle equipped with the standard symplectic form. It is known that there is a one-parameter family of Lagrangian tori $\lbrace L_{\lambda}\rbrace_{\lambda\in (0, +\infty)}$ in $(T^{*}S^{3}, \omega_{0})$ such that
\begin{enumerate}
\item $L_{\lambda}$ is monotone with a monotonicity constant $\lambda$ and has minimal Maslov number two;
\item $L_{\lambda}$ has nonzero Floer cohomology with certain weak bounding cochains, hence it is nondisplaceable in $T^{*}S^{3}$;
\item for any neighborhood of the zero section $S^{3}$, $L_{\lambda}$ is contained in this neighborhood if $\lambda$ is small enough.
\end{enumerate}
We will review the explicit construction in Section 3 following \cite{CKO} and \cite{CPU}, where they computed the Gromov-Witten disk potential of $L_{\lambda}$. Moreover, by the local study there is a compatible almost complex structure $J_{0}$ on $(T^{*}S^{3}, \omega_{0})$ such that
\begin{enumerate}
\item $J_{0}$ is cylindrical outside a large compact set;
\item all $J_{0}$-holomorphic disks on $L_{\lambda}$ with Maslov index two are regular;
\item the images of all $J_{0}$-holomorphic disks on $L_{\lambda}$ with Maslov index two are outside a neighborhood $V_{\lambda}$ of the zero section.
\end{enumerate}
A compactly supported small generic perturbation of $J_{0}$ also satisfies above properties.

Then let $S$ be a Lagrangian 3-sphere in a symplectic 6-manifold $X$ and $U$ be a Weinstein neighborhood of $S$ which is symplectomorphic to some disk cotangent bundle $(D_{r}T^{*}S^{3}, \omega_{0})$. A subfamily of $\lbrace L_{\lambda}\rbrace_{\lambda\in (0, +\infty)}$ sits in $U=D_{r}T^{*}S^{3}$ and one can ask whether $L_{\lambda}$ is globally nondisplaceable in $X$. Note that if $L_{\lambda}$ is nondisplaceable for all small $\lambda$ then the Lagrangian sphere $S$ is also nondisplaceable. We will use this approach to obtain some estimates of the displacement energy of $S$ by estimating the displacement energy of $L_{\lambda}$ near it.

The above idea is motivated by concrete examples in \cite{CKO} and \cite{P}. Let $F_{3}$ be the manifold of full flags in $\mathbb{C}^{3}$. When $F_{3}$ is equipped with a monotone symplectic form, a Lagrangian 3-sphere in $F_{3}$ with vanishing Floer cohomology was found in \cite{NNU}. Later in \cite{CKO} this Lagrangian sphere is shown to be nondisplaceable by showing the local one-parameter family of Lagrangian tori $L_{\lambda}$ near it has nontrivial Floer cohomology. When $F_{3}$ is equipped with a non-monotone symplectic form then the ``same'' Lagrangian 3-sphere is proved to be displaceable, see \cite{P}.

Both \cite{CKO} and \cite{P} use explicit geometric properties of $F_{3}$. And here we try to study a general theory without knowing much about the ambient symplectic manifold $X$. One difficulty is that for a general ambient symplectic manifold, there is no ``canonical'' ambient 4-cycle to deform the Floer cohomology to be non-zero. Locally in $T^{*}S^{3}$ there are only 4-chains with boundary as the zero section. Directly using these chains to deform will cause that some boundary operators do not have zero square, since the 4-chain has a codimension one boundary. Our strategy is to consider the moduli space of holomorphic cylinders to cancel this possible boundary effect, such that those 4-chains can be used as deformations.

Assuming $(1)-(3)$ in Condition \ref{condition}, the one-pointed open Gromov-Witten invariant $n_{\beta}$ is defined, for any Maslov two disk class $\beta\in\pi_{2}(X, L)$ with energy less than $E_{+}$. We consider the sequence
$$
\lbrace \beta_{k}\mid n_{\beta}\neq 0, E(\beta_{k})\leq E(\beta_{k+1})\rbrace_{k=1}^{\infty}
$$
of disk classes with Maslov index two, enumerated by their symplectic energy, see Figure 1. From the local study we know that $L_{\lambda}$ bounds four $J$-holomorphic disks with Maslov index two inside $U$, with same energy $E_{1,\lambda}$. There maybe other Maslov two holomorphic disks of which the images are not contained in $U$. We call them outside disk contributions. If $L_{\lambda}$ is near $S$, the local disk contributions are the first four elements in the above sequence. This is why we need $\lambda_{0}$ small in Theorem \ref{dis}. Now let $E_{5,\lambda}=E(\beta_{5})$ be the least energy of outside disk contributions. Then when $L_{\lambda}$ is close to $S$ we have that $E_{5,\lambda}\gg E_{1,\lambda}$.

As we mentioned before, the $A_{\infty}$-structure $\lbrace \mathfrak{m}^{cy, \mathfrak{b}}_{k}\rbrace$ is a deformation of the usual $A_{\infty}$-structure, by using holomorphic cylinder counting. Let $\beta$ be a class of cylinders with one end on $L$, the other end on $S$, with Maslov index two. Consider the moduli space of holomorphic cylinders $\mathcal{M}^{cy}_{1}(\beta; J)$ representing $\beta$, with one boundary marked point on the side of $L$. By Condition \ref{condition}, there is no disk bubble or sphere bubble off the cylinders, if $\omega(\beta)<E_{S}, E_{+}$. The only possible boundary of $\mathcal{M}^{cy}_{1}(\beta; J)$ is where the circle boundary of a cylinder on $S$ shrinking to a point. And we will glue $\mathcal{M}^{cy}_{1}(\beta; J)$ with another moduli space of holomorphic disks with one interior marked point to cancel this boundary. However, note that there is an almost complex structure $J_{0}$ on $(T^{*}S^{3}, \omega_{0})$ such that the images of all $J_{0}$-holomorphic disks on $L_{\lambda}$ with Maslov index two are outside a neighborhood $V_{\lambda}$ of the zero section. In this case, the disk counting and cylinder counting are separated. Then we can count such holomorphic cylinders by defining the following mapping degree
$$
n_{\beta}^{cy}:=\deg (ev: \mathcal{M}^{cy}_{1}(\beta; J)\rightarrow L) \in \mathbb{Q}.
$$
The next condition gives some control on the local cylinder contributions.

\begin{condition}\label{condition+}
There exists a compatible almost complex structure $J$ satisfying Condition \ref{condition} and the following
\begin{enumerate}
\setcounter{enumi}{5}
\item the images of all $J$-holomorphic disks on $L_{\lambda}$ with $\mu(\beta)=2, \omega(\beta)<E_{5, \lambda}$ are outside a neighborhood $V_{\lambda}$ of $S$;
\item $0\neq 1+n_{\beta_{1}}^{cy}+n_{\beta_{2}}^{cy}-n_{\beta_{3}}^{cy}-n_{\beta_{4}}^{cy}$;
\item $0= n_{\beta}^{cy}$ for any class $\beta$ with $\mu(\beta)=2, \omega(\beta)<E_{5, \lambda}, \beta\neq \beta_{i}$.
\end{enumerate}
\end{condition}

Note that $(6)$ can be achieved by assuming $J=J_{0}$ on $U$, $(7)$ and $(8)$ may put non-trivial constraints. Nonetheless, we give an argument that after degenerating the complex structure near $S$, there is no holomorphic disk or cylinder touching $S$ with Maslov index two. Hence $n_{\beta}^{cy}$ is always zero and $(7), (8)$ are satisfied. However, this degenerate almost complex structure may not satisfy other assumptions in Condition \ref{condition}. Hence we view them as assumptions, rather than consequences, under which we can perform computations to obtain geometric applications from this deformed $A_{\infty}$-structure. On the other hand, the mapping degree of some disk class (or the one-pointed open Gromov-Witten invariant of that class) is known to be one. Then Condition \ref{condition+} helps to show that after adding cylinder contributions the leading order terms of the potential function have non-degenerate critical points. So we can use an implicit function theorem to perturb away higher energy terms.

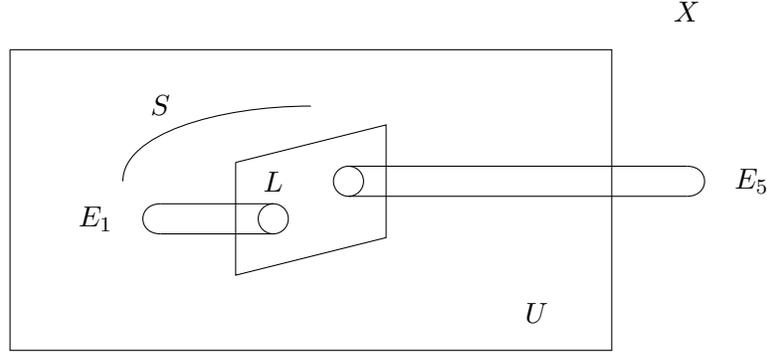
\begin{figure}
  \begin{tikzpicture}[xscale=1, yscale=0.5]
  \draw (-1,1)--(1,2)--(1,-1)--(-1,-2)--(-1,1);
  \draw (-4,4)--(4,4)--(4,-4)--(-4,-4)--(-4,4);
  \draw (-2.5,0.5) to [in=180,out=90] (0,2.5);
  \draw (0.5,0.5) ellipse (0.2 and 0.4);
  \draw (0.5,0.9)--(5,0.9);
  \draw (0.5,0.1)--(5,0.1);
  \draw (5,0.9) to [out=0,in=0] (5,0.1);
  \draw (-0.5,-0.5) ellipse (0.2 and 0.4);
  \draw (-0.5,-0.9)--(-2,-0.9);
  \draw (-0.5,-0.1)--(-2,-0.1);
  \draw (-2,-0.9) to [out=180,in=180] (-2,-0.1);
  \node [above] at (-2,2) {$S$};
  \node at (3,-3) {$U$};
  \node at (-0.5,0.5) {$L$};
  \node at (5,5) {$X$};
  \node [left] at (-2.5,-0.5) {$E_{1}$};
  \node [right] at (5.5,0.5) {$E_{5}$};
  \end{tikzpicture}
  \caption{Disk contributions from inside and outside.}
\end{figure}

\begin{theorem}\label{dis}
Let $X$ be a closed symplectic 6-manifold which contains an integrally homologically trivial Lagrangian 3-sphere $S$. Consider the Lagrangian embedding
$$
L_{\lambda}\hookrightarrow U=D_{r}T^{*}S^{3}\subset X
$$
for $\lambda\in (0, \lambda_{0})$ where $\lambda_{0}$ is a sufficiently small number. If $L_{\lambda}$ satisfies Condition \ref{condition+}, and it can be displaced by a Hamiltonian isotopy $\phi_{t}$ generated by $G_{t}$, then
$$
\norm{G_{t}}_{X} \geq E_{5,\lambda}
$$
and
$$
\norm{G_{t}}_{X} +2\norm{G_{t}}_{S} \geq \min\lbrace E_{S}+E_{5,\lambda}- E_{1,\lambda}, 2(E_{5,\lambda}- E_{1,\lambda}), E_{+}\rbrace.
$$
\end{theorem}

Here $\norm{\cdot}_{X}$ is the Hofer norm and $\norm{\cdot}_{S}$ is a relative Hofer norm defined by
$$
\norm{G_{t}}_{S}= \int_{0}^{1} (\max_{S} G_{t}- \min_{S} G_{t}) dt.
$$
Theorem \ref{dis} is an explicit application of Theorem \ref{construction}, where a certain bulk chain $wK$ is used with $v(w)=E_{5,\lambda}- E_{1,\lambda}$. By definition we know that $\norm{G_{t}}_{X}\geq \norm{G_{t}}_{S}$. But for the above two inequalities we can not say which one is stronger, unless we know the behavior of $G_{t}$ on $S$. For example, the displaceable Lagrangian sphere $S$ in $F_{3}$ can be displaced by a group action. In particular the Hamiltonian function is constant on $S$, hence $\norm{G_{t}}_{S}=0$ and the second inequality is much stronger than the first one and almost optimal, see Section 6.3.

Note that our family of tori approaches the sphere infinitesimally, as a corollary we obtain an estimate of the displacement energy of our Lagrangian sphere $S$.

\begin{corollary}\label{sphere}
With the same notation in Theorem \ref{dis} and assuming that $E_{S}\geq E_{5,\lambda}, E_{+}\geq 2E_{5,\lambda}$, if $S$ can be displaced by a Hamiltonian isotopy $\phi_{t}$ generated by $G_{t}$ then
$$
\norm{G_{t}}_{X} \geq  \lim_{\lambda\rightarrow 0} E_{5,\lambda}, \quad \norm{G_{t}}_{X} +2\norm{G_{t}}_{S} \geq \lim_{\lambda\rightarrow 0} 2 (E_{5,\lambda}- E_{1,\lambda})= \lim_{\lambda\rightarrow 0} 2 E_{5,\lambda}.
$$
\end{corollary}

As $\lambda$ tends to zero, the parameter $E_{1, \lambda}$ tends to zero and $E_{5,\lambda}- E_{1,\lambda}$ increases to $E_{5,\lambda=0}$. The energy $E_{5,\lambda=0}$ is roughly the least energy of a holomorphic disk with boundary on $S$. (This also explains why the assumption that $E_{S}\geq E_{5,\lambda}$ is somehow reasonable.) Then the Hofer norm of the Hamiltonian which displaces $S$ is roughly twice the least energy of a holomorphic disk, with a modification term given by the relative Hofer norm.

Next we remark about the applicable range of the above theorems. Condition \ref{condition} is a strong condition, where $(1)-(3)$ are usually assumed to define potential functions and Lagrangian Floer cohomology for a non-monotone Lagrangian torus, see Assumption 3.2 in \cite{Au}. And $(4),(5)$ are technical for us to avoid possible disk and cylinder bubbles on the Lagrangian sphere $S$, with small energy.

In addition to the monotone case, possible examples which satisfy Condition \ref{condition} come from the toric fiber of a symplectic Fano (almost) toric manifold. When $J$ is toric, $(1)-(3)$ are true for any $E_{+}>0$. But in practice we need to fix some number $E_{+}$ to keep Condition \ref{condition} an open condition (except $(2)$), by Gromov compactness theorem, see Lemma 6.4.7 and Lemma 6.4.8 in \cite{MS}. More specifically, let $X_{0}$ be a nodal toric Fano threefold and let $X$ be the smoothing of $X_{0}$. Each node gives us a Lagrangian $S^{3}$ and the local tori near the spheres become toric fibers. There is a full classification \cite{Ga} of 100 nodal toric Fano threefolds, 18 out of which are smooth. In theory one can compute explicitly all the disk potential functions of the toric fibers therein to find the torsion thresholds, by using the combinatorial data from their polytopes. But we do not try to do it here. We also assume that $S$ is homologically trivial over $\mathbb{Z}$. This condition is needed such that $S$ bounds a 4-chain in $X$ and some cylinder counting can be defined. The smoothings of nodal toric Fano threefolds still satisfy this condition. Note that rationally homologically non-trivial Lagrangian spheres are always nondisplaceable, by Theorem H \cite{FOOO}. Therefore our goal here is to use certain examples to convey the idea of Floer cohomology deformed by \textit{chains}, instead of proving results in very general cases.

Besides introducing this new version of Floer theory, we also carry out some computations of the classical Floer cohomology, which neither uses bounding cochains and bulk-deformations, nor assumes Condition \ref{condition}. Let $(X, \omega)$ be a closed symplectic 6-manifold such that
\begin{equation}\label{propo}
[c_{1}(TX)]=c \cdot[\omega], \quad c\in\mathbb{R}
\end{equation}
on the image of the Hurewicz map $\pi_{2}(X; \mathbb{Z})\rightarrow H_{2}(X; \mathbb{Z})$. We say $X$ is \textit{monotone} if $c>0$, it is \textit{Calabi-Yau} if $c=0$ and it is \textit{negatively monotone} if $c<0$. Note that $\pi_{1}(S^{3})=\pi_{2}(S^{3})=0$ implies that $\pi_{2}(X, S)\cong \pi_{2}(X)$. If (\ref{propo}) is satisfied then the two homomorphisms $c_{1}$ and $\omega$ on the relative homotopy group are also proportional to each other with the same constant $c$. In particular if $X$ is monotone then $S$ is automatically a monotone Lagrangian submanifold in the usual sense.

First, by a degeneration method \cite{FT, FTZ} from the symplectic cut and sum construction, we can determine the displaceability of $S$ and $L_{\lambda}$ when $X$ is Calabi-Yau and negatively monotone. Note that Theorem \ref{dis} uses cylinder counting to cancel the outside disk contributions to some extent, here we find that the outside disk contributions can be perturbed away. Combined with the Oakley-Usher's families \cite{OU} of monotone nondisplaceable Lagrangian submanifolds in $T^{*}S^{n}$, a general phenomenon holds in all higher dimensions.

\begin{theorem}\label{cy}
For any integer $n\geq 3$, let $(X^{2n}, S^{n}, \omega)$ be a Calabi-Yau or negatively monotone symplectic manifold with a Lagrangian sphere. Then there are continuum families of Lagrangian submanifolds
$$
L^{\lambda}_{k, m}\cong (S^{1}\times S^{k}\times S^{m})\big/\mathbb{Z}_{2}, \quad k,m \in\mathbb{Z}_{+}, k\leq m, k+m=n-1, \lambda\in (0, \lambda_{0}) \subset\mathbb{R}
$$
near the Lagrangian sphere $S$ and are nondisplaceable in $X$.
\end{theorem}

The nondisplaceability of a Lagrangian sphere in a Calabi-Yau manifold was proved in Theorem L \cite{FOOO}. And M.F.Tehrani \cite{FT} gave an alternative proof by the symplectic sum and cut method. Here we are using his approach to analyze the Lagrangian submanifolds near the sphere. For readers who are interested in the Lagrangian skeleta of a Calabi-Yau manifold, this theorem helps to show that if a symplectic manifold is the divisor complement of a Calabi-Yau manifold and contains a Lagrangian sphere, then its skeleta must intersect all those $L^{\lambda}_{k, m}$ near this sphere, see the work \cite{TVa} by Tonkonog-Varolgunes.

When the dimension $n=2$ the existence of a one-parameter family of nondisplaceable Lagrangian tori near a Lagrangian two-sphere has also been studied. First Fukaya-Oh-Ohta-Ono \cite{FOOO3} proved the existence when the ambient space is $S^{2}\times S^{2}$ and the Lagrangian sphere is the anti-diagonal. Then the general case, without assumption on $c_{1}$, is studied by the author in \cite{S}, by using the local geometry of $T^{*}S^{2}$ to control the global picture. Similar local-to-global philosophy also played an important role in \cite{CW}, \cite{TV} and \cite{Wu} for other local models, where many geometric applications are obtained.

Next we discuss the case when $X$ is monotone. Since $S^{3}$ is simply-connected, orientable and spin, the classical Floer cohomology $HF(S^{3}; \Lambda(\mathbb{Z}))$ is well-defined over the integers, which can be computed by using the pearl complex. Here $\Lambda(\mathbb{Z})$ is the Novikov ring with $\mathbb{Z}$ as the ground ring. The underlying complex is generated by critical points of a Morse function over the Novikov ring. If we use the height function to compute it, the only essential maps count pearly trajectories with Maslov index four, connecting the minimum and maximum point. For example, when $X$ has minimal Chern number $N\geq 3$, these maps are zero and we have that $HF(S^{3}; \Lambda(\mathbb{Z}))=H^{*}(S^{3}; \mathbb{Z})\otimes \Lambda(\mathbb{Z})$. When $X$ has minimal Chern number $N=2$, these maps are two-pointed open Gromov-Witten invariants of class $\beta$. When $X$ has minimal Chern number $N=1$, these maps count holomorphic disks connected by negative gradient flow lines. However, the usual two-pointed open Gromov-Witten invariant of class $\beta$ is not well-defined due to splittings of disks with Maslov index two.

On the other hand, Welschinger \cite{W} defined $F$-valued open counts of disks for a Lagrangian submanifold $L$ when $H_{1}(L; F)\rightarrow H_{1}(X; F)$ is injective, for a commutative ring $F$. Given a disk class $\beta$ and $r\geq 1$ boundary constraints, his invariant $n^{W}_{r,\beta}$ counts multi-disks weighted by linking numbers. We compare his invariants and the Floer differential and find they are equal to each other, up to sign.

\begin{theorem}\label{wel}
Let $S$ be a Lagrangian 3-sphere in a monotone symplectic 6-manifold $X$. Choose a generic triple $(f,\rho,J)$ of a Morse function $f$, a Riemannian metric $\rho$ and a compatible almost complex structure $J$. Assume that $f$ is the height function and $p$ is its minimum and $q$ is its maximum. Then given a disk class $\beta\in \pi_{2}(X, S)$ with Maslov index four, we have an equality
$$
\sharp \mathcal{P}_{prl}(p, q; \beta; f, \rho, J)= \pm n^{W}_{2,\beta}.
$$
Here the left-hand side of the equation is the signed count of pearly trajectories connecting $p$ to $q$, see Section 3.2.
\end{theorem}

Therefore we can define the following invariant
$$
n^{W}_{2}:=\sum_{\mu(\beta)=4} n^{W}_{2, \beta} \in \mathbb{Z}
$$
to determine the Floer cohomology. That is, $HF(S; \Lambda(\mathbb{Z}_{2}))\neq\lbrace 0\rbrace$ if $n^{W}_{2}$ is an even number. So we hope the Welschinger invariants help us to compute the Floer homology in certain settings, like the one-pointed open Gromov-Witten invariants in the case of toric fibers. Moreover, we expect to define similar enumerative equations for Lagrangian submanifolds of general topological type. For example, for a Lagrangian $S^{3}\times T^{n}$ we may need both two types of enumerative invariants to determine its Floer homology.

\subsection{Future questions}
Now we formulate three questions motivated by above discussions.

\begin{question}
Is it true that all monotone Lagrangian 3-spheres in a closed symplectic 6-manifold are nondisplaceable? What about higher dimensions?
\end{question}

When the ambient space is open, there are certain monotone Lagrangian $(2k+1)$-spheres in $\mathbb{C}^{k+1}\times \mathbb{C}P^{k}$ which are displaceable, see \cite{ALP} and \cite{A2}. In higher dimensions, Solomon-Tukachinsky \cite{ST} and Chen \cite{Chen} have generalized Welschinger invariants. We expect their invariants also have some meaning in Floer theory.

\begin{question}
Is it true that all Lagrangian 3-spheres in a closed Calabi-Yau manifold are not isolated?
\end{question}

By ``not isolated'' we mean for a Lagrangian sphere $S$ there is another Lagrangian submanifold $L$ such that $L$ is not Hamiltonian isotopic to $S$ and not displaceable from $S$. A stronger question would be that any Lagrangian 3-sphere in a Calabi-Yau manifold must intersect another Lagrangian 3-sphere. One possibility is that $L$ comes from a ``completion'' of the cotangent fiber of $S$. Then, how to relate the cotangent fiber generation \cite{A} to a global statement will be a deeper question.

\begin{question}
Is it true that all simply-connected Lagrangian submanifold in a Calabi-Yau manifold are nondisplaceable?
\end{question}

In the 3-dimensional case it is true since the only candidate is the 3-sphere. And this question can be viewed as a generalization for Gromov's theorem that no simply-connected Lagrangian submanifold exists in $\mathbb{C}^{n}$. If a simply-connected Lagrangian submanifold $L$ lives in a Calabi-Yau manifold, it has vanishing Maslov class. The possible appearance of Maslov zero holomorphic disks with boundary on $L$ gives obstructions to define the Lagrangian Floer cohomology. The obstructions are cohomology classes living in $H^{2}(L; \mathbb{Q})$. One treatment is to totally exclude them by assuming that $H^{2}(L; \mathbb{Q})=0$. Then this question has an affirmative answer, see Theorem 6.1.15 and Corollary 6.1.16 in \cite{FOOO}.

The outline of this article is as follows. In Section 2 we give the background on potential functions with bulk deformations. In Section 3 we review the symplectic sum and cut method and prove Theorem \ref{cy} and Theorem \ref{wel}. In Section 4 and 5 we construct three types of Floer theories with cylinder corrections and show some geometric properties of these theories. The first model is a disk model with cylinder corrections, which gives us a deformed potential function to do concrete computations. The second and third models are complexes generated by Hamiltonian chords and intersection points respectively, which will be used to study the intersection behavior of our Lagrangians under Hamiltonian perturbations. Once the equivalences between the three models is established, we apply them, in Section 6, to obtain estimates of displacement energy and prove Theorem \ref{dis}.

\subsection*{Acknowledgements}
The author acknowledges Kenji Fukaya for his enlightening guidance and support during these years. The author acknowledges Xujia Chen, Mark McLean, Yi Wang, Hang Yuan and Aleksey Zinger for helpful discussions. The author acknowledges Yunhyung Cho, Tobias Ekholm and Yong-Geun Oh for explaining their related work. The author also acknowledges the anonymous referee for lots of helpful comments and suggestions.

\section{Preliminaries}
We give a very brief summary to the theory of deformed Floer cohomology and potential functions, referring to Section 2 and Appendix 1 in \cite{FOOO3} for more details.

First we specify the ring and field that will be used. The Novikov ring $\Lambda_{0}$ and its field $\Lambda$ of fractions are defined by
$$
\Lambda_{0}=\lbrace \sum_{i=0}^{\infty}a_{i}T^{\lambda_{i}}\mid a_{i}\in \mathbb{C}, \lambda_{i}\in\mathbb{R}_{\geq 0}, \lambda_{i}<\lambda_{i+1}, \lim_{i\rightarrow \infty}\lambda_{i}=+\infty \rbrace
$$
and
$$
\Lambda=\lbrace \sum_{i=0}^{\infty}a_{i}T^{\lambda_{i}}\mid a_{i}\in \mathbb{C}, \lambda_{i}\in\mathbb{R}, \lambda_{i}<\lambda_{i+1}, \lim_{i\rightarrow \infty}\lambda_{i}=+\infty \rbrace
$$
where $T$ is a formal variable. The maximal ideal of $\Lambda_{0}$ is defined by
$$
\Lambda_{+}=\lbrace \sum_{i=0}^{\infty}a_{i}T^{\lambda_{i}}\mid a_{i}\in \mathbb{C}, \lambda_{i}\in\mathbb{R}_{>0}, \lambda_{i}<\lambda_{i+1}, \lim_{i\rightarrow \infty}\lambda_{i}=+\infty \rbrace.
$$
We remark that the field $\Lambda$ is algebraically closed since the ground field is $\mathbb{C}$, see Appendix A in \cite{FOOO1}. All the nonzero elements in $\Lambda_{0}-\Lambda_{+}$ are units in $\Lambda_{0}$. Next we define a valuation $v$ on $\Lambda$ by
$$
v(\sum_{i=0}^{\infty}a_{i}T^{\lambda_{i}})=\inf \lbrace \lambda_{i}\mid a_{i}\neq 0\rbrace, \quad v(0)=+\infty.
$$
This valuation gives us a non-Archimedean norm
$$
\abs{a=\sum_{i=0}^{\infty}a_{i}T^{\lambda_{i}}}=e^{-v(a)}.
$$

In the following we abuse the notations between singular chains and cochains via the following conventional Poincar\'e duality, see Remark 3.5.8 \cite{FOOO}. For a singular chain $x$ in $L$, the Poincar\'e dual $PD(x)$, regarded as a current satisfies that
\begin{equation}
\int_{x} \alpha\mid_{x}= \int_{L} PD(x)\wedge \alpha
\end{equation}
for any differential form $\alpha \in \Omega^{\dim x}(L)$. Geometrically when we define moduli space of holomorphic curves with point constraints, we think our curves with points attached on certain chains. But for algebraic convenience we think these chains as cochains with gradings reversed and shifted.

Let $X$ be a closed symplectic 6-manifold and $L$ be an oriented Lagrangian submanifold. Fukaya-Oh-Ohta-Ono constructed a filtered $A_{\infty}$-algebra structure on $H^{*}(L; \Lambda_{0})$ where
$$
\mathfrak{m}_{k}: H^{*}(L; \Lambda_{0})^{\otimes k}\rightarrow H^{*}(L; \Lambda_{0})
$$
are the $A_{\infty}$-operations, see section 3 in \cite{FOOO1}. The operators $\mathfrak{m}_{k}$ are defined as
$$
\mathfrak{m}_{k}(x_{1}, \cdots, x_{k})=\sum_{\beta\in \pi_{2}(X,L)} T^{\omega(\beta)}\cdot\mathfrak{m}_{k;\beta}(x_{1}, \cdots, x_{k})
$$
where geometrically $\mathfrak{m}_{k;\beta}(x_{1}, \cdots, x_{k})$ count holomorphic disks, representing the class $\beta$, with boundary marked points attached on given cocycles $(x_{1}, \cdots, x_{k})$ in $L$. We remark that the operators $\mathfrak{m}_{k}$ are first defined at the chain level then can be passed to their ``canonical model'' at the cohomology level. Here we directly use the canonical model at the cohomology level.

An element $b\in H^{1}(L; \Lambda_{+})$ is called a weak bounding cochain if it satisfies the $A_{\infty}$-Maurer-Cartan equation
\begin{equation}
\sum _{k=0}^{\infty} \mathfrak{m}_{k}(b, \cdots, b) \equiv 0 \mod PD([L]).
\end{equation}
Here $PD([L])\in H^{0}(L; \mathbb{Z})$ is the Poincar\'e dual of the fundamental class and it is the unit of the filtered $A_{\infty}$-algebra. We denote by $\mathcal{M}_{weak}(L)$ the set of weak bounding cochains of $L$. If $\mathcal{M}_{weak}(L)$ is not empty then we say $L$ is weakly unobstructed.

The coefficients of weak bounding cochains can be extended from $\Lambda_{+}$ to $\Lambda_{0}$. For $b\in H^{1}(L; \Lambda_{0})$ we can write $b=b_{0}+b_{+}$ where $b_{0}\in H^{1}(L; \mathbb{C})$ and $b_{+}\in H^{1}(L; \Lambda_{+})$. Then we define
\begin{equation}
\mathfrak{m}_{k, \beta}(b, \cdots, b):= e^{\langle \partial\beta, b_{0}\rangle}\mathfrak{m}_{k, \beta}(b_{+}, \cdots, b_{+})
\end{equation}
where the pairing $\langle \partial\beta, b_{0}\rangle= \int_{\partial\beta} b_{0}$. Note that if $b_{0}=b_{0}' +2\pi\sqrt{-1}\mathbb{Z}$ then $e^{\langle \partial\beta, b_{0}\rangle}=e^{\langle \partial\beta, b_{0}'\rangle}$. So the weak bounding cochains with $\Lambda_{0}$ coefficients are actually defined modulo this equivalence. More precisely, they should be regarded as elements in
$$
H^{1}(L; \Lambda_{0})\big/H^{1}(L; 2\pi\sqrt{-1}\mathbb{Z}):= H^{1}(L; \mathbb{C})\big/H^{1}(L; 2\pi\sqrt{-1}\mathbb{Z}) \oplus H^{1}(L; \Lambda_{+}).
$$

Now for a weak bounding cochain $b$ we can deform the $A_{\infty}$-operations in the following way. Define
$$
\mathfrak{m}_{k}^{b}(x_{1}, \cdots, x_{k}):= \sum_{l=0}^{\infty} \sum_{l_{0}+\cdots +l_{k}=l} \mathfrak{m}_{k+l_{0}+\cdots +l_{k}}(b^{\otimes l_{0}}, x_{1}, b^{\otimes l_{1}}, x_{2}, \cdots, x_{k}, b^{\otimes l_{k}}).
$$
That is, we insert $b$ in all possible ways. Then $\lbrace \mathfrak{m}_{k}^{b}\rbrace$ is a new sequence of $A_{\infty}$-operations on $H^{*}(L; \Lambda_{0})$ which satisfies that
\begin{equation}
\mathfrak{m}_{1}^{b}\circ \mathfrak{m}_{1}^{b}=0,
\end{equation}
see Proposition 3.6.10 in \cite{FOOO}. So we can define the deformed Floer cohomology $HF(L, b; \Lambda_{0})$ as the cohomology of $\mathfrak{m}_{1}^{b}$ whenever $b$ is a weak bounding cochain.

When $L$ is a torus, we define a \textit{potential function}
$$
\mathfrak{PO}: \mathcal{M}_{weak}(L)\rightarrow \Lambda_{+}
$$
by setting
$$
\sum _{k=0}^{\infty} \mathfrak{m}_{k}(b, \cdots, b)=\mathfrak{PO}(b)\cdot PD([L]).
$$

The new $A_{\infty}$-operations $\lbrace \mathfrak{m}_{k}^{b}\rbrace$ can be regarded as a deformation of $\lbrace \mathfrak{m}_{k}\rbrace$ by a weak Maurer-Cartan element $b$, which is from the cohomology of $L$ itself. Similarly we can deform the $A_{\infty}$-operations by the cohomology of the ambient symplectic manifold $X$. Such a deformation is called a \textit{bulk deformation}.

Let $E_{l}H^{*}(X; \Lambda_{+})$ be the subspace of $H^{*}(X; \Lambda_{0})^{\otimes l}$ which is invariant under the action of the $l$th symmetric group. Then in \cite{FOOO} a sequence of operators $\lbrace\mathfrak{q}_{l, k; \beta}\rbrace_{l\geq 0; k\geq 0}$ is constructed
$$
\mathfrak{q}_{l, k; \beta}: E_{l}H^{*}(X; \Lambda_{+})\otimes H^{*}(L; \Lambda_{0})^{\otimes k}\rightarrow H^{*}(L; \Lambda_{0}).
$$
Geometrically those operators count holomorphic disks with both boundary marked points attached on given cocycles in $L$ and interior marked points attached on given cocycles in $X$. And we define the operator $\mathfrak{q}_{l, k}:=\sum_{\beta}T^{\omega(\beta)}\cdot\mathfrak{q}_{l, k; \beta}$. Again, here we are using the operators constructed on the canonical model. When $l=0$ we have that
\begin{equation}\label{q0}
\mathfrak{q}_{0, k}(1; x_{1}, \cdots, x_{k})=\mathfrak{m}_{k}(x_{1}, \cdots, x_{k})
\end{equation}
where $1\in H^{*}(X; \Lambda_{0})$ is the unit.

Now for any $\mathfrak{b}\in H^{*}(X; \Lambda_{+})$ and $x_{1}, \cdots, x_{k}\in H^{*}(L; \Lambda_{0})$ we define
\begin{equation}\label{bulk}
\mathfrak{m}^{\mathfrak{b}}_{k}(x_{1}, \cdots, x_{k})=\sum_{l=0}^{\infty}\mathfrak{q}_{l, k}(\mathfrak{b}^{\otimes l}; x_{1}, \cdots, x_{k}).
\end{equation}
Then $\lbrace\mathfrak{m}^{\mathfrak{b}}_{k}\rbrace$ also defines a filtered $A_{\infty}$-algebra structure on $H^{*}(L; \Lambda_{0})$. For a fixed $\mathfrak{b}$, an element $b\in H^{1}(L; \Lambda_{+})$ is called a weak bounding cochain (with respect to $\mathfrak{b}$) if it satisfies the $A_{\infty}$-Maurer-Cartan equation given by the deformed operators $\lbrace\mathfrak{m}^{\mathfrak{b}}_{k}\rbrace$. And we write $\mathcal{M}_{weak}(L; \mathfrak{b})$ as the set of weak bounding cochains of $L$ with respect to $\mathfrak{b}$.

To do concrete computations there are two \textit{divisor axioms} for the operators $\mathfrak{m}_{k}$ and $\mathfrak{q}_{l,k}$. For $\mathfrak{b}\in H^{2}(X; \Lambda_{+}), b\in H^{1}(L; \Lambda_{+})$ and $\mu(\beta)=2$ we have that
\begin{equation}\label{divisor axiom}
\begin{aligned}
& \mathfrak{m}_{k;\beta}(b^{\otimes k})= \dfrac{(b(\partial\beta))^{k}}{k!}\cdot \mathfrak{m}_{0;\beta}(1);\\
& \mathfrak{q}_{l,k;\beta}(\mathfrak{b}^{\otimes l}; x_{1},\cdots, x_{k})=\dfrac{(\mathfrak{b}\cdot \beta)^{l}}{l!}\cdot \mathfrak{q}_{0,k;\beta}(1; x_{1},\cdots, x_{k}).
\end{aligned}
\end{equation}
where $b(\partial\beta)$ and $\mathfrak{b}\cdot \beta$ are pairings between homology and cohomology classes (we assume that $\mathfrak{b}\cap L=\emptyset$). These are first studied in \cite{F} and we refer to Section 7 in \cite{FOOO2} for a proof.

Next we put those two deformations together, one from the Lagrangian itself and the other from the ambient space. Define an operator
\begin{equation}\label{boundary operator}
d_{\mathfrak{b}}^{b}(x)=\sum_{k_{0}, k_{1}}\mathfrak{m}^{\mathfrak{b}}_{k_{0}+k_{1}+1}(b^{\otimes k_{0}}, x, b^{\otimes k_{1}}) : H^{*}(L; \Lambda_{0})\rightarrow H^{*}(L; \Lambda_{0}).
\end{equation}
When $b\in \mathcal{M}_{weak}(L; \mathfrak{b})$ we have that
\begin{equation}
d_{\mathfrak{b}}^{b}\circ d_{\mathfrak{b}}^{b}(x)=0
\end{equation}
and the resulting cohomology
$$
HF(L, \mathfrak{b}, b; \Lambda_{0})
$$
is called the deformed Floer cohomology of $L$ by the bulk deformation $\mathfrak{b}$. If we expand the summation of $d_{\mathfrak{b}}^{b}$ we will find that the new differential $d_{\mathfrak{b}}^{b}$ contains the differential $\mathfrak{m}_{1}^{b}$.
\begin{equation}
\begin{aligned}
d_{\mathfrak{b}}^{b}(x)& =\sum_{k_{0}, k_{1}}\mathfrak{m}^{\mathfrak{b}}_{k_{0}+k_{1}+1}(b^{\otimes k_{0}}, x, b^{\otimes k_{1}})\\
&= \sum_{l, k_{0}, k_{1}}\mathfrak{q}_{l, k_{0}+k_{1}+1}(\mathfrak{b}^{\otimes l}; b^{\otimes k_{0}}, x, b^{\otimes k_{1}})\\
&= \mathfrak{m}_{1}^{b}(x)+ \sum_{l\geq 1, k_{0}, k_{1}}\mathfrak{q}_{l, k_{0}+k_{1}+1}(\mathfrak{b}^{\otimes l}; b^{\otimes k_{0}}, x, b^{\otimes k_{1}}).
\end{aligned}
\end{equation}
Hence the differential $d_{\mathfrak{b}}^{b}$ is a sum of the ``zeroth order'' term $\mathfrak{m}_{1}^{b}$ and ``higher order'' deformations which count holomorphic disks with interior marked points attached on given cocycles in $X$.

Similarly we define a \textit{bulk-deformed potential function}
$$
\mathfrak{PO}^{\mathfrak{b}}: \mathcal{M}_{weak}(L; \mathfrak{b})\rightarrow \Lambda_{+}
$$
by setting
$$
\sum _{k=0}^{\infty} \mathfrak{m}^{\mathfrak{b}}_{k}(b, \cdots, b)=\mathfrak{PO}^{\mathfrak{b}}(b)\cdot PD([L]).
$$
From the above discussion we have that $\mathfrak{PO}^{\mathfrak{b}=0}(b)=\mathfrak{PO}(b)$.

Since the operators $\mathfrak{m}_{1}^{b}$ and $d_{\mathfrak{b}}^{b}$ are defined by infinite sums, we need to assume that $b, \mathfrak{b}$ are with $\Lambda_{+}$ coefficients for the convergence issue. But they can be extended with coefficients in $\Lambda_{0}$ by using similar idea in (2.2). Hence we obtain a cohomology theory totally over $\Lambda_{0}$. So we omit the coefficients in $\mathcal{M}_{weak}(L)$ and $\mathcal{M}_{weak}(L; \mathfrak{b})$ when we do not emphasize them.

A structural result, Theorem 6.1.20 in \cite{FOOO}, tells us a decomposition formula for the deformed Floer cohomology
\begin{equation}
HF(L, \mathfrak{b}, b; \Lambda_{0})\cong (\Lambda_{0})^{a}\oplus (\bigoplus_{i=1}^{l} \dfrac{\Lambda_{0}}{T^{\lambda_{i}}\Lambda_{0}})
\end{equation}
where $a\in\mathbb{Z}_{\geq 0}$ is called the Betti number and $\lambda_{i}\in \mathbb{R}_{+}$ are called the torsion exponents of the deformed Floer cohomology. It is proved that only the free part of the deformed Floer cohomology is an invariant under Hamiltonian diffeomorphisms, see Theorem J in \cite{FOOO}. Hence it suffices to show that $a>0$ if we want to prove some $L$ is nondisplaceable. When $a=0$, the torsion exponents are closely related to the displacement energy of $L$, which we will discuss in detail in Section 5.

\section{Computations of classical Floer cohomology}
In this section we carry out some computations of classical Floer cohomology, which are free of bounding cochains and bulk-deformations. Condition \ref{condition} will not be used.

\subsection{Symplectic cut and sum construction}
First we summarize the symplectic cut and sum construction to analyze holomorphic disks on Lagrangian submanifolds near a Lagrangian sphere. The whole construction is fully described in section 2 of \cite{FT} and section 3 of \cite{FTZ}.

Let
$$
Q_{n}=\lbrace [z_{0}, \cdots, z_{n+1}]\in \mathbb{C}P^{n+1}\mid z_{0}^{2}=\sum_{j=1}^{n+1}z_{j}^{2}\rbrace
$$
be the complex quadric hypersurface and
$$
D_{n}=\lbrace [z_{0}, \cdots, z_{n+1}]\in Q_{n}\mid z_{0}=0\rbrace\cong Q_{n-1}
$$
be the divisor at infinity. Then the real part $Q_{n, \mathbb{R}}=Q_{n}\cap \mathbb{R}P^{n+1}$ is a Lagrangian $n$-sphere in $(Q_{n}, \omega_{FS})$ and $Q_{n}-D_{n}$ is a Weinstein neighborhood of $Q_{n, \mathbb{R}}$. Another perspective is that there is a Hamiltonian $S^{1}$-action on $T^{*}S^{n}$ such that the sphere bundles of it are regular level sets. If we collapse the circles on a sphere bundle with radius $r$ then $(D_{r}T^{*}S^{n}, \partial D_{r}T^{*}S^{n})$ goes to $(Q_{n},D_{n})$ with a rescaled Fubini-Study symplectic form.

Now we recall the materials from Proposition 2.1 in \cite{FT} and Proposition 3.1 in \cite{FTZ}. Let $\Delta\subset \mathbb{C}$ be a disk centered at the origin. A \textit{symplectic fibration} is a pair $\pi : (\mathcal{X}, \omega_{\mathcal{X}}) \rightarrow \Delta$ such that $\pi$ is surjective, the total space $(\mathcal{X}, \omega_{\mathcal{X}})$ is a smooth symplectic manifold, the fiber $X_{z}$ is a smooth symplectic submanifold when $z\neq 0$, and $X_{0}$ is a union of symplectic submanifolds of $(\mathcal{X}, \omega_{\mathcal{X}})$ meeting along a smooth symplectic submanifold $D$. We call $D$ the singular locus of $X_{0}$. A \textit{Lagrangian subfibration} of $\pi : (\mathcal{X}, \omega_{\mathcal{X}}) \rightarrow \Delta$ is a submanifold $\mathcal{S}\subset \mathcal{X}$ disjoint from $D$, such that $\pi(\mathcal{S})=\Delta$ and $S_{z}:=\mathcal{S}\cap X_{z}$ is a Lagrangian submanifold for every $z$.

\begin{proposition}
Let $(X, \omega, S)$ be a symplectic $2n$-manifold with a Lagrangian $n$-sphere $S$. There exists a symplectic fibration $\pi: (\mathcal{X}, \omega_{\mathcal{X}}) \rightarrow \Delta$ with a Lagrangian subfibration $\mathcal{S}$. Let $X_{z}$ be the fiber at $z\in \Delta$ then we have
\begin{enumerate}
\item $X_{0}=X_{-}\cup_{D}X_{+}$ where both $X_{\pm}$ are closed smooth symplectic manifolds and $D=X_{-}\cap X_{+}$ is a common symplectic hypersurface;
\item when $z\neq 0$ the pair $(X_{z}, \left.\omega_{\mathcal{X}}\right|_{X_{z}}, \mathcal{S}_{z})$ is symplectomorphic to $(X, \omega, S)$;
\item when $z=0$ then $\mathcal{S}_{0}$ is in $X_{-}$ and the 4-tuple $(X_{-}, \left.\omega_{\mathcal{X}}\right|_{X_{-}}, D, \mathcal{S}_{0})$ is symplectomorphic to $(Q_{n}, \epsilon\omega_{FS}, D_{n}, Q_{n, \mathbb{R}})$ for some scaling parameter $\epsilon >0$, which depends on a choice of the size of a Weinstein neighborhood of $S$ in $X$.
\end{enumerate}
\end{proposition}

Next we specify the almost complex structures we will use on this fibration. An almost complex structure $J$ on the fibration $\pi: (\mathcal{X}, \omega_{\mathcal{X}}) \rightarrow \Delta$ is said to be \textit{admissible} if
\begin{enumerate}
\item it is compatible with $\omega_{\mathcal{X}}$ and preserves ker $d\pi$;
\item it restricts to an almost complex structure on the singular locus $D$ of $X_{0}$ and satisfies that
$$
N_{J}(u, v)\in T_{x}D \quad \forall u\in T_{x}D, v\in T_{x}X_{0}, x\in  D
$$
where $N_{J}$ is the Nijenhuis tensor of $J$.
\end{enumerate}
We denote the set of all admissible almost complex structures on $\mathcal{X}$ by $\mathcal{J}_{\mathcal{X}}$ and the subset of $l$-differentiable elements by $\mathcal{J}^{l}_{\mathcal{X}}$. Both spaces $\mathcal{J}_{\mathcal{X}}$ and $\mathcal{J}^{l}_{\mathcal{X}}$ are non-empty and path-connected.

With respect to an admissible almost complex structure we can compare the first Chern numbers and Maslov indices between $(X, \omega)$ and $(X_{\pm}, \left.\omega_{\mathcal{X}}\right|_{X_{\pm}})$. Let $\beta \in H_{2}(X_{-}, S; \mathbb{Z})$ and $A\in H_{2}(X_{+}; \mathbb{Z})$ such that $\beta\cdot_{X_{-}}D=A\cdot_{X_{+}}D$ then we can deform the connected sum of $\beta$ and $A$ to be a homology class $\beta+A\in H_{2}(X, S; \mathbb{Z})$ in the smooth fiber. Note that $\partial\beta=0$ so the pairings $\langle c_{1}(TX_{-}), \beta\rangle$ and $\langle c_{1}(TX), \beta+A\rangle$ are well-defined and we have the following relation.

\begin{proposition}
With the above notation,
\begin{equation}\label{index}
\begin{split}
\langle c_{1}(TX), \beta+A\rangle&= \langle c_{1}(TX_{-}), \beta\rangle+ \langle c_{1}(TX_{+}), A\rangle- 2A\cdot_{X_{+}}D\\
&= \langle c_{1}(TX_{+}), A\rangle +(n-2)A\cdot_{X_{+}}D
\end{split}
\end{equation}
\end{proposition}

We remark that the first line of (\ref{index}) is a general formula and the second line uses that $X_{-}$ is the quadric hypersurface.

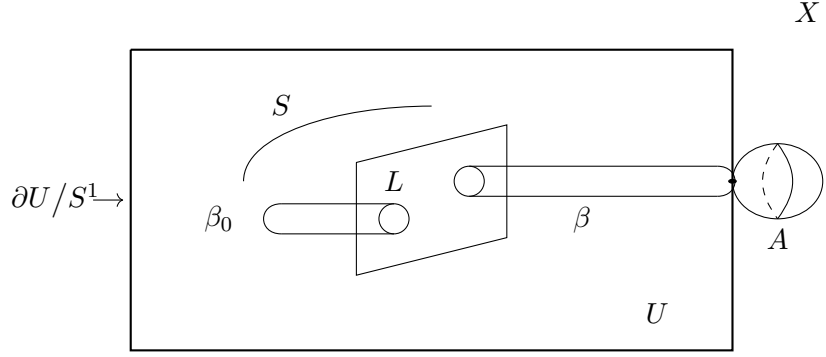
\begin{figure}
  \begin{tikzpicture}[xscale=1, yscale=0.5]
  \draw (-1,1)--(1,2)--(1,-1)--(-1,-2)--(-1,1);
  \draw [thick] (-4,4)--(4,4)--(4,-4)--(-4,-4)--(-4,4);
  \draw (-2.5,0.5) to [in=180,out=90] (0,2.5);
  \node [above] at (-2,2) {$S$};
  \draw (0.5,0.5) ellipse (0.2 and 0.4);
  \draw (0.5,0.9)--(3.8,0.9);
  \draw (0.5,0.1)--(3.8,0.1);
  \draw (3.8,0.9) to [out=0,in=0] (3.8,0.1);
  \draw (4.6,0.5) ellipse (0.6 and 1);
  \draw [fill] (4,0.5) circle [radius=0.05];
  \draw (4.6,1.5) to [out=290,in=70] (4.6,-0.5);
  \draw [dashed] (4.6,1.5) to [out=250,in=110] (4.6,-0.5);
  \draw (-0.5,-0.5) ellipse (0.2 and 0.4);
  \draw (-0.5,-0.9)--(-2,-0.9);
  \draw (-0.5,-0.1)--(-2,-0.1);
  \draw (-2,-0.9) to [out=180,in=180] (-2,-0.1);
  \node at (3,-3) {$U$};
  \node at (-5,0) {$\partial U\big/S^{1}$};
  \draw [->] (-4.5,0)--(-4.1,0);
  \node at (-0.5,0.5) {$L$};
  \node [left] at (-2.5,-0.5) {$\beta_{0}$};
  \node [below] at (2,0.1) {$\beta$};
  \node [below] at (4.6,-0.5) {$A$};
  \node at (5,5) {$X$};
  \end{tikzpicture}
  \caption{Degeneration of a holomorphic disk.}
\end{figure}

Next we use the symplectic cut and sum construction to analyze holomorphic disks with boundaries on a Lagrangian submanifold near a Lagrangian sphere, with respect to some almost complex structure. In \cite{OU} Oakley-Usher constructed many families of monotone nondisplaceable Lagrangian submanifolds in $T^{*}S^{n}$.

\begin{theorem}(Oakley-Usher)
There exist continuum families of monotone Lagrangian submanifolds
$$
L^{\lambda}_{k, m}\cong (S^{1}\times S^{k}\times S^{m})\big/\mathbb{Z}_{2}, \quad k,m \in\mathbb{Z}_{+}, k\leq m, k+m=n-1, \lambda \in(0, +\infty) \subset\mathbb{R}
$$
with non-zero Floer cohomology in $T^{*}S^{n}$.
\end{theorem}

Then for a closed symplectic manifold $X$ containing a Lagrangian sphere $S$, a sub-family of these Lagrangian submanifolds sits inside a Weinstein neighborhood $U$ of $S$. In the following we will show that when the symplectic manifold is Calabi-Yau or negatively monotone, any Lagrangian submanifolds in $U$ does not bound $J$-holomorphic disks which are not totally contained in $U$, for some $J$. Hence we get continuum families of nondisplaceable Lagrangian submanifolds stated in Theorem \ref{cy}.

\begin{theorem}(Theorem \ref{cy})
For any integer $n\geq 3$, let $(X^{2n}, S^{n}, \omega)$ be a Calabi-Yau or negatively monotone symplectic manifold with a Lagrangian sphere. Then there are continuum families of Lagrangian submanifolds
$$
L^{\lambda}_{k, m}\cong (S^{1}\times S^{k}\times S^{m})\big/\mathbb{Z}_{2}, \quad k,m \in\mathbb{Z}_{+}, k\leq m, k+m=n-1, \lambda\in (0, \lambda_{0}] \subset\mathbb{R}
$$
near the Lagrangian sphere $S$ and are nondisplaceable in $X$.
\end{theorem}
\begin{proof}
The proof is based on a dimension-counting argument, similar to the proof of Theorem 1.1 in \cite{FT}. Let $(\mathcal{X}, \omega_{\mathcal{X}})$ be the fibration constructed above and $\mathcal{J}_{\mathcal{X}}$ be the set of all admissible almost complex structures on $\mathcal{X}$. For an admissible almost complex structure $J\in \mathcal{J}_{\mathcal{X}}$, let $J_{z}$ be the restriction of $J$ on $X_{z}$. Let $L^{\lambda}_{z}$ be the image of $L^{\lambda}_{k, m}$ in $X_{z}$. For small $\lambda$, we have that $L^{\lambda}_{0}\subset X_{-}$.

We will study the limit of holomorphic disks from smooth fibers to the central fiber $X_{0}=X_{-}\cup_{D}X_{+}$. In particular, we will show that for $z$ sufficiently close to $0$, the Lagrangian submanifold $L^{\lambda}_{z}\subset X_{z}$ does not bound any $J_{z}$-holomorphic disk which are not inside the Weinstein neighborhood of $S$, if $J$ is generic enough.

Let $\mathcal{M}^{reg}(X_{+}, A, J_{+})$ be the moduli space of somewhere injective $J_{+}$-holomorphic curves of class $A\in H_{2}(X_{+}; \mathbb{Z})$, where $J_{+}= \left. J\right|_{X_{+}}$. Then classic result shows that there is a dense subset $\mathcal{J}^{reg}_{\mathcal{X}}\subset \mathcal{J}_{\mathcal{X}}$ such that $\mathcal{M}^{reg}(X_{+}, A, J_{+})$ is a smooth manifold of dimension
$$
\dim _{\mathbb{R}}\mathcal{M}^{reg}(X_{+}, A, J_{+})= 2n-6+ 2\langle c_{1}(TX_{+}), A\rangle
$$
for $J\in \mathcal{J}^{reg}_{\mathcal{X}}$ and $J_{+}= \left. J\right|_{X_{+}}$. In particular, if $0> 2n-6+ 2\langle c_{1}(TX_{+}), A\rangle$ then the moduli space $\mathcal{M}^{reg}(X_{+}, A, J_{+})$ is empty.

Next let $z_{i}\in \Delta$ be a sequence converging to 0 and $J_{i}$ be the restriction of $J$ on $X_{z_{i}}$. Consider a sequence of $J_{i}$-holomorphic disks of class $\beta$ in $X_{z_{i}}$, with boundary on $L^{\lambda}_{z_{i}}$. Then by Gromov compactness we get a nodal disk $u$ in $X_{0}=X_{-}\cup_{D}X_{+}$ with boundary on $L^{\lambda}_{0}$. In general this nodal curve at the Gromov limit can have several components and some of them might lie inside $D$ or is a multiple cover of a somewhere injective curve in $X_{+}$.

Suppose that $u$ has a component in $X_{+}$ which are not contained in $D$. Let $u'$ be its underlying somewhere injective curve, representing a class $A \in H_{2}(X_{+}; \mathbb{Z})$. Note that $J$ is admissible therefore the image of $u'$ intersects with $D$ in a finite set with positive multiplicities. That is, $A\cdot_{X_{+}}D= s>0$. Then we choose $B\in H_{2}(X_{-}; \mathbb{Z})$ such that $B\cdot_{X_{-}}D= s$. The class $A+B$ can be deformed into a class $A\#B \in H_{2}(X_{z}; \mathbb{Z})$ in the smooth fiber $X_{z}$. By (\ref{index}) we have that
\begin{equation}
\begin{split}
\langle c_{1}(TX_{z}), A\#B\rangle&= \langle c_{1}(TX_{-}), B\rangle+ \langle c_{1}(TX_{+}), A\rangle- 2A\cdot_{X_{+}}D\\
&= \langle c_{1}(TX_{+}), A\rangle +(n-2)s.
\end{split}
\end{equation}
Note that $X_{-}$ is a monotone symplectic manifold and $B\cdot_{X_{-}}D= s>0$, the symplectic area of $B$ is positive. Hence the class $A\#B$ has positive area in $X_{z}$. Since we assume that $X=X_{z}$ is Calabi-Yau or negatively monotone, we have that $0\geq \langle c_{1}(TX_{z}), A\#B\rangle$. Combining this with above equality and $s\geq 1$ we get
\begin{equation}
\begin{split}
2\langle c_{1}(TX_{+}), A\rangle +2n-6 &\leq 2(2-n)s+ 2n- 6\\
&\leq 2(2-n)+ 2n -6 =4-6 =-2 <0.
\end{split}
\end{equation}

Therefore when $J\in \mathcal{J}^{reg}_{\mathcal{X}}$ the moduli space $\mathcal{M}^{reg}(X_{+}, A, J_{+})$ is empty, a contradiction to that $u'$ is a $J_{+}$-holomorphic curve in $X_{+}$. Actually we can make the moduli spaces $\mathcal{M}^{reg}(X_{+}, A, J_{+})$ to be empty for all classes $A$ with a uniform area bound. Then by Gromov compactness $L^{\lambda}_{z}\subset X_{z}$ does not bound any $J_{z}$-holomorphic disk which are not inside the Weinstein neighborhood $U$ of $S$ and with energy less than the area of $A$, when $z$ is small enough. Otherwise we get the non-trivial curve $u'$ in $X_{+}$.

In conclusion, for each positive number $E$, there is an admissible almost complex structure $J^{E}$ such that $L^{\lambda}_{z}\subset X_{z}$ does not bound any non-constant $J^{E}_{z}$-holomorphic disk which are not inside $U$ and with area smaller than $E$, when $z$ is small. In order to achieve transversality, we only need to perturb $J_{+}$ or $J_{z}$ outside $U$, while keeping $J_{z}$ in $U$ as some fixed almost complex structure. For example, we can take the almost complex structure used by Oakley-Usher to compute the Floer cohomology in $U$. Then we know that $L^{\lambda}_{z}$ only bounds non-constant $J_{z}$-holomorphic disk in $U$, which gives a non-trivial Floer cohomology modulo energy $E$. Since the above argument works for any $E>0$, we obtain that $L^{\lambda}_{z}$ is nondisplaceable in $X_{z}=X$.
\end{proof}

Oakley and Usher proved that if we compactify the cotangent bundle to be the quadric then $L_{0, m}$ is displaceable in $Q_{m+2}$ for $m\geq 2$. This matches the discussion above that when the ambient space is monotone there will be holomorphic disks coming from outside, which may break the Floer cohomology. The major task of following sections will be studying possible deformations of Floer cohomology to deal with those outside contributions.

One may also use the neck-stretching technique in symplectic field theory to prove above results. For example when $\dim S=3$, on the contact hypersurface $\partial U$, there is a contact form such that the minimal Conley-Zehnder index of Reeb orbits is two. Starting with a holomorphic disk with boundary on a Lagrangian $L^{\lambda}_{k, m}$, we stretch along $\partial U$ and get a holomorphic building. The $X_{+}$-part of the holomorphic building will be a holomorphic curve with several negative punctures, having a non-positive Chern number. The dimension of the moduli space of such curves (or its underlying somewhere injective curves) in $X_{+}$ is negative.

\subsection{Welschinger invariants and the pearl complex}
Now we change the gear to compare the open Gromov-Witten invariants defined by Welschinger \cite{W} and the Floer differential in the pearl complex. The former count of holomorphic disks is weighted by linking numbers, and the later is weighted by number of Morse flow lines. An observation is that these two weights are related, and we will compare them up to signs. The conventions of orientations of moduli spaces on each side are in \cite{W}, and in the Appendix A of \cite{BC3} respectively. We expect that after a careful comparison of signs, the results in this subsection have some generalizations with integer coefficients.

We first review the construction in \cite{W}. For a closed oriented Lagrangian 3-manifold $L$ of a symplectic 6-manifold $X$, such that $L$ does not bound Maslov zero holomorphic disks for some $J$ and $H_{1}(L; A)$ injects into $H_{1}(X;A)$ for some commutative ring $A$, Welschinger defines enumerative invariants which count holomorphic disks with boundary on $L$. Here we will focus on a special case of a monotone Lagrangian sphere $S$ in a monotone symplectic 6-manifold $X$.

Since $S$ is simply-connected, we can define Welschinger invariants with value in any commutative ring. We will mainly use $\mathbb{Z}$ or $\mathbb{Z}_{2}$ coefficients. And the final conclusion of the relation between the Welschinger invariants and the Floer differentials will be over $\mathbb{Z}_{2}$ coefficients.

Fix an orientation and (the unique) spin structure on $S$. For a class $\beta\in \pi_{2}(X,S)$ and a generic compatible almost complex structure $J$, we write $\mathcal{P}^{\beta}_{r}(X, S; J)$ as the moduli space of simple $J$-holomorphic disks with boundary on $S$, representing the class $\beta$, and with $r$ marked points on the boundary. And we write $\mathcal{M}^{\beta}_{r}(X, S; J)$ as the quotient of $\mathcal{P}^{\beta}_{r}(X, S; J)$ by the automorphism group of the unit disk. For generic $J$, the moduli space $\mathcal{M}^{\beta}_{r}(X, S; J)$ is an oriented manifold of dimension $\mu(\beta)+r$, with an evaluation map to $S^{r}$. Then we compactify this moduli space and still write it as $\mathcal{M}^{\beta}_{r}(X, S; J)$.

Next we introduce the moduli spaces of nodal disks and multi-disks. Let $\beta_{1}, \beta_{2}\in \pi_{2}(X,S)$ be two disk classes. We denote by $\mathcal{P}^{(\beta_{1}, \beta_{2})}_{0}(X, S; J)$ the fiber product $\mathcal{P}^{\beta_{1}}_{0}(X, S; J)_{ev_{1}}\times _{ev_{-1}}\mathcal{P}^{\beta_{2}}_{0}(X, S; J)$, by using the evaluation maps at $-1$ and $1$. And we define
$$
\mathcal{P}^{(\beta_{1}, \beta_{2})}_{r}(X, S; J)= \mathcal{P}^{(\beta_{1}, \beta_{2})}_{0}(X, S; J)\times ((\partial \Delta_{0}-\lbrace\text{node}\rbrace)^{r}- \text{diag}_{\Delta_{0}}).
$$
Here $\Delta_{0}$ is the nodal disk as the closure of $(\mathbb{H}\times \lbrace 0\rbrace)\cup (\lbrace 0\rbrace\times \mathbb{H})\subset \mathbb{C}^{2}$ in $\mathbb{C}P^{2}$, and $\text{diag}_{\Delta_{0}}$ is the big diagonal. We then denote by
$$
\mathcal{M}^{(\beta_{1}, \beta_{2})}_{r}(X, S; J)= \mathcal{P}^{(\beta_{1}, \beta_{2})}_{r}(X, S; J)/ Aut(\Delta_{0}).
$$
By Gromov compactness and gluing theorems, the moduli space $\mathcal{M}^{(\beta_{1}, \beta_{2})}_{r}(X, S; J)$ canonically identifies as a component of the boundary of the (compactified) moduli space $\mathcal{M}^{\beta_{1}+\beta_{2}}_{r}(X, S; J)$. In the notation of \cite{W}, their orientations differ by $-1$, see Proposition 2 in \cite{W}.

Let $\beta_{1}, \beta_{2}, \cdots, \beta_{n}\in \pi_{2}(X,S)$ be disk classes. We define $\mathcal{P}^{\beta_{1}, \beta_{2}, \cdots, \beta_{n}}_{0}(X, S; J)$ as the direct product of $\mathcal{P}^{\beta_{i}}_{0}(X, S; J)$'s. Likewise, we define $\mathcal{M}^{\beta_{1}, \beta_{2}, \cdots, \beta_{n}}_{r}(X, S; J)$ as the quotient
$$
\mathcal{M}^{\beta_{1}, \beta_{2}, \cdots, \beta_{n}}_{r}(X, S; J)= (\mathcal{P}^{\beta_{1}, \beta_{2}, \cdots, \beta_{n}}_{0}(X, S; J)\times ((\partial \Delta \cup \cdots \cup \partial \Delta )^{r}- \text{diag}_{\partial\Delta}))/(Aut(\Delta))^{n}.
$$
Here $\Delta$ is the closed unit disk in $\mathbb{C}$ and $\text{diag}_{\partial\Delta}$ is the big diagonal of $\partial \Delta \cup \cdots \cup \partial \Delta$. For every $n\geq 2$, we denote by $\mathcal{M}^{(\beta_{1}, \beta_{2}), \cdots, \beta_{n}}_{0}(X, S; J)$ as the direct product
$$
\mathcal{M}^{(\beta_{1}, \beta_{2}), \cdots, \beta_{n}}_{0}(X, S; J)= \mathcal{M}^{(\beta_{1}, \beta_{2})}_{0}(X, S; J)\times \mathcal{M}^{\beta_{3}}_{0}(X, S; J)\times \cdots \mathcal{M}^{\beta_{n}}_{0}(X, S; J).
$$
And likewise we define $\mathcal{M}^{(\beta_{1}, \beta_{2}), \cdots, \beta_{n}}_{r}(X, S; J)$. The moduli space $\mathcal{M}^{(\beta_{1}, \beta_{2}), \cdots, \beta_{n}}_{r}(X, S; J)$ can be identified as a codimension one part of the moduli space $\mathcal{M}^{\beta_{1}, \beta_{2}, \cdots, \beta_{n}}_{r}(X, S; J)$, after Gromov compactification. See Section 2.4 in \cite{W} for the compatibilities of orientations between product, forgetful map and subspaces.

Now we denote by $\mathcal{M}^{\beta_{1},\cdots, \beta_{n}}_{r, int}(X,S;J)$ the dense open subset of $\mathcal{M}^{\beta_{1},\cdots, \beta_{n}}_{r}(X,S;J)$ made of disks whose $n$ boundary components have pairwise disjoint images in $L$. It is equipped with a boundary map which sends a holomorphic map $u$ to $u(\partial\Delta \cup \cdots \cup \partial\Delta)$. Hence we can view $u(\partial\Delta \cup \cdots \cup \partial\Delta)$ as a link in $S$ with $n$ components, oriented by the boundary orientation as the boundary of a holomorphic multi-disk.

In our setting, we are interested in disk classes $\beta\in \pi_{2}(X,S)$ with Maslov index four. Since $S$ is orientable and monotone, the class $\beta$ can only split into two classes with Maslov index two. So we only need the case where there are at most two components. Then we can define a linking weight on the moduli space $\mathcal{M}^{\beta_{1},\cdots, \beta_{n}}_{r, int}(X,S;J)$ in a simpler way. It is a locally constant function
$$
lk_{n}: \mathcal{M}^{\beta_{1},\cdots, \beta_{n}}_{r, int}(X, S;J)\rightarrow \mathbb{Z}.
$$
When $n=1$, there is only one component of $\mathcal{M}^{\beta_{1}}_{r, int}(X,S;J)$, we define $lk_{1}=1$ be the constant function. When $n=2$, for an element $u\in\mathcal{M}^{\beta_{1},\beta_{2}}_{r, int}(X, S;J)$ we define $lk_{2}(u)=lk(u(\partial \Delta \cup \partial\Delta))$, the linking number of two boundary components of $u$. Now let $\beta$ be a disk class of Maslov index four. We set
\begin{equation}\label{Wel-invariant}
[\mathcal{M}_{\beta,2}(X,S;J)]:= \sum_{n=1}^{2} \dfrac{1}{n!}\sum_{\beta_{1}+\cdots +\beta_{n}=\beta} lk_{n}[\mathcal{M}^{\beta_{1},\cdots, \beta_{n}}_{2, int}(X,S;J)]
\end{equation}
to define the two-pointed Welschinger invariants. Here $lk_{n}[\mathcal{M}^{\beta_{1},\cdots, \beta_{n}}_{2, int}(X,S;J)]$ is defined as the linking number $lk_{n}(\mathcal{M}^{\beta_{1},\cdots, \beta_{n}}_{2, int}(X,S;J))$ times the fundamental class $[\mathcal{M}^{\beta_{1},\cdots, \beta_{n}}_{2, int}(X,S;J)]$.

\begin{theorem}(Theorem 2, \cite{W})
Let $S$ be a monotone Lagrangian 3-sphere in a symplectic 6-manifold $X$. For a Maslov four disk class $\beta$, the chain
$$
ev_{*}[\mathcal{M}_{\beta,2}(X,S;J)]:= \sum_{n=1}^{2} \dfrac{1}{n!}\sum_{\beta_{1}+\cdots +\beta_{n}=\beta} ev_{*}(lk_{n}[\mathcal{M}^{\beta_{1},\cdots, \beta_{n}}_{2, int}(X,S;J)])
$$
is a cycle whose homology class in $H_{6}(S\times S; \mathbb{Z})$ does not depend on the generic choice of $J$.
\end{theorem}

The idea of the proof is the following. First we have a main component $\mathcal{M}^{\beta}_{2, int}(X,S;J)$ of which the codimension one boundaries are $\pm\mathcal{M}^{(\beta_{1}, \beta_{2})}_{2, int}(X,S;J)$ with $\beta= \beta_{1}+ \beta_{2}$. Then for each pair of $\beta_{1}, \beta_{2}$ we have a moduli space of multi-disks $\mathcal{M}^{\beta_{1}, \beta_{2}}_{2, int}(X,S;J)$ of which the codimension one boundaries are also $\pm\mathcal{M}^{(\beta_{1}, \beta_{2})}_{2, int}(X,S;J)$, weighted by linking numbers. After unioning all these moduli spaces together, all the codimension one boundary cancel by a careful check of signs and linking numbers. Hence we get a moduli space without codimension one boundary and the homology class given by evaluation maps does not depend on various choices. We also remark that the factor $\frac{1}{n!}$ is used to compensate the permutations of $\beta_{1}, \cdots, \beta_{n}$. For example, both $\mathcal{M}^{\beta_{1}, \beta_{2}}_{2, int}(X,S;J)$ and $\mathcal{M}^{\beta_{2}, \beta_{1}}_{2, int}(X,S;J)$ appear in the formula, but as a boundary of the main component, the moduli space $\mathcal{M}^{(\beta_{1}, \beta_{2})}_{2, int}(X,S;J)$ only appears once.

Then the two-pointed Welschinger invariant of class $\beta$ is defined as
$$
n_{2, \beta}^{W}:=\langle ev_{*}[\mathcal{M}_{\beta,2}(X,S;J)], PD[pt]\cup PD[pt]\rangle \in \mathbb{Z},
$$
which is independent of a generic choice of $J$.

Next we review the pearl complex to compute the Floer cohomology, and compare its differential with the two-pointed Welschinger invariant. We refer to \cite{BC1,BC2,BC3} for more details about the pearl complex.

Let $L$ be a monotone Lagrangian submanifold of a symplectic manifold $X$, which is closed or convex at infinity. Fix a triple $(f, \rho, J)$ where $f: L\rightarrow \mathbb{R}$ is a Morse function, $\rho$ is a Riemannian metric on $L$ and $J$ is a compatible almost complex structure. Define a complex generated by the critical points of $f$:
$$
\mathcal{C}(f, \rho, J)= F\langle \text{Crit}(f)\rangle\otimes F[T, T^{-1}].
$$
Here $F$ is a commutative ring and $F[T, T^{-1}]$ is the ring of formal Laurent polynomials. When $L$ is spin, the ring $F$ can be taken as $\mathbb{Z}$ otherwise we need to take $F=\mathbb{Z}_{2}$. In \cite{BC1, BC2} the ground ring is assumed to be $\mathbb{Z}_{2}$ and in \cite{BC3} it is extended to the case of $\mathbb{Z}$. See Appendix A in \cite{BC3} for the orientation data for all related moduli spaces.

Let $\Phi_{t}$ be the flow of the negative gradient line of $(f, \rho)$, for $0\leq t\leq\infty$ (or $-\infty\leq t\leq\infty$ if $X$ is closed). Given two points $x, y\in L$ and a non-zero class $\beta\in\pi_{2}(X, L)$, consider the space of all sequences $(u_{1},\cdots, u_{l})$ of every possible length $l\geq 1$, where:
\begin{enumerate}
\item $u_{l}$ is a non-constant $J$-holomorphic disk with boundary on $L$;
\item there exists $-\infty\leq t'<0$ such that $\Phi_{t'}(u_{1}(-1))=x$;
\item for every $1\leq i\leq l-1$ there exists $0<t_{i}<\infty$ such that $\Phi_{t_{i}}(u_{i}(1))=u_{i+1}(-1)$;
\item there exists $0< t''\leq \infty$ such that $\Phi_{t''}(u_{l}(1))=y$;
\item $\sum_{1\leq i\leq l} [u_{i}]= \beta$.
\end{enumerate}
We view two sequences $(u_{1},\cdots, u_{l})$ and $(u'_{1},\cdots, u'_{l'})$ as equivalent if $l=l'$ and for every $1\leq i\leq l$ there exists $\delta_{i}\in Aut(\Delta)$ with $\delta_{i}(-1)=-1, \delta_{i}(1)=1$ and such that $u'_{i}=u_{i}\circ \delta_{i}$. The space of equivalence classes is denoted by $\mathcal{P}_{prl}(x, y; \beta; f, \rho, J)$. Elements of this space will be called pearly trajectories connecting $x$ to $y$.

We will be interested in the case when both $x$ and $y$ are critical points of $f$. (In particular, $t'=-\infty$ and $t''=+\infty$.)  The central theorem (for example, Theorem 2.1 in \cite{BC2}) states that for a generic triple $(f,\rho,J)$, the space $\mathcal{P}_{prl}(x, y; \beta; f, \rho, J)$ behaves like a manifold of dimension $\abs{x}-\abs{y}+\mu(\beta)-1$. Then we define a differential
$$
d:\mathcal{C}(f, \rho, J)\rightarrow \mathcal{C}(f, \rho, J), \quad d(x)=\sum_{y, \beta} \sharp \mathcal{P}_{prl}(x, y; \beta; f, \rho, J)y\cdot T^{\mu(\beta)}
$$
where the sum is over all $y, \beta$ with $\abs{x}-\abs{y}+\mu(\beta)-1=0$. The rest of the central theorem says that $d$ gives us a homology theory and the homology is independent of the triple $(f,\rho,J)$. Following the notation in \cite{BC1,BC2,BC3}, we call this homology the quantum homology $QH(L)$ of $L$. One property of $QH(L)$ is that it is isomorphic to the self-Floer homology $HF(L, L)$, hence if $L$ is displaceable then $QH(L)$ is zero.

Now we focus on our case where $L$ is a monotone Lagrangian 3-sphere $S$. We choose the height function $f:S\rightarrow \mathbb{R}$ on the unit sphere $S^{3}\subset \mathbb{R}^{4}$ and a Riemannian metric $\rho$ such that $(f,\rho)$ is Morse-Smale. Then $f$ has exactly one critical point $p$ (minimum) of index zero and one critical point $q$ (maximum) of index three. And the differentials between $p$ and $q$ determine $QH(S)$. Note that by degree reasons, the only possible non-trivial maps are
$$
d_{\beta}: p\mapsto \sharp \mathcal{P}_{prl}(p, q; \beta; f, \rho, J)q\cdot T^{4}
$$
where $\beta$ is a disk class with Maslov index four. And the differential $d=\sum_{\mu(\beta) =4}d_{\beta}$.

The key point of this subsection is the observation that $\sharp \mathcal{P}_{prl}(p, q; \beta; f, \rho, J)$ equals $n^{W}_{2,\beta}$, up to sign. The proof is based on the following lemma.

\begin{lemma}
Let $f:S\rightarrow \mathbb{R}$ be a perfect Morse function on a 3-sphere. Let $K_{1}, K_{2}$ be two oriented disjoint knots in $S$. Then the linking number $lk(K_{1}, K_{2})$ equals the signed count of negative gradient flow lines of $f$ starting from one point on $K_{1}$, ending at one point on $K_{2}$, up to sign.
\end{lemma}
\begin{proof}
Let $p$ be the minimal point of $f$, consider the image of $K_{1}$ under the flow $\Phi$. That is, define
$$
C:= \bigcup_{x\in K_{1}} \lbrace y\in L\mid \exists t\in [0, \infty),\quad \Phi_{t}(x)=y\rbrace \cup \lbrace p \rbrace.
$$
We give an orientation on $C$ which is compatible with the orientation of $K_{1}$ as its boundary. Then there is a one-to-one correspondence between intersection points of $K_{2}$ and $C$ and flow lines starting from one point on $K_{1}$, ending at one point on $K_{2}$. (We perturb the metric $\rho$ a bit such that they intersect transversally.) Moreover, this intersection number between $K_{2}$ and $C$ equals the linking number $lk(K_{1}, K_{2})$. Hence we complete the proof.

Note that we assume that $f$ is perfect, for any $x\in K_{1}$ there is a unique smooth flow line connecting $x$ and $p$. For a general Morse function, there may be broken flow lines going back to other critical points. We suggest \cite{As} for general discussions.
\end{proof}

We remark that the signs of intersection points are given by the orientations of $K_{1}, K_{2}$ and $C$. The signs of flow lines depend on conventions in Morse theory. In \cite{As}, it is shown that there is a convention such that those two signs match. For our purpose, we need to compare the convention in \cite{As} and that in the Appendix of \cite{BC3}, to get a genuine equality between Welschinger invariants and Floer differentials. We will not do it here.

\begin{theorem}(Theorem \ref{wel})
In the above notations, given a disk class $\beta\in \pi_{2}(X, S)$ with Maslov index four, we have an equality
$$
\sharp \mathcal{P}_{prl}(p, q; \beta; f, \rho, J)= \pm n^{W}_{2,\beta}.
$$
\end{theorem}
\begin{proof}
Fix the height function $f$ on $S$ and generic choices of $(\rho, J)$ to define the pearl complex. Let $p$ ($q$ respectively) be the minimal (maximal respectively) point of $f$. Given a disk class $\beta$ with Maslov index four, let $\mathcal{M}_{\beta,2}(X,S;J)$ be the moduli space in (\ref{Wel-invariant}) and let $\mathcal{M}_{\beta,2}(X,S;(p,q);J)$ be the moduli space of elements such that two marked points go to $p$ and $q$ respectively. Then the two-pointed Welschinger invariant $n^{W}_{\beta}$ is the number of elements in $\mathcal{M}_{\beta,2}(X,S; (p,q);J)$.

We will construct a one-to-one correspondence between the moduli space of pearly trajectories connecting $p$ to $q$ and the moduli space $\mathcal{M}_{\beta,2}(X,S; (p,q);J)$. Pick an element $u$ in $\mathcal{M}_{\beta,2}(X,S; (p,q);J)$. First, if the underlying disk of $u$ is a single disk $u_{1}$. After reparametrization we assume that $u_{1}(-1)=p$ and $u_{1}(1)=q$. So this configuration is counted once in the space of pearly trajectories. On the other hand, a single disk has self-linking number one by definition hence it contributes once to $n^{W}_{2, \beta}$. Next, if the underlying disk of $u$ is a multi-disk $u_{1}, u_{2}$. (It has at most two components since $S$ is monotone.) Note that if two marked points are both on one component, then we have a Maslov index two disk with two-pointed constraints, which does not happen generically. So we assume that $u_{1}(-1)=p$ and $u_{2}(1)=q$. Then this configuration is weighted by the number of  flow lines from the boundary of $u_{1}$ to the boundary of $u_{2}$, which is the same as the linking number by the above lemma. Hence the multi-disks are counted by the same number in both moduli spaces, up to sign.
\end{proof}

Then we can define an invariant
$$
n^{W}_{2}:=\sum_{\mu(\beta)=4} n^{W}_{2, \beta} \in \mathbb{Z}
$$
to determine the Floer cohomology $HF(S; \Lambda(\mathbb{Z}_{2}))$. By above arguments, we have that
$$
n^{W}_{2}\equiv \sum_{\mu(\beta)=4} \sharp \mathcal{P}_{prl}(p, q; \beta; f, \rho, J)= d(p) \mod 2.
$$
So $HF(S; \Lambda(\mathbb{Z}_{2}))\neq \lbrace 0\rbrace$ if $n^{W}_{2}$ is an even number.

Therefore the Welschinger invariants give us another way to compute the Floer homology of a monotone Lagrangian 3-sphere. We sketch two  examples for this application. Consider the following Lagrangian embedding
$$
S^{3}\rightarrow \mathbb{C}^{2}\times \mathbb{C}P^{1}, \quad x \mapsto (i(x), -h(x))
$$
where $i$ is the inclusion of the unit sphere and $h$ is the Hopf map. The symplectic form on $\mathbb{C}^{2}\times \mathbb{C}P^{1}$ is the standard one times the Fubini-Study form. By the $\mathbb{C}^{2}$ factor, any compact subset of $\mathbb{C}^{2}\times \mathbb{C}P^{1}$ is displaceable. Hence we get a displaceable Lagrangian 3-sphere which is monotone with minimal Maslov number four. Note that $\pi_{2}(\mathbb{C}^{2}\times \mathbb{C}P^{1}, S)\cong \pi_{2}(\mathbb{C}P^{1})$ is one-dimensional, there is only one class $\beta$ with Maslov index four. The two-pointed Welschinger invariant $n^{W}_{2, \beta}$ is just the usual two-pointed open Gromov-Witten invariant since $\beta$ is minimal. We expect that $n^{W}_{2, \beta}=1$, which comes from the factor of $\mathbb{C}P^{1}$. On the other hand, we could compactify $\mathbb{C}^{2}\times \mathbb{C}P^{1}$ to be $\mathbb{C}P^{2}\times \mathbb{C}P^{1}$ by symplectic cutting on the boundary of a round ball $B^{4}(R)\subset \mathbb{C}^{2}$. Choosing $R$ properly, we get a monotone Lagrangian 3-sphere in $\mathbb{C}P^{2}\times \mathbb{C}P^{1}$. Then $\pi_{2}(\mathbb{C}P^{2}\times \mathbb{C}P^{1}, S)$ is two-dimensional with two generators $\beta_{1}, \beta_{2}$ from $\mathbb{C}P^{1}$ and $\mathbb{C}P^{2}$. We expect both $n^{W}_{2, \beta_{1}}$ and $n^{W}_{2, \beta_{2}}$ to be $1$. Hence $n^{W}_{2}\equiv 0 \mod 2$, which shows that the Lagrangian 3-sphere is nondisplaceable after compactifying the ambient space. Note that in this case the Lagrangian 3-sphere is known to have non-trivial Floer homology by using Lagrangian correspondence, see Section 6 in \cite{WW}.

\begin{remark} Here are two remarks about possible generalisations of the relation between Welschinger invariants and Floer homology.
\begin{enumerate}
\item As we have noticed, when the Lagrangian is spin, both Welschinger invariants and Floer homology can be defined over $\mathbb{Z}$ hence over any $\mathbb{Z}_{p}$. We expect a careful comparison of orientation data between \cite{W} and the Appendix in \cite{BC3} would give a genuine equality between invariants, not only modulo $2$.
\item For a higher dimensional monotone Lagrangian sphere, one may use the structural maps in the Biran-Oh spectral sequence which computes $QH(L)$, to define two-pointed open Gromov-Witten invariants. Those structural maps might correspond to the open Gromov-Witten invariants defined by Solomon-Tukachinsky \cite{ST} and Chen \cite{Chen}.
\end{enumerate}
\end{remark}

\section{A deformed Floer complex}
In this section we construct a Floer complex by counting holomorphic disks and cylinders, and show that these moduli spaces give an $A_{\infty}$-algebra modulo some energy. A second Floer complex, counting holomorphic strips and strips with one interior hole, will be constructed in Section 5. Although our later applications are about local tori, the construction of a deformed $A_{\infty}$-algebra does not use that the Lagrangian being a torus. Hence we state the results in a general setting.

\subsection{Monotone Lagrangian tori in $T^{*}S^{3}$}
We first review the construction of a family of monotone Lagrangian 3-tori $\lbrace L_{\lambda}\rbrace_{\lambda\in (0, +\infty)}$ in $T^{*}S^{3}$, which serve as our main examples. Let $T^{*}S^{3}$ be the cotangent bundle of $S^{3}$ with the standard symplectic structure. And let $Y_{0}=\lbrace xy-zw=0 \rbrace \subset \mathbb{C}^{4}$ be a singular hypersurface. One key fact is that there is a symplectomorphism
$$
T^{*}S^{3}-\lbrace\text{zero section}\rbrace \rightarrow Y_{0}-\lbrace(0,0,0,0)\rbrace.
$$
Hence $T^{*}S^{3}$ admits a Hamiltonian $T^{3}$-action outside the zero section:
$$
(e^{i\theta_{1}}, e^{i\theta_{2}}, e^{i\theta_{3}})\cdot (x,y,z,w)= (e^{i\theta_{1}}x, e^{-i\theta_{2}}y, e^{i\theta_{1}-i\theta_{3}}z, e^{-i\theta_{2}+i\theta_{3}}w).
$$
This Hamiltonian $T^{3}$-action gives us a singular torus fibration
$$
\pi: T^{*}S^{3}\rightarrow P\subset \mathbb{R}^{3}.
$$
Here the base $P$ is a convex polytope in $\mathbb{R}^{3}$, cut out by 4 affine functions
$$
x\geq 0; \quad -y\geq 0;\quad x-z\geq 0;\quad z-y\geq 0
$$
where $(x,y,z)$ are coordinates in $\mathbb{R}^{3}$. This polytope $P$ has four faces $P_{i}$ corresponding to the above four affine functions. A regular fiber over an interior point is a smooth Lagrangian torus and the fiber over the vertex at $(0,0,0)$ is a Lagrangian 3-sphere, the zero section. We refer to \cite{CKO} and \cite{P} for the details of the construction in view of a Gelfand-Tsetlin system. Moreover the fiber $L_{\lambda}:=\pi^{-1}(\lambda, -\lambda, 0)$ is a monotone Lagrangian torus with minimal Maslov number two, for any $\lambda\in (0, +\infty)$. This is the one-parameter family of monotone Lagrangian tori in $T^{*}S^{3}$ which are the main objects of this note. Similar to the toric case in \cite{CO} and \cite{FOOO1}, the open Gromov-Witten theory of regular fibers of certain Gelfand-Tsetlin systems was studied in \cite{NNU}, which we state below in our setting.

\begin{theorem}(Section 9, \cite{NNU})
Let $L$ be a monotone fiber of the Gelfand-Tsetlin system on $T^{*}S^{3}$, then there exists a compatible almost complex structure $J_{0}$ such that
\begin{enumerate}
\item There is a one-to-one correspondence between the $J_{0}$-holomorphic disks with Maslov index two bounded by $L$ and the faces of the Gelfand-Tsetlin polytope $P$;
\item Every Maslov index two class $\beta \in H_{2}(X, L)$ is Fredholm regular and the one-pointed open Gromov-Witten invariant $n_{\beta}=1$;
\item There is a neighborhood $W$ of the zero section $S$ such that the images of all $J_{0}$-holomorphic disks with Maslov index two bounded by $L$ are outside $W$, see Lemma 9.9 in \cite{NNU}.
\end{enumerate}
\end{theorem}

Therefore in our case each fiber bounds four holomorphic disks with Maslov index two, which span the relative homology $H_{2}(X, L)$. We remark that since $L_{\lambda}$ is monotone the one-pointed open Gromov-Witten invariant of a given class is independent of the choice of $J$. So $n_{\beta}=1$ is not only true for $J_{0}$ but also for other regular compatible almost complex structures on $T^{*}S^{3}$. Strictly speaking, Section 9 in \cite{NNU} consider the case of a closed symplectic manifold which locally looks like $T^{*}S^{3}$, and the proof therein is a local argument. Hence the above conclusions are also true for our monotone fibers, when the ambient space is $T^{*}S^{3}$.

Another description of this one-parameter family of monotone Lagrangian tori comes from a Lefschetz fibration, see \cite{CPU} where they also computed all the one-pointed Gromov-Witten invariants. We consider the smoothing
$$
Y=\lbrace xy-zw=\epsilon \rbrace \subset \mathbb{C}^{4}
$$
which is symplectomorphic to $T^{*}S^{3}$. It can be embedded into
$$
\hat{Y}=\lbrace xy=u-a, zw=u-b \rbrace \subset \mathbb{C}^{5}
$$
where $a, b$ are positive real numbers and $\epsilon=b-a>0$. The projection $\hat{Y}\rightarrow \mathbb{C}$ to the $u$-variable gives us a double conic fibration with singular fibers over $u=a$ and $u=b$. There is a fiberwise 2-torus action
$$
(\theta_{1}, \theta_{2})\cdot (x, y, z, w)=(e^{i\theta_{1}}x, e^{-i\theta_{1}}y, e^{i\theta_{2}}z, e^{-i\theta_{2}}w) \quad \forall (\theta_{1}, \theta_{2})\in T^{2}.
$$
We call an above torus orbit an equator in the fiber. Then pick a circle in the base $C_{r}=\lbrace \abs{z} =r, r>b>a \rbrace \subset \mathbb{C}$. The 3-tori formed by crossing an equator with a base circle are of our interest. In particular these tori are monotone with minimal Maslov number two. Note that if we pick a segment connecting $a$ and $b$ and cross the segment with equators which degenerate at endpoints then we get a Lagrangian 3-sphere, Hamiltonian isotopic to the zero section. To compare this one-parameter family of Lagrangian tori with the Oakley-Usher construction \cite{OU} we mentioned in Section 3, this family $L_{\lambda}$ corresponds to $L^{\lambda}_{1, 1}$.

From above approaches we get all the information to count Maslov two disks with boundary on $L_{\lambda}$ so that we can write down the disk potential function explicitly. In specific coordinates it is
\begin{equation}
\mathfrak{PO}(b)=x+y^{-1}+xz^{-1}+y^{-1}z, \quad b\in H^{1}(L_{\lambda}; \Lambda_{0}).
\end{equation}
We omit the energy parameter here since $L_{\lambda}$ is monotone. It is easy to check that this potential function has a one-dimensional critical loci, which indicates that with respect to some weak bounding cochain the Floer cohomology of $L_{\lambda}$ is nonzero hence $L_{\lambda}$ is nondisplaceable in $T^{*}S^{3}$.

If we consider a Lagrangian 3-sphere $S$ in a symplectic 6-manifold $X$ then $L_{\lambda}$ sits inside a neighborhood of $S$ for small $\lambda$. Due to the global symplectic geometry of $X$ our local torus $L_{\lambda}$ may bound more higher energy holomorphic disks with Maslov index two. Therefore the potential function may have more higher energy terms and may fail to have global critical points. Correspondingly, our torus $L_{\lambda}$ may be displaceable in $X$. Indeed if the Lagrangian 3-sphere $S$ is displaceable in $X$, then $L_{\lambda}$ is displaceable for small $\lambda$.

\subsection{Conifold transition}
Before constructing the moduli spaces of holomorphic cylinders we first describe some topological aspects of the conifold transition, mostly following \cite{STY}. By a 3-fold ordinary double point, or a \textit{node}, we mean a complex singularity analytically equivalent to
$$
\lbrace xy-zw=0 \rbrace \subset \mathbb{C}^{4}.
$$
There are two ways to desingularize the node. One is by considering its deformation, or the \textit{smoothing}
$$
\lbrace xy-zw=\epsilon \rbrace \subset \mathbb{C}^{4}
$$
which is a complex symplectic smooth hypersurface equipped with the induced symplectic structure on $\mathbb{C}^{4}$. It is symplectomorphic to the total space of the cotangent bundle of a 3-sphere, no matter $\epsilon$ is, while its complex structure depends on $\epsilon$. The other desingularisation is a \textit{small resolution}. We first blow up the singular point, getting a smooth complex manifold with an exceptional divisor $\mathbb{C}P^{1}\times \mathbb{C}P^{1}$, then blow down either family of $\mathbb{C}P^{1}$. We have two choices of $\mathbb{C}P^{1}$ to blow down and the resulting manifolds are related by a flop. The complex structure on either one is canonical while the symplectic structure depends on the size of $\mathbb{C}P^{1}$. As a complex manifold, the small resolution is the total space of the holomorphic vector bundle $\mathcal{O}(-1)\oplus \mathcal{O}(-1)\rightarrow \mathbb{C}P^{1}$. We say a \textit{conifold transition} by passing from one desingularisation to the other.

Beyond this local picture, the conifold transition was generalized in \cite{STY} as a surgery of symplectic 6-manifolds, replacing a Lagrangian 3-sphere by a holomorphic $\mathbb{C}P^{1}$ with a correct normal bundle. In order to patch the local parameters together, some topological conditions on the symplectic manifold are needed.

\begin{theorem}(Theorem 2.9, \cite{STY})
Fix a symplectic 6-manifold $X$ with a collection of $n$ disjoint embedded Lagrangian 3-spheres $S_{i}$. There is a ``good'' relation
$$
\sum_{i} a_{i}[S_{i}]=0 \in H_{3}(X; \mathbb{Z}), \quad a_{i}\neq 0\quad \forall i
$$
if and only if there is a symplectic structure on one of the $2^{n}$ choices of conifold transitions of $X$ in the Lagrangian $S_{i}$, such that the resulting $\mathbb{C}P^{1}$s are symplectic.
\end{theorem}

One interesting question is that how symplectic invariants change under conifold transitions. The closed string case, like quantum cohomology, has been more studied by algebraic geometry and by symplectic sum constructions. The open string case like Floer theory is less touched, in particular for a global symplectic manifold, and we will explore some points in this note.

\subsection{An example about the quadric hypersurface}
Now we discuss a motivating example about the quadric hypersurface. Let
$$
Q_{3}=\lbrace [z_{0}, \cdots, z_{4}]\in \mathbb{C}P^{4}\mid z_{0}^{2}=\sum_{j=1}^{4}z_{j}^{2}\rbrace
$$
be the quadric hypersurface in $\mathbb{C}P^{4}$. It is a monotone symplectic manifold with the induced symplectic structure. And the real part $Q_{3, \mathbb{R}}=Q_{3}\cap \mathbb{R}P^{4}$ is a Lagrangian $3$-sphere. We can also obtain $Q_{3}$ by performing a symplectic cutting on the boundary of some disk bundle of $T^{*}S^{3}$. Then the zero section corresponds to the real part $Q_{3, \mathbb{R}}$ and the boundary of the disk bundle, after quotienting by the Hamiltonian $S^{1}$-action, becomes the divisor at infinity which is isomorphic to $\mathbb{C}P^{1}\times \mathbb{C}{P}^{1}$. In this point of view the quadric hypersurface is the ``simplest'' compactification of $T^{*}S^{3}$ by adding one divisor at infinity.

Note that the symplectic cutting behaves well with respect to the moment map
$$
\pi: T^{*}S^{3}\rightarrow P\subset \mathbb{R}^{3}
$$
we get a singular toric fibration
$$
\pi: Q_{3}\rightarrow P_{Q}\subset \mathbb{R}^{3}
$$
of $Q_{3}$. The new polytope $P_{Q}$ will be cut out by five affine functions
$$
x\geq 0; \quad -y\geq 0;\quad x-z\geq 0;\quad z-y\geq 0;\quad y-x+1\geq 0.
$$
So compared with the polytope of $T^{*}S^{3}$ there is one more face $y-x+1=0$, which corresponds to the divisor at infinity. Here we fix the constant $1$ just for simplicity. The symplectic manifold of the polytope $P_{Q}$ is only isomorphic to the actual hypersurface $Q_{3}$ up to a conformal parameter.

By using the toric degeneration method in \cite{NNU} the disk potential function of regular fibers can be explicitly computed. For example, over the point $(\frac{1}{3}, -\frac{1}{3}, 0)$ there is a monotone Lagrangian 3-torus $L$. Its disk potential function is
\begin{equation}
\mathfrak{PO}(b)=x+y^{-1}+xz^{-1}+y^{-1}z+x^{-1}y, \quad b\in H^{1}(L; \Lambda_{0}).
\end{equation}
Compared with the case in $T^{*}S^{3}$, there is one more term in the potential function due to the new divisor at infinity. Directly we can check that the new potential function has three critical points, which shows that $L$ carries three different local systems as three different objects in the monotone Fukaya category of $Q_{3}$.

Moreover, by the work of Smith \cite{Smith} the Lagrangian sphere $Q_{3, \mathbb{R}}$ split-generates the monotone Fukaya category with eigenvalue zero. (It also follows from Evans-Lekili \cite{EL} since $Q_{3, \mathbb{R}}$ is a Lagrangian $SU(2)$-orbit.) Note that the sum of Betti numbers of $Q_{3}$ is four. Therefore the sphere and the monotone torus with three bounding cochains split-generate the whole monotone Fukaya category.

Since the Lagrangian sphere $Q_{3, \mathbb{R}}$ is homologically trivial we can perform conifold transition on it. The resulting manifold $\widetilde{Q_{3}}$ happens to be toric and one can check that the critical loci of the potential function are six toric fibers with bounding cochains, which match the sum of Betti numbers of $\widetilde{Q_{3}}$. \textit{Therefore three torus branes are merged and transformed into a sphere brane under the (reversed) conifold transition!} This is a 6-dimensional analogue of 4-dimensional phenomenon in \cite{FOOO3}, where the ``baby conifold transition'' of the second quadric hypersurface $Q_{2}=\mathbb{C}P^{1}\times \mathbb{C}P^{1}$ was studied.

Hence motivated by \cite{FOOO3} all the Lagrangian tori over the line in the polytope connecting the sphere brane and the monotone torus brane are expected to be nondisplaceable. The proof of the 4-dimensional case in \cite{FOOO3} considers the bulk-deformed potential functions of these tori, which have critical points for particular bulk deformations. However the same technique fails in our 6-dimensional situation. One reason is that the topology of $Q_{3}$ is ``too simple'' for us. To compute the bulk-deformed potential function explicitly one often uses divisors as bulk deformations. The only 4-cycle of $Q_{3}$ is the divisor at infinity. After direct computations we find it does not help us to produce critical points of potential functions of our Lagrangians. This motivates us to use other faces of the polytope as bulk deformations. However, the preimages of other four faces attaching the Lagrangian sphere are four 4-chains, not 4-cycles since they bound the 3-sphere. And we cannot naively use chains as bulk deformation since the squares of some boundary operators are not zero.

If we want to use those 4-chains to perturb the Floer cohomology of our toric fiber, the key problem is to cancel the ``boundary effect'' of these chains. To achieve this goal we introduce the moduli space of holomorphic cylinders.

Another direction which avoids using these 4-chains is to look at other nodal toric Fano 3-folds. In particular when the second Betti number is large. Then there are more 4-cycles to do bulk deformation and one is more likely to prove the local tori are nondisplaceable since there are more parameters. As we mentioned in the introduction, there is a full classification \cite{Ga} of 100 nodal toric Fano threefolds where one can do computations explicitly.

\subsection{Weakly unobstructedness of local tori}
In the last subsection we compactify $D_{r}T^{*}S^{3}$ to be the quadric hypersurface, which has a toric degeneration to a Fano variety admitting a small resolution, such that our local tori become toric fibers. A direct consequence is that the local tori are weakly unobstructed, see Theorem 10.1 \cite{NNU} and Theorem 1.2 \cite{NNU2}. However, for a general symplectic 6-manifold $X$ containing a Lagrangian sphere $S$, to show that local tori near $S$ are weakly unobstructed is not easy. For example, in the general toric case without assuming the Fano condition, the weakly unobstructedness \cite{FOOO1} is proved by using the $T^{n}$-action on moduli space of disks. In this article we assume Condition \ref{condition} to make our local Lagrangian tori to be weakly unobstructed under an energy bound. Now we give some speculations about how to relax this condition. We can first allow $J$-holomorphic spheres with zero first Chern numbers, as indicated in Remark 3.6 \cite{Au}. Next we may use the notion of ``broken Floer theory'' by Charest-Woodward \cite{CW}.

In \cite{CW} the weakly unobstructedness is shown for (local) toric fibers near a blow-up locus, a reverse flip and for the Clifford torus in a Darboux chart. For example, they have the following result.

\begin{theorem}(Theorem 7.21 \cite{CW})
Let $(X^{2n}, E, \omega)$ be a rational symplectic manifold with an exceptional divisor $E$ of small volume. That is, $E\simeq\mathbb{C}P^{n-1}$ with normal bundle isomorphic to $\mathcal{O}(-1)$. Let $L$ be a local toric fiber near $E$. Then there exists suitable perturbation data such that the Fukaya algebra of $L$ is weakly unobstructed. Moreover, we have that
$$
H^{1}(L; \Lambda_{0})\subset \mathcal{M}_{weak}(L)
$$
hence for any $b\in H^{1}(L; \Lambda_{0})$ the Floer cohomology $HF(L, b)$ is well-defined.
\end{theorem}

Their method seems very likely to be applicable to our case for any rational symplectic manifold, since our local tori live in a Fano almost toric piece $Q_{3}$ after degeneration. That is, when $X$ is rational we hope to prove that the local tori are always weakly unobstructed without assuming Condition \ref{condition}. This will be pursued in future work, and currently in this article we still assume that our local torus satisfies Condition \ref{condition} for some $J$.

\subsection{Holomorphic disks and cylinders}
Let $X$ be a symplectic 6-manifold and $S$ be an integrally homologically trivial Lagrangian 3-sphere in $X$. Now we will prove Theorem \ref{construction}, which works in a more general case without assuming $L$ to be a local torus. In next subsection we will apply it to the case of local tori. Fix a 4-chain $K$ in $X$ such that $\partial K=S$ and $\mathfrak{b}=w\cdot K$ with $w\in \Lambda_{+}$. Let $L$ be an oriented Lagrangian submanifold of $X$ such that $L\cap K=\emptyset$ and $L$ satisfies Condition \ref{condition}.

Set $E:=\min\lbrace E_{S}+ v(w), 2v(w), E_{+}\rbrace$, we will construct an $A_{\infty}$-structure $\lbrace \mathfrak{m}^{cy, \mathfrak{b}}_{k}\rbrace$ on the singular cohomology of $L$, with coefficients in $\Lambda_{0}\big/ T^{E}\cdot\Lambda_{0}$. For a disk class $\beta\in \pi_{2}(X,L)$, consider the moduli space of holomorphic disks $\mathcal{M}_{l,k}(\beta)$, the set of $J$-holomorphic maps
$$
u: (D, \partial D)\rightarrow (X,L)
$$
with $l$ interior marked points and $k$ boundary marked points modulo automorphism, of the class $\beta$. Since we only construct a theory modulo $T^{E}$, we assume that $\omega(\beta)<E$.

When $l=0$ and $k\geq 0$, let
$$
\mathcal{M}_{0, k+1}(\beta; (x_{1}, \cdots, x_{k}))
$$
be the compactified moduli space of holomorphic disks of class $\beta$ with $k+1$ boundary marked points, such that the last $k$ marked points are mapped to $k$ singular chains $(x_{1}, \cdots, x_{k})$ in $L$ respectively. Then by using the first boundary marked point, we get a chain
$$
ev: \mathcal{M}_{0, k+1}(\beta; (x_{1}, \cdots, x_{k}))\rightarrow L
$$
and we can define an operator
$$
\mathfrak{m}_{k;\beta}: C_{*}(L)^{\otimes k}\rightarrow C_{*}(L).
$$
Setting $\mathfrak{m}_{k}:=\sum_{\beta}\mathfrak{m}_{k;\beta}T^{\omega(\beta)}$ we get a collection of operators $\lbrace \mathfrak{m}_{k}\rbrace$ which satisfies the (curved) $A_{\infty}$-relation. This is the $A_{\infty}$-structure on $C_{*}(L)$ defined in \cite{FOOO}, without bulk deformations.

Now consider the case when $l=1$ and $k\geq 0$, let
$$
\mathcal{M}_{1, k+1}(\beta; (x_{1}, \cdots, x_{k}), K)
$$
be the compactified moduli space of holomorphic disks of class $\beta$ with $k+1$ boundary marked points and one interior marked point, such that the last $k$ marked points are mapped to $k$ singular chains $(x_{1}, \cdots, x_{k})$ in $L$, and the interior marked point is mapped to $K$. When $K$ is replaced by a cycle, those moduli spaces are used in \cite{FOOO} to define an $A_{\infty}$-structure on $C_{*}(L)$ with bulk deformation by that cycle. In the following we will explain how to use the chain $K$ as a bulk deformation, up to some energy.

Different from the case that $K$ is a cycle, there is one more codimension one boundary in $\mathcal{M}_{1, k+1}(\beta; (x_{1}, \cdots, x_{k}), K)$, which corresponds to that the interior marked point goes to the boundary of $K$. We write this codimension one boundary as $\mathcal{M}_{1, k+1}(\beta; (x_{1}, \cdots, x_{k}), S=\partial K)$. If we directly use $\mathcal{M}_{1, k+1}(\beta; (x_{1}, \cdots, x_{k}), K)$ to define operators, then these operators will not obey the $A_{\infty}$-relation, nor have the invariance property. So we introduce moduli space of holomorphic cylinders to compensate these extra codimension one boundaries.

Consider the homotopy classes of cylinders, with one end on $S$ and the other on $L$. Note that $S$ is simply-connected, any homotopy class of such cylinders corresponds to a disk class $\beta\in\pi_{2}(X,L)$. Hence we abuse the notation to use $\beta$ as a disk class or a cylinder class. Fix a class $\beta$ with $\omega(\beta)<\min\lbrace E_{S}, E_{+}\rbrace$, we will construct another moduli space $\mathcal{M}_{1, k+1}^{cy}(\beta; (x_{1}, \cdots, x_{k}), S)$. We remark that for those classes having energy greater than $\min\lbrace E_{S}, E_{+}\rbrace$, once they are weighted by the energy parameter $T^{v(w)}$ from $\mathfrak{b}$, the total energy is greater than $E$. So we do not need to study them by energy reasons.

We write the domain of a cylinder as
$$
A_{\epsilon, p}=\lbrace z\in \mathbb{C}\mid \abs{z} \leq 1, \abs{z-p}\geq \epsilon, \epsilon< 1-\abs{p}\rbrace
$$
where $0<\epsilon <1$ is a conformal parameter and $p$ is an interior point in the unit disk. Topologically the domain is an annulus with two disjoint boundaries $C_{\epsilon}$ and $C_{1}$. With respect to an almost complex structure $J$ in Condition \ref{condition}, we consider the $J$-holomorphic maps
$$
\lbrace u: A_{\epsilon, p}\rightarrow X\mid u(C_{1}) \in L, u(C_{\epsilon})\in S \rbrace.
$$
And the moduli space $\widetilde{\mathcal{M}}_{1, k+1}^{cy}(\beta; (x_{1}, \cdots, x_{k}), S)$ contains all such maps $u$ representing a homotopy class $\beta$ with $k+1$ marked points on the boundary $C_{1}$, modulo automorphisms. Here the last $k$ marked points are mapped to $k$ singular chains $(x_{1}, \cdots, x_{k})$ in $L$. Next we compactify this moduli space.

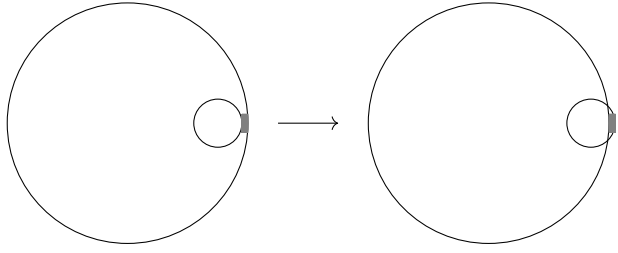
\begin{figure}
  \begin{tikzpicture}[xscale=0.8, yscale=0.8]
  \draw (0,0) circle [radius=2];
  \draw (1.5,0) circle [radius=0.4];
  \filldraw [gray] (1.9,0.15)--(2,0.15)--(2,-0.15)--(1.9,-0.15)--(1.9,0.15);
  \draw [->] (2.5,0)--(3.5,0);
  \draw (6,0) circle [radius=2];
  \draw (7.7,0) circle [radius=0.4];
  \filldraw [gray] (8,0.15)--(8.1,0.15)--(8.1,-0.15)--(8,-0.15)--(8,0.15);
  \end{tikzpicture}
  \caption{Degeneration when circle ends meet.}
\end{figure}

\begin{theorem}\label{compactification}
There is a compactification $\mathcal{M}_{1, k+1}^{cy}(\beta; (x_{1}, \cdots, x_{k}), S)$ of the moduli space $\widetilde{\mathcal{M}}_{1, k+1}^{cy}(\beta; (x_{1}, \cdots, x_{k}), S)$, such that it has a codimension one boundary component
$$
\partial^{cy} \mathcal{M}_{1, k+1}^{cy}(\beta; (x_{1}, \cdots, x_{k}), S)=-\mathcal{M}_{1, k+1}(\beta; (x_{1}, \cdots, x_{k}), S)
$$
with respect to suitably chosen orientations.
\end{theorem}
\begin{proof}
The construction of the compactification is by adding all possible degenerations. And the verification of the compactness will be proved by a gluing method.

First we consider the case when $p$ is fixed but $\epsilon$ goes to zero. Then in the limit we add a holomorphic disk with one interior point attaching on $S$. Conversely we need to do the gluing to resolve this interior point. The gluing analysis here is similar to the gluing when one study open Gromov-Witten theory and the boundary class of the given disk class is trivial. We describe the construction here following Proposition 3.8.27 and Subsection 7.4.1 in \cite{FOOO}.

For a holomorphic disk with an interior point mapping to $S$, the idea to ``blow up'' this interior point to get a holomorphic cylinder is first glue a constant disk to this point then convert this boundary gluing to a interior gluing. Let $D(1)$ be the unit disk in $\mathbb{C}$. Consider a holomorphic map
$$
u: D(1)\rightarrow X, \quad u(\partial D(1))\in L
$$
with two marked points. One marked point $z_{0}=(1, 0)$ on the boundary and one interior marked point $w_{0}=(0, 0)$ with $u(w_{0})\in S$. Let $D(\sigma)$ be a small disk with one boundary marked point $z_{1}$ and $\Sigma$ be a nodal surface such that
$$
\Sigma= D(1)\sqcup D(\sigma) \big/ (0,0)\sim (0,0).
$$
Then we consider a holomorphic map $w_{u}$, which is induced from $u$.
$$
w_{u}(z)=
\begin{cases}
u(z) \quad z\in D(1),\\
u(w_{0}) \quad z\in D(\sigma).
\end{cases}
$$
That is, the restriction of the map $w_{u}$ on $D(\sigma)$ is a constant map. Next, under transversality assumptions, several standard steps give us the gluing conclusion.
\begin{enumerate}
\item First we smooth the singular point of $\Sigma$ as an interior singular point to get the pregluing map, without being holomorphic.
\item Then we apply the implicit function theorem to get a genuine holomorphic cylinder with two boundary marked points $z_{0}$ and $z_{1}$. Here $z_{0}$ is on the positive boundary and $z_{1}$ is on the negative boundary.
\item We forget the marked point $z_{1}$ by a forgetful map. The image of the forgetful map is parameterized by the small disk $D(\sigma)$.
\item In the end we check that the implicit function theorem and the forgetful map is $S^{1}$-equivariant with respect to the standard rotation action on $D(\sigma)$. And we modulo this action to obtain a neighborhood of $u$ as $u\times D(\sigma)\big/S^{1}=u\times [0, \sigma)$.
\end{enumerate}
This cylinder-to-disk degeneration gives us a codimension one boundary
$$
\partial^{cy} \mathcal{M}_{1, k+1}^{cy}(\beta; (x_{1}, \cdots, x_{k}), S)
$$
which matches the moduli space
$$
-\mathcal{M}_{1, k+1}(\beta; (x_{1}, \cdots, x_{k}), S)
$$
up to an orientation. When $\beta$ is energy minimal, this codimension one boundary is the only boundary of the moduli space.

The second case is that $p$ is fixed and $\epsilon$ goes to $1-\abs{p}$. That is, two boundary $C_{1}$ and $C_{\epsilon}$ meet. In the limit a small region between these two circle boundaries converges to a holomorphic strip, see Figure 3. Since this strip splits from a finite energy map, itself also has finite energy. Hence the two ends of the strip converge to intersection points of $S$ and $L$, which is empty. In conclusion such a degeneration does not happen.

The third case is that when $\epsilon$ goes to zero and $p$ goes to $C_{1}$.
\begin{enumerate}
\item When $\lim\frac{\epsilon}{1-\abs{p}}=c> 0$. Then by a conformal change the domain becomes a disk with an annulus bubble, with the modulus of the annulus bubble determined by $c$.

\item When $\lim\frac{\epsilon}{1-\abs{p}}=0$. It is a similar case as above where $c=0$, the annulus bubble becomes a disk bubble, with one interior point attached to $S$.

\item When $\lim\frac{\epsilon}{1-\abs{p}}=+\infty$. Then two circle boundaries meet much faster than $\epsilon$ goes to zero. This degeneration will end up with a holomorphic strip as in the second case. So we exclude it in the same way.
\end{enumerate}

Other cases include disk splittings and sphere bubbles. As we remarked before, sphere bubbles are of codimension two and disk bubbles with boundary on $S$ do not happen by energy reasons. The only possible disk splittings are on the side of $L$, which will eventually give us the $A_{\infty}$-relation, together with (1) and (2) in the third case.

In conclusion we add all possible degenerations to compactify the moduli space. And there is a particular codimension one boundary component which comes from the circle boundary $C_{\epsilon}$ shrinking to a point.
\end{proof}

Then we glue the two moduli spaces together to obtain a new moduli space.
\begin{corollary}\label{union}
There are fundamental chains on moduli spaces $\mathcal{M}_{1, k+1}(\beta; (x_{1}, \cdots, x_{k}), K)$ and $\mathcal{M}_{1, k+1}^{cy}(\beta; (x_{1}, \cdots, x_{k}), S)$ such that we can glue them along their boundaries to obtain a moduli space
$$
\begin{aligned}
&\mathcal{M}_{1, k+1}(\beta; (x_{1}, \cdots, x_{k}), K+S)\\
=& \mathcal{M}_{1, k+1}(\beta; (x_{1}, \cdots, x_{k}), K)\sqcup\mathcal{M}_{1, k+1}^{cy}(\beta; (x_{1}, \cdots, x_{k}), S)\big/ \sim
\end{aligned}
$$
where the equivalence $\sim$ is given in Theorem \ref{compactification}.
\end{corollary}
\begin{proof}
The fundamental chains on related moduli spaces are constructed by classical methods. We start with energy minimal classes to glue moduli spaces such that corresponding holomorphic curves are somewhere injective. Then for general classes we use domain dependent almost complex structures $J_{\Sigma}$ for $\Sigma$ a disk or a cylinder. For each point $p\in\Sigma$, we assume it is in the component of $J$ in Condition 1.1. Moreover when $\Sigma$ is a cylinder, we also requires that the almost complex structure depends on the modulus of domain.
\end{proof}

By using the first boundary marked point we get a singular chain
$$
ev: \mathcal{M}_{1, k+1}(\beta; (x_{1}, \cdots, x_{k}); K+S)\rightarrow L.
$$
The expected dimension of this fundamental chain is
$$
\dim \mathcal{M}_{1, k+1}(\beta; (x_{1}, \cdots, x_{k}); K_{i}+S)=\mu(\beta)+k +1- \sum_{j=1}^{k}(3- d_{j})
$$
where $d_{j}$ is the dimension of the singular chain $x_{j}$.

These chains
$$
ev: \mathcal{M}_{1, k+1}(\beta; (x_{1}, \cdots, x_{k}); K+S)\rightarrow L
$$
will play the role of $\mathfrak{q}_{l,k; \beta}$ when the interior marked point is attached on a chain $K$, not a cycle. For $\mathfrak{b}=K$ we define
$$
\mathfrak{q}_{1, k; \beta}^{cy, \mathfrak{b}}: C^{*}(L)^{\otimes k}\rightarrow C^{*}(L)
$$
by
\begin{equation}
(x_{1}, \cdots, x_{k}) \mapsto (ev: \mathcal{M}_{1, k+1}(\beta; (x_{1}, \cdots, x_{k}); K+S)\rightarrow L)
\end{equation}
and extend it linearly over $\Lambda_{+}$. That is, for $\mathfrak{b}= wK$ with $w\in \Lambda_{+}$, we define
\begin{equation}
\mathfrak{q}_{1, k; \beta}^{cy, \mathfrak{b}}(x_{1}, \cdots, x_{k}):= w\cdot\mathfrak{q}_{1, k; \beta}^{cy, K}(x_{1}, \cdots, x_{k}).
\end{equation}
Similarly we define $\mathfrak{q}_{1, 0; \beta}^{cy, \mathfrak{b}}$ as the chain
\begin{equation}
w\cdot (ev: \mathcal{M}_{1, 1}(\beta; K+S)\rightarrow L)
\end{equation}
with coefficient $w$. Then we define the operator $\mathfrak{q}_{1, k}^{cy, \mathfrak{b}}$ to be
\begin{equation}
\mathfrak{q}_{1, k}^{cy, \mathfrak{b}}= \sum_{\beta} \mathfrak{q}_{1, k; \beta}^{cy, \mathfrak{b}}\cdot T^{\omega(\beta)}.
\end{equation}
By the Gromov compactness theorem the right hand side is a finite sum. These operators are only defined for class $\beta$ with $\omega(\beta)<\min\lbrace E_{S}, E_{+}\rbrace$. Otherwise, after weighted by $T^{\omega(\beta)+v(w)}$, they vanish automatically since we work modulo $T^{E}$.

As we mentioned in the beginning of Section 2, here we abuse the notations between singular chains and cochains via the Poincar\'e duality. Geometrically when we define moduli space of holomorphic curves with point constraints, we think our curves with points attached on certain chains. But for algebraic convenience we think these chains as cochains with gradings reversed and shifted. And for notational simplicity, we use $K$ to both represent the 4-chain $K$ or $PD(K)$, as a degree two cochain.

Being analogous to (\ref{bulk}), we define
\begin{equation}
\begin{aligned}
\mathfrak{m}^{cy,\mathfrak{b}}_{k}(x_{1}, \cdots, x_{k})& :=\sum_{l=0}^{\infty}\mathfrak{q}_{l, k}^{cy,\mathfrak{b}}(\mathfrak{b}^{\otimes l}; x_{1}, \cdots, x_{k})\\
& :=\mathfrak{q}_{0, k}^{cy,\mathfrak{b}}(1; x_{1}, \cdots, x_{k})+ \mathfrak{q}_{1, k}^{cy,\mathfrak{b}}(\mathfrak{b}; x_{1}, \cdots, x_{k}).
\end{aligned}
\end{equation}
Here operators $\mathfrak{q}_{1, k}^{cy,\mathfrak{b}}$ are defined as above by using holomorphic cylinders. The operators $\mathfrak{q}_{0, k}^{cy,\mathfrak{b}}$ are defined to be $\mathfrak{m}_{k}$, see (\ref{q0}), which are the operators in the $A_{\infty}$-structure on $L$ without using interior marked points. Hence by definition, the above definition becomes
\begin{equation}
\begin{aligned}
\mathfrak{m}^{cy,\mathfrak{b}}_{k}(x_{1}, \cdots, x_{k})& :=\sum_{l=0}^{\infty}\mathfrak{q}_{l, k}^{cy,\mathfrak{b}}(\mathfrak{b}^{\otimes l}; x_{1}, \cdots, x_{k})\\
&=\mathfrak{m}_{k}(x_{1}, \cdots, x_{k})+ \mathfrak{q}_{1, k}^{cy,\mathfrak{b}}(\mathfrak{b}; x_{1}, \cdots, x_{k}).
\end{aligned}
\end{equation}
Now we hope it is more clear that these operators serve as a zeroth and first order approximation of the genuine $A_{\infty}$-structure with bulk deformations, since all $\mathfrak{q}_{l, k}^{cy,\mathfrak{b}}$ are defined to be zero for $l\geq 2$. We will show that they also give an $A_{\infty}$-relation modulo some energy.

\begin{proposition}
With above notations, the operators $\lbrace \mathfrak{m}^{cy,\mathfrak{b}}_{k} \rbrace$ give a curved $A_{\infty}$-relation on $C^{*}(L)$ modulo $T^{E}$.
\end{proposition}
\begin{proof}
To prove the $A_{\infty}$-relation for $\lbrace \mathfrak{m}^{cy,\mathfrak{b}}_{k} \rbrace$ we need to check for each $k$ that
\begin{equation}\label{Ain}
\sum_{k_{1}+k_{2}=k+1}\sum_{i}(-1)^{*} \mathfrak{m}^{cy,\mathfrak{b}}_{k_{1}}(x_{1}, \cdots, \mathfrak{m}^{cy,\mathfrak{b}}_{k_{2}}(x_{i}, \cdots, x_{i+k_{2}-1}), \cdots, x_{k}) \equiv 0 \mod T^{E}
\end{equation}
where $*= \deg x_{1}+\cdots +\deg x_{i-1}+i-1$. This can be shown by looking at boundaries of various one-dimensional moduli spaces.

Note that $\mathfrak{m}^{cy,\mathfrak{b}}_{k}= \mathfrak{m}_{k}+ \mathfrak{q}_{1, k}^{cy,\mathfrak{b}}$ by definition. There are four types of compositions in (\ref{Ain}):
$$
\mathfrak{m}_{k_{1}}\circ \mathfrak{m}_{k_{2}},\quad \mathfrak{q}_{1, k_{1}}^{cy,\mathfrak{b}}\circ \mathfrak{m}_{k_{2}},\quad \mathfrak{m}_{k_{1}}\circ \mathfrak{q}_{1, k_{2}}^{cy,\mathfrak{b}},\quad \mathfrak{q}_{1, k_{1}}^{cy,\mathfrak{b}}\circ \mathfrak{q}_{1, k_{2}}^{cy,\mathfrak{b}}.
$$
The first type of terms $\mathfrak{m}_{k_{1}}\circ \mathfrak{m}_{k_{2}}$ is zero since $\mathfrak{m}_{k}$ satisfies an $A_{\infty}$-relation. The last type of terms $\mathfrak{q}_{1, k_{1}}^{cy,\mathfrak{b}}\circ\mathfrak{q}_{1, k_{2}}^{cy,\mathfrak{b}}$ is zero modulo $T^{E}$ since both of them are weighted by at least $T^{v(w)}$.

Now it remains to show that the sum of second type and third type terms
$$
\pm\mathfrak{q}_{1, k_{1}}^{cy,\mathfrak{b}}\circ \mathfrak{m}_{k_{2}}\pm \mathfrak{m}_{k_{1}}\circ \mathfrak{q}_{1, k_{2}}^{cy,\mathfrak{b}}
$$
is zero. We consider the moduli space
$$
\mathcal{M}_{1, k+1}(\beta; (x_{1}, \cdots, x_{k}); K+S)
$$
in Theorem \ref{compactification}. For simplicity, we assume that $x_{1}, \cdots, x_{k}$ are singular \textit{cycles}. By definition, this moduli space is a union of two moduli spaces. One part is a moduli space of holomorphic cylinders with boundary insertions. The other part is a moduli space of holomorphic disks with boundary insertions and one interior insertion. Both parts have a codimension one boundary coming from cylinder-to-disk degeneration or marked point going to the boundary of the chain $K$. And by gluing these two moduli spaces together, this codimension one boundary is cancelled. Then we look at other codimension one boundaries. With this understood, the following check is similar to the usual $A_{\infty}$-relations on a monotone spin Lagrangian submanifold.

Under Condition \ref{condition}, there is no disk bubble on the side of Lagrangian sphere $S$, with energy less than $E_{S}$. The only codimension one boundaries are disk splitting on the side of the Lagrangian torus $L$. So the codimension one boundary of the cylinder part corresponds to a disk splitting from a cylinder. And the codimension one boundary of the disk part corresponds to a disk splitting from a disk with interior marked point attached on $K$. Those codimension one boundaries together give the terms $\pm\mathfrak{q}_{1, k_{1}}^{cy,\mathfrak{b}}\circ \mathfrak{m}_{k_{2}; \beta_{2}}\pm \mathfrak{m}_{k_{1}; \beta_{1}}\circ \mathfrak{q}_{1, k_{2}}^{cy,\mathfrak{b}}$ with $\beta_{1}, \beta_{2}\neq 0$. That is, the disk splitting is not a constant disk (see Figure 4). On the other hand, the operator $\mathfrak{m}_{1;\beta=0}$ is defined as taking the boundary of a singular chain weighted by signs. Then $\mathfrak{m}_{1;\beta=0}\circ \mathfrak{q}_{1, k_{1}+k_{2}-1}^{cy,\mathfrak{b}}$ corresponds to the codimension one boundary of the moduli space $\mathcal{M}_{1, k+1}(\beta; (x_{1}, \cdots, x_{k}); K+S)$. Therefore after a careful check of signs, those two types of operators cancel with each other since they both correspond to the codimension one boundary of a moduli space. We assume $L$ is oriented hence spin, the signs will be satisfied by similar computations in \cite{FOOO} for the genuine bulk-deformed $A_{\infty}$-relations. Then we complete the proof of Theorem \ref{construction}.

We remark that here we assume $x_{1}, \cdots, x_{k}$ are singular cycles. If they are chains then the moduli space $\mathcal{M}_{1, k+1}(\beta; (x_{1}, \cdots, x_{k}); K+S)$ will have more codimension one boundaries when boundary marked points goto $\partial x_{i}$. These boundary terms will cancel with $\mathfrak{q}_{1, k_{1}+k_{2}-1}^{cy,\mathfrak{b}}\circ \mathfrak{m}_{1;\beta=0}$ in the $A_{\infty}$-relation.
\end{proof}

\begin{figure}
  \begin{tikzpicture}
  \draw (1,0) circle [radius=0.8];
  \draw (1,0) circle [radius=0.25];
  \draw [fill] (1.8,0) circle [radius=0.05];
  \draw [fill] (1,0.8) circle [radius=0.05];
  \draw [fill] (0.2,0) circle [radius=0.05];
  \draw [fill] (1,-0.8) circle [radius=0.05];
  \draw [->] (2.5,0)--(3.5,0);
  \draw (4.5,0) circle [radius=0.5];
  \draw (5.5,0) circle [radius=0.5];
  \draw (5.5,0) circle [radius=0.15];
  \draw [fill] (4.15,0.35) circle [radius=0.05];
  \draw [fill] (4.15,-0.35) circle [radius=0.05];
  \draw [fill] (5.85,0.35) circle [radius=0.05];
  \draw [fill] (5.85,-0.35) circle [radius=0.05];
  \node at (6.5,0) {$\pm$};
  \draw (7.5,0) circle [radius=0.5];
  \draw (7.5,0) circle [radius=0.15];
  \draw (8.5,0) circle [radius=0.5];
  \draw [fill] (7.15,0.35) circle [radius=0.05];
  \draw [fill] (7.15,-0.35) circle [radius=0.05];
  \draw [fill] (8.85,0.35) circle [radius=0.05];
  \draw [fill] (8.85,-0.35) circle [radius=0.05];

  \draw (1,-2) circle [radius=0.8];
  \draw [fill] (1,-2) circle [radius=0.05];
  \draw [fill] (1.8,-2) circle [radius=0.05];
  \draw [fill] (1,-1.2) circle [radius=0.05];
  \draw [fill] (0.2,-2) circle [radius=0.05];
  \draw [fill] (1,-2.8) circle [radius=0.05];
  \draw [->] (2.5,-2)--(3.5,-2);
  \draw (4.5,-2) circle [radius=0.5];
  \draw (5.5,-2) circle [radius=0.5];
  \draw [fill] (5.5,-2) circle [radius=0.05];
  \draw [fill] (4.15,-1.65) circle [radius=0.05];
  \draw [fill] (4.15,-2.35) circle [radius=0.05];
  \draw [fill] (5.85,-1.65) circle [radius=0.05];
  \draw [fill] (5.85,-2.35) circle [radius=0.05];
  \node at (5,-1) {$\mathfrak{q}_{1, k_{1}}^{cy,\mathfrak{b}}\circ \mathfrak{m}_{k_{2}}$};
  \node at (6.5,-2) {$\pm$};
  \node at (8,-1) {$\mathfrak{m}_{k_{1}}\circ \mathfrak{q}_{1, k_{2}}^{cy,\mathfrak{b}}$};
  \draw (7.5,-2) circle [radius=0.5];
  \draw [fill] (7.5,-2) circle [radius=0.05];
  \draw (8.5,-2) circle [radius=0.5];
  \draw [fill] (7.15,-1.65) circle [radius=0.05];
  \draw [fill] (7.15,-2.35) circle [radius=0.05];
  \draw [fill] (8.85,-1.65) circle [radius=0.05];
  \draw [fill] (8.85,-2.35) circle [radius=0.05];
  \end{tikzpicture}
  \caption{Splittings of disks and cylinders.}
\end{figure}

We expect that the transversality result needed to define the above $A_{\infty}$-structure can be obtained by using one almost complex structure $J$ satisfying Condition \ref{condition} via virtual perturbation, or a family of $J$'s in a small neighborhood of a $J$ satisfying Condition \ref{condition} via classical means, see Remark \ref{div}. For another $J'$ satisfying Condition \ref{condition}, if we can connect $J$ with $J'$ by a family of $J$'s satisfying Condition \ref{condition} except $(2)$, we will get a homotopy equivalence between $A_{\infty}$-structures, by a cobordism argument.

\begin{proposition}
Let $\lbrace J_{t}\rbrace_{t\in[0,1]}$ be a smooth family of almost complex structures such that $J_{0}$ and $J_{1}$ satisfy Condition \ref{condition} and $J_{t}$ satisfies Condition \ref{condition} except $(2)$ for each $t\in[0,1]$. Then two $A_{\infty}$-structures $\mathfrak{m}^{cy,\mathfrak{b}; J_{0}}_{k}$ and $\mathfrak{m}^{cy,\mathfrak{b}; J_{1}}_{k}$ are homotopy equivalent.
\end{proposition}

Now we remark about the dependence of the choice of the chain $K$ we use. Suppose that we have two four-chains $K$ and $K'$ such that $\partial K=\partial K' =S$, $K\cap L=K'\cap L=\emptyset$ and they represent the same class in $H_{4}(X, S)$. Then we expect that there is an $A_{\infty}$-homotopy equivalence between $\mathfrak{m}^{cy,\mathfrak{b}}_{k}$ and $\mathfrak{m}^{cy,\mathfrak{b}'}_{k}$, where $\mathfrak{b}=wK$ and $\mathfrak{b}'=wK'$. If $K$ and $K'$ can be joined by a family of four-chains $K_{t}$ with $\partial K_{t}= S$ such that $K_{t}\cap L=\emptyset$ for all $t$, then the above $A_{\infty}$-homotopy equivalence induces an identity map on potential functions. In general there may be some $K_{t}$ intersecting $L$. Then the above $A_{\infty}$-homotopy equivalence induces a coordinate change on potential functions. For genuine bulk deformations, this $A_{\infty}$-homotopy equivalence (and the coordinate change) is discussed in \cite{FOOO} and Section 2.5 in \cite{FOOO5}. We expect the technique therein combined with our gluing of the interior marked point gives the desired $A_{\infty}$-homotopy equivalence.

Once we have the $A_{\infty}$-relation modulo some energy, we can talk about the \textit{canonical model} of $\lbrace \mathfrak{m}^{cy,\mathfrak{b}}_{k} \rbrace$, see \cite{FOOO-ca}. That is, we have a new collection of operators $\lbrace \mathfrak{m}^{cy,\mathfrak{b}, can}_{k} \rbrace$ on the cohomology $H^{*}(L)$, such that $\lbrace \mathfrak{m}^{cy,\mathfrak{b}, can}_{k} \rbrace$ give an $A_{\infty}$-structure on $H^{*}(L)$ and $(\lbrace \mathfrak{m}^{cy,\mathfrak{b}, can}_{k} \rbrace, H^{*}(L))$ is homotopy equivalent to $(\lbrace \mathfrak{m}^{cy,\mathfrak{b}}_{k} \rbrace, C^{*}(L))$. From now on we will only use the canonical model $(\lbrace \mathfrak{m}^{cy,\mathfrak{b}, can}_{k} \rbrace, H^{*}(L))$ and we abuse the notation just write $\mathfrak{m}^{cy,\mathfrak{b}}_{k}$ instead of $\mathfrak{m}^{cy,\mathfrak{b}, can}_{k}$.

As in the usual Lagrangian Floer theory, an $A_{\infty}$-structure gives a complex. Now we define the $\mathfrak{b}$-deformed Floer complex, analogous to (\ref{bulk}) and (\ref{boundary operator}).

\begin{definition}
For $\mathfrak{b}=wK$ with $w\in \Lambda_{+}$ and $\rho\in H^{1}(L; \Lambda_{+})$, we define the operator
$$
\partial_{cy, \mathfrak{b}}^{\rho}: H^{*}(L; \Lambda_{+})\rightarrow H^{*}(L; \Lambda_{+})
$$
by
$$
\partial_{cy, \mathfrak{b}}^{\rho}(x)=\sum_{\beta} e^{\rho(\partial\beta)} \mathfrak{q}_{1,1;\beta}^{cy, \mathfrak{b}}(x)\cdot T^{\omega(\beta)}.
$$

The deformed complex is defined by
$$
(H^{*}(L; \Lambda_{+}), d^{\rho}_{cy, \mathfrak{b}}=\delta^{\rho}+ \partial_{cy, \mathfrak{b}}^{\rho}).
$$
\end{definition}
Here $\delta^{\rho}$ is similarly defined as
$$
\delta^{\rho}:=\mathfrak{m}_{1}^{\rho}(x)=\sum_{\beta} e^{\rho(\partial\beta)}\mathfrak{m}_{1,\beta}(x)\cdot T^{\omega(\beta)}.
$$
Then the $A_{\infty}$-formula (\ref{Ain}) gives us a complex.

\begin{proposition}
The operator $d^{\rho}_{cy, \mathfrak{b}}$ satisfies that
$$
(d^{\rho}_{cy, \mathfrak{b}})^{2} \equiv 0 \mod T^{E}.
$$
Hence we have a cohomology modulo $T^{E}$ which we write as $HF_{cy}(L; (\mathfrak{b}, \rho))$.
\end{proposition}

\begin{remark}
In this section we define the operators $\mathfrak{m}_{1}^{\rho}$ and $\partial_{cy, \mathfrak{b}}^{\rho}$ by using local systems $\rho$, which is different in the usual definition of bulk-deformed potential functions, where weak bounding cochains are used. But under Condition \ref{condition} there is no disk bubble with non-positive Maslov index, these two approaches are the same. This is proved in Section 4.1 in \cite{FOOO5} for the genuine bulk deformation case with all operators $\mathfrak{q}_{l,k;\beta}$. And here we only need to adapt the proof therein for our operators $\mathfrak{q}_{1,k;\beta}^{cy, \mathfrak{b}}$. More precisely, the proof boils down to prove the divisor axiom for the operator $\mathfrak{q}_{1, k; \beta}^{cy, \mathfrak{b}}$, which is given by the integration-along-fiber technique on the moduli spaces of disks, see Section 4.1 in \cite{FOOO5} and Lemma 7.1 in \cite{FOOO2} for the proof, or Section 3 in \cite{F} for more original statements. We expect those technique to work to establish the divisor axiom for our operators here, also see Remark 4.12.
\end{remark}

Therefore we obtain a cohomology theory for a fixed bulk chain $\mathfrak{b}=wK$. Its underlying complex is the singular cohomology of $L$ and its differential counts a combination of holomorphic disks and cylinders. An advantage of this cohomology is that we can do explicit computation by the help of a $\mathfrak{b}$-deformed potential function. For example, the existence of a critical point of the potential function gives us a non-vanishing result of the cohomology.

Now we assume that $L$ is a torus and define a $\mathfrak{b}$-deformed potential function
$$
\mathfrak{PO}^{cy, \mathfrak{b}}: H^{1}(L; \Lambda_{+})\rightarrow \Lambda_{+}.
$$
For a group homomorphism
$$
\rho: \pi_{1}(L)=H_{1}(L; \mathbb{Z})\rightarrow \Lambda_{0}-\Lambda_{+}
$$
we regard $e^{\rho(\cdot)}$ as an element in $H^{1}(L; \Lambda_{+})$. Then we define
\begin{equation}
\mathfrak{PO}^{cy, \mathfrak{b}}(\rho)= \sum_{\beta}e^{\rho(\partial\beta)}T^{\omega(\beta)}(\mathfrak{m}_{0; \beta}(1)+ \mathfrak{q}_{1, 0;\beta}^{cy, \mathfrak{b}}(1))
\end{equation}
where $\mathfrak{m}_{0;\beta}$ is the (undeformed) $A_{\infty}$-structure on $H^{*}(L)$, see Section 2. Here
$$
\mathfrak{m}_{0; \beta}(1)= PD([L])(\mathfrak{m}_{0; \beta})
$$
and
$$
\mathfrak{q}_{1, 0;\beta}^{cy, \mathfrak{b}}(1)=PD([L])(\mathfrak{q}_{1, 0;\beta}^{cy, \mathfrak{b}})
$$
are pairings between cochains and chains, which give us two numbers. We will call $\mathfrak{q}_{1, 0;\beta}^{cy, \mathfrak{b}}(1)$ a mixed type one-pointed open Gromov-Witten invariant of class $\beta$, since it has both disk and cylinder contributions. In particular cases, we will see that the mixed invariant is a sum of the usual one-pointed open Gromov-Witten invariant and the cylinder contribution. We remark that this mixed type invariant is invariant under the deformation of $J$ within Condition \ref{condition}, and small Hamiltonian perturbation $\phi(L)$ of $L$ if $\phi(L)$ satisfies Condition \ref{condition} with respect to the same $J$.

\begin{proposition}\label{decom}
Let $L$ be a Lagrangian torus. If the potential function $\mathfrak{PO}^{cy, \mathfrak{b}}(\rho)$ for $L$ has a critical point for some $(\mathfrak{b}, \rho)$ modulo $T^{E'}$, $E'< E$, then the deformed Floer cohomology satisfies that
$$
HF_{cy}(L; (\mathfrak{b}, \rho))\cong H^{*}(L; \dfrac{\Lambda_{0}}{T^{E'}\Lambda_{0}})\cong (\dfrac{\Lambda_{0}}{T^{E'}\Lambda_{0}})^{\oplus 8}.
$$
\end{proposition}
\begin{proof}
By a direct computation below we can find that if there is a critical point for some $(\mathfrak{b}, \rho)$ modulo $T^{E'}$ then the deformed boundary operator $d^{\rho}_{cy, \mathfrak{b}}\equiv 0$ modulo $T^{E'}$. So the cohomology is isomorphic to the underlying complex.

Let $\lbrace e_{i}\rbrace_{i=1, 2, 3}$ be a set of generators of $ H^{1}(L; \mathbb{Z})$. Then any $\rho\in H^{1}(L; \Lambda_{+})$ can be written as $\rho(x)=\sum_{i=1}^{3} x_{i}e_{i}$. And we have that
\begin{equation}
\begin{aligned}
\frac{\partial}{\partial x_{1}} \mathfrak{PO}^{cy, \mathfrak{b}}(\rho(x))=& \frac{\partial}{\partial x_{1}} \sum_{\beta}e^{\rho(\partial\beta)}T^{\omega(\beta)}(\mathfrak{m}_{0; \beta}(1)+ \mathfrak{q}_{1, 0;\beta}^{cy, \mathfrak{b}}(1))\\
=& \frac{\partial}{\partial x_{1}} \sum_{\beta}e^{(x_{1}e_{1}+x_{2}e_{2}+x_{3}e_{3})(\partial\beta)}T^{\omega(\beta)}(\mathfrak{m}_{0; \beta}(1)+ \mathfrak{q}_{1, 0;\beta}^{cy, \mathfrak{b}}(1))\\
=& \sum_{\beta}(e_{1}(\partial\beta))\cdot e^{(x_{1}e_{1}+x_{2}e_{2}+x_{3}e_{3})(\partial\beta)}T^{\omega(\beta)}(\mathfrak{m}_{0; \beta}(1)+ \mathfrak{q}_{1, 0;\beta}^{cy, \mathfrak{b}}(1))\\
=& \sum_{\beta}e^{(x_{1}e_{1}+x_{2}e_{2}+x_{3}e_{3})(\partial\beta)}T^{\omega(\beta)}(\mathfrak{m}_{1; \beta}(e_{1})+ \mathfrak{q}_{1, 1;\beta}^{cy, \mathfrak{b}}(e_{1}))\\
=& \delta^{\rho}(e_{1})+ \partial_{cy, \mathfrak{b}}^{\rho}(e_{1})=d^{\rho}_{cy, \mathfrak{b}}(e_{1}).
\end{aligned}
\end{equation}
The third last equality uses the divisor axiom, see (\ref{divisor axiom}). For other partial derivatives, similar computation also works. Therefore if all the partial derivatives of $\mathfrak{PO}^{cy, \mathfrak{b}}$ vanishes then our deformed Floer boundary operators vanishes on $H^{1}(L; \Lambda_{+})$. Since $L$ is a torus of which the cohomology is generated by degree one elements, we can perform an induction to show that the deformed Floer boundary operator vanishes on the whole $H^{*}(L; \Lambda_{+})$. We refer to Section 13 in \cite{FOOO1} for the induction process and the extension from $\rho\in H^{1}(L; \Lambda_{+})$ to $\rho\in H^{1}(L; \Lambda_{0})$.
\end{proof}

\begin{remark}\label{div}
We have completed the construction of a deformed $A_{\infty}$-structure by using a chain $K$. Next we will perform computations to get some geometric applications. We list here what we expect to be true but didn't spell out in full details.
\begin{enumerate}
\item We need to construct a certain chain model of $C_{*}(L)$ to serve as the domain and codomain of our operators. We expect the construction of a singular chain model in \cite{FOOO} or a de Rham chain model in \cite{FOOO2} work here.
\item For sign conventions of our operators, we expect the calculation in Chapter 8 of \cite{FOOO} work here.
\item For transversality of moduli spaces, we expect that we can use either a single $J$ via virtual perturbation or use families of $J$ via classical means.
\item Along with the transversality of moduli spaces, we expect that our operators enjoy the divisor axioms, similar to (\ref{divisor axiom}).
\end{enumerate}
\end{remark}

\subsection{Computations for a local torus}
In last subsection we defined a version of deformed Floer theory in a general setting. Now we study the case where $L$ is a Lagrangian 3-torus near a Lagrangian 3-sphere.

Let $X$ be a symplectic 6-manifold with an integrally homologically trivial Lagrangian sphere $S$. We fix a Weinstein neighborhood $U$ of $S$ such that there is a singular toric fibration on $U$, as we described in Subsection 4.1. Topologically $U$ is isomorphic to $S^{3}\times B^{3}$ where $B^{3}$ is a 3-ball. We say $U$ is a round neighborhood if it is symplectomorphic to $D_{r}T^{*}S^{3}$ with respect to the round metric on the unit sphere $S^{3}$. The preimages of four faces in the moment polytope are four 4-chains $K_{i}=\pi^{-1}(P_{i}), i=1, 2, 3, 4$. Each of them is homeomorphic to $S^{3}\times [0, 1]$ with two boundary components. Up to orientation $\partial_{0}(K_{i})$ is the zero section $S$ and $\partial_{1}(K_{i})$ is the generator of $H_{3}(\partial U; \mathbb{Z})$. First we study some topological condition on $S$ to perform the conifold transition. Let $V$ be a small closed neighborhood containing $X-U$ such that $U\cap V$ is homeomorphic to $S^{3}\times S^{2}\times [1-\epsilon, 1]$.

\begin{lemma}
The Lagrangian sphere $S$ is homologically trivial in $X$ if and only if the inclusion
$$
j_{2}: H_{3}(U\cap V; \mathbb{Z})\rightarrow H_{3}(V; \mathbb{Z})
$$
is trivial.
\end{lemma}
\begin{proof}
Note that $U\cap V$ is homeomorphic to $S^{3}\times S^{2}\times [1-\epsilon, 1]$. Consider the Mayer-Vietoris sequence
$$
\cdots \rightarrow H_{3}(U\cap V)\rightarrow H_{3}(U)\oplus H_{3}(V)\rightarrow H_{3}(X)\rightarrow H_{2}(U\cap V)\rightarrow \cdots.
$$
And we write $j_{1}, j_{2}, i_{1}, i_{2}$ as the inclusion maps
$$
j_{1}: H_{3}(U\cap V)\rightarrow H_{3}(U), j_{2}: H_{3}(U\cap V)\rightarrow H_{3}(V), i_{1}: H_{3}(U)\rightarrow H_{3}(X), i_{2}: H_{3}(V)\rightarrow H_{3}(X).
$$
Then we have that $j_{1}$ is an isomorphism hence $H_{4}(U, U\cap V)=\lbrace 0\rbrace$. Note that $H_{4}(X, V)=H_{4}(U, U\cap V)=\lbrace 0\rbrace$ so $i_{2}$ is an injection. By the exactness of the Mayer-Vietoris sequence, we get that $i_{1}\circ j_{1}-i_{2}\circ j_{2}=0$. In conclusion $i_{1}$ is zero if and only if $j_{2}=0$.
\end{proof}

By the above lemma all the four $K_{i}$'s can be completed into four 4-chains in $X$. In other words the boundary $\partial_{1}(K_{i})$ can be capped in $X-U$. From now on we only consider those ``completed'' chains and still write them as $K_{i}$. So $K_{i}$'s are 4-chains in $X$ such that $\partial(K_{i})=\pm S$ and $K_{i}\cap U$ is the preimage of $P_{i}$. We remark that there may be different choices of $K$. For example any chain $K+A$ for $A\in H_{4}(X)$ is another choice of a chain with boundary $S$. Now we fix a ``completion'' for each $K_{i}$ and regard them as 4-chains in $X$.

Now for this fixed Weinstein neighborhood $U$ and a local torus $L\subset U$, we estimate some energy parameters of holomorphic disks. Let $J$ be a compatible almost complex structure on $X$ which satisfies Condition \ref{condition+}. Then the one-pointed open Gromov-Witten invariant $n_{\beta}$ is defined with respect to $J$, for a disk class $\beta\in\pi_{2}(X, L)$ with Maslov index two and with energy less than $E_{+}$. We consider the sequence
$$
\lbrace \beta_{k}\mid n_{\beta}\neq 0, E(\beta_{k})\leq E(\beta_{k+1})\rbrace_{k=1}^{\infty}
$$
of disk classes with Maslov index two, enumerated by their symplectic energy. We know that $L$ bounds four $J$-holomorphic disks with Maslov index two inside $U$, with same energy $E_{1}$. Those are the first four elements in the above sequence if $L$ is near $S$. In particular we can assume that $E_{1}<<E_{S}$. Let $E_{5}=E(\beta_{5})$ be the least energy of outside disk contribution. Moreover, we assume that $2E_{5}<E_{+}$.

In order to directly compute this deformed potential function, we need to use the special properties of $J$. We know that $L$ bounds four families of $J$-holomorphic disks in $U$, with Maslov index two, representing four classes $\beta_{1}, \beta_{2}, \beta_{3}, \beta_{4}$. All of this four families of disks are away from a small neighborhood of $S$, by $(6)$ in Condition \ref{condition+}. Hence the moduli space $\mathcal{M}_{1, 1}(\beta_{i}; K)$ is a closed manifold, since the interior marked point cannot go to $\partial K$. Moreover, the moduli space of cylinders $\mathcal{M}_{1}^{cy}(\beta; S)$, representing a Maslov two class $\beta$ with $\omega(\beta)<E_{5}$, does not have codimension one boundary. This is because no $J$-holomorphic disk touches $S$ if it has Maslov index two and has energy less than $E_{5}$.

Therefore, the moduli space $\mathcal{M}_{1, 1}(\beta; K+S)$ is a \textit{disjoint union} of two closed moduli spaces, see Corollary \ref{union}. By definition, the coefficient $\mathfrak{q}_{1, 0;\beta}^{cy, \mathfrak{b}}(1)$ in the potential function is the mapping degree
$$
\deg( ev: \mathcal{M}_{1, 1}(\beta; K+S)\rightarrow L) ).
$$
Hence it is the sum of two mapping degrees
\begin{equation}
\begin{aligned}
\mathfrak{q}_{1, 0;\beta}^{cy, \mathfrak{b}}(1)&= \deg( ev: \mathcal{M}_{1, 1}(\beta; K+S)\rightarrow L) )\\
&= \deg( ev: \mathcal{M}_{1, 1}(\beta; K)\rightarrow L) )+ \deg( ev: \mathcal{M}_{1, 1}^{cy}(\beta; S)\rightarrow L) ).
\end{aligned}
\end{equation}
We write
$$
n_{\beta}(K):= \deg( ev: \mathcal{M}_{1, 1}(\beta; K)\rightarrow L) )
$$
and
$$
n_{\beta}^{cy}:= \deg( ev: \mathcal{M}_{1, 1}^{cy}(\beta; S)\rightarrow L) ).
$$
By local study, the first mapping degree is known to be one if and only if $\beta=\beta_{i}$ and $K$ is a capped $K_{i}$.

Now we can compute the deformed potential function. By the degree computation it is enough to only consider $\beta$ with Maslov index two. Let $\mathfrak{b}=wK_{1}$, where $K_{1}$ is the capped first facet in the singular toric fibration. With respect to a chosen basis of $H^{1}(L; \mathbb{Z})$ (the same basis as in (4.1)), the potential function is
\begin{equation}\label{po}
\begin{aligned}
&\mathfrak{PO}^{cy, \mathfrak{b}}(\rho)\\
=&[(1+(1+n_{\beta_{1}}^{cy})w)x+ (1+n_{\beta_{2}}^{cy}w)y^{-1}+ (1+n_{\beta_{3}}^{cy}w)xz^{-1}+ (1+n_{\beta_{4}}^{cy}w)y^{-1}z]T^{E_{1}}+\\
+& \sum_{\mu(\beta)=2, \omega(\beta)=E_{1}, \beta\neq\beta_{i}}n_{\beta}^{cy}f_{\beta}(x,y,z)T^{E_{1}}+ H(w, x, y, z, T)
\end{aligned}
\end{equation}
where $H(w, x, y, z, T) $ are higher energy terms.

We explain (\ref{po}) as follows. There is no term with energy less than $E_{1}$, since such a low energy disk or cylinder will be contained in $U$, where $L$ is monotone and $E_{1}$ is the minimal energy. The term $(1+(1+n_{\beta_{1}}^{cy})w)x$ corresponds to the class $\beta_{1}$. Its coefficient is
$$
\mathfrak{m}_{0; \beta_{1}}(1)+ \mathfrak{q}_{1, 0;\beta_{1}}^{cy, \mathfrak{b}}(1)= \mathfrak{m}_{0; \beta_{1}}(1)+ (n_{\beta_{1}}(K_{1})+ n_{\beta_{1}}^{cy})w = 1+(1+n_{\beta_{1}}^{cy})w,
$$
weighted by the boundary class $e^{\rho(\partial\beta_{1})}=x$ in our chosen basis. Similarly, for $\beta_{i}$ with $i=2,3,4$ we have other three terms, since holomorphic disks representing $\beta_{i}$ with $i=2,3,4$ do not intersect $K_{1}$, we have $\mathfrak{m}_{0; \beta_{i}}(1)=1, n_{\beta_{1}}(K_{1})=0$ but $n_{\beta_{i}}^{cy}$ might not be zero. This gives the coefficient $1+n_{\beta_{i}}^{cy}w$ in these three terms.

And the term
$$
\sum_{\mu(\beta)=2, \omega(\beta)=E_{1}, \beta\neq\beta_{i}}n_{\beta}^{cy}f_{\beta}(x,y,z)T^{E_{1}}
$$
corresponds to other disk classes with $\mu(\beta)=2, \omega(\beta)=E_{1}, \beta\neq\beta_{i}$. By the local study we know that $\mathfrak{m}_{0; \beta}(1)=n_{\beta}(K_{1})=0$. However, there might be a nonzero cylinder contribution $n_{\beta}^{cy}$, which we weight by its energy $T^{E_{1}}$ and boundary class $f_{\beta}(x,y,z):=e^{\rho(\partial\beta)}$. Later we will use $(8)$ in Condition \ref{condition+} to exclude those contributions. The last term is from contributions of classes with energy more than $E_{1}$.

In the low energy terms, the disk contributions are explicitly known but the cylinder contributions $n_{\beta}^{cy}$ are not. Next we show that for a degenerate almost complex structure $J_{\infty}$, there are \textit{no} cylinder contributions. However, we do not know while we degenerating the almost complex structure, whether other assumptions in Condition \ref{condition} will be preserved. So when we use the vanishing of some cylinder contributions, we regard it as an extra condition.

\begin{proposition}
With respect to some almost complex structure $J_{\infty}$, the moduli space $\mathcal{M}_{1, 1}^{cy}(\beta, S)$ is empty. Here $\beta$ is any class with Maslov index two.
\end{proposition}
\begin{proof}
The proof uses the degeneration technique in Section 3. First we fix a Weinstein neighborhood $U$ of $S$ and $L$ is in $U$. Then we chose a smaller Weinstein neighborhood $U'\subset U$ of $S$ such that $L$ is not in $U'$. That is, $U'$ is symplectomorphic to $D_{r'}T^{*}S^{3}$ with the canonical symplectic form with a smaller $r'$. Then by the symplectic sum construction, we have a fibration $\mathcal{X}\rightarrow \Delta$ over the unit disk. The fiber $X_{z}$ is symplectomorphic to $X$ if $z \neq 0$. The fiber over zero is a singular symplectic manifold $X_{0}=X_{-}\cup_{D}X_{+}$ where $X_{-}$ is symplectomorphic to the quadric with a scaled symplectic form. The symplectic manifold $X_{+}$ is obtained from $X-U'$ by quotienting the $S^{1}$-action on $\partial U'$. Hence $L$ is in $X_{+}$.

Next we do dimension counting to show that there is no holomorphic cylinder in the fiber $X_{z}$ if $z$ is close to zero. We assume that there is a compatible almost complex structure $J$ on $\mathcal{X}$ such that
\begin{equation}\label{degenerate}
\textit{the torus $L$ satisfies Condition \ref{condition} in $X_{+}$ with respect to $J\mid_{X_{+}}$.}
\end{equation}
For a class $\beta$, suppose that we have a sequence $u_{z}$ of $J_{z}$-holomorphic cylinders in $X_{z}$ representing class $\beta$, with $z\rightarrow 0$. By Gromov compactness theorem, we have a limit $u_{0}$ in $X_{0}$. Suppose that for the nodal curve $u_{0}$, its component in the quadric, which is not contained in $D$ is of class $A$. Since one end of our cylinder is on $S$, the class $A$ is non-zero. The component of $u_{0}$ in $X_{+}$ has two sub-components: one is totally contained in $D$ and the other not in $D$. We write the class of the former as $B\in H_{2}(D)\subset H_{2}(X_{+})$ and the later as class $\beta'\in H_{2}(X_{+}, L)$. We assume that both components $A$ and $\beta'$ intersect $D$ at finitely many points. If $A\neq 0$ we write $s_{1}$ as the intersection number of $A$ and the divisor $D=Q_{2}\subset Q_{3}$. Note that $D\cong \mathbb{C}P^{1}\times \mathbb{C}P^{1}$ is both a complex and symplectic submanifold of $X_{+}$, it is also monotone with respect to the induced symplectic form. For each ruling $C=\mathbb{C}P^{1}$ in $D$, its intersection number (computed in $X_{+}$) with $D$ is $-1$, since the normal bundle of $D$ in $X_{+}$ is $\mathcal{O}(-1, -1)$. Hence $c_{1}(N_{X_{+}}C)(C)=-1$ and $c_{1}(TX_{+})(C)=1$. Then we write the intersection number $B\cdot D= s_{11}$, which is a negative number. And we write $\beta'\cdot D=s_{12}$. By above discussion, we know that $s_{11}+ s_{12}= s_{1}$ and $c_{1}(TX_{+})(B)=-s_{11}$ since $B$ is a positive linear combination of two rulings.

Now we use the Maslov index formula, which can be deduced from (\ref{index})
$$
\begin{aligned}
\mu(\beta)&= \mu(A)+ \mu(B)+ \mu(\beta')- 4s_{1}= 6s_{1} -2s_{11} -4s_{1} +\mu(\beta')\\
&\geq 2s_{1} -2s_{11} +2s_{12}= 4s_{12}\geq 4.
\end{aligned}
$$
The second equality uses that $c_{1}(TQ_{3})=3[Q_{2}]$ and the last inequality uses assumption (\ref{degenerate}) so that $\mu(\beta')\geq 2s_{12}$.

Therefore if the class $\beta$ has Maslov index two then the nodal curve cannot have a component in the top level $Q_{3}$. That is, for a Maslov index two class $\beta$, there is no $J_{z}$-holomorphic cylinder representing $\beta$ with respect to some $J_{z}$ when $z$ is small. Hence we write $J_{\infty}$ as one of $J_{z}$ for small $z$.
\end{proof}

\begin{remark}
A similar proof, replacing a sequence of holomorphic cylinders by a sequence $u_{z}$ of $J_{z}$-holomorphic disks in $X_{z}$ with boundary on $L$ representing class $\beta$, shows that there is no $J_{z}$-holomorphic disk going to the neighborhood $U'$, representing a Maslov two $\beta$ with respect to some $J_{z}$ when $z$ is small. Hence the assumption (\ref{degenerate}) serves as an alternative of Condition \ref{condition+} for our later computations. We expect (\ref{degenerate}) is true at least locally. For example, when $X=T^{*}S^{3}$ if we choose $U'$ in the above degeneration carefully, our fixed $L$ is still monotone in $X_{+}$. So we can first do local computation to exclude cylinder contributions with respect to some $J_{\infty}$, then try to extend $J_{\infty}$ from a Weinstein neighborhood to the whole of $X$.

One may also use neck-stretching in symplectic field theory to prove above results. Note that there is a contact form on the boundary of $U'$ such that the Conley-Zehnder indices of its Reeb orbits are at least two. Starting with a disk or cylinder with Maslov index two, after stretching along $\partial U'$, in the bottom level we have a curve with several negative punctures and with Maslov index at most two. The dimension of this moduli space is smaller than two, hence the boundary of such a curve does not pass through a generic point on $L$.
\end{remark}

In the next section we will relate the cohomology $HF_{cy}(L; (\mathfrak{b}, b))$ to another model of cohomology such that the underlying complex is generated by Hamiltonian chords with ends on $L$ and a Hamiltonian perturbation $\phi(L)$. The first cohomology $HF_{cy}(L; (\mathfrak{b}, b))$ is for computational purpose and the later cohomology is more geometrical. Once we established the equivalence between these two theories we get a critical points theory to detect the displacement energy of $L$.

\section{A second deformed Floer complex}
Now we will construct another deformed Floer complex and study its change of filtration under Hamiltonian diffeomorphisms. As in the setting of Theorem \ref{construction}, let $X$ be a symplectic 6-manifold, $S$ be an integrally homologically trivial Lagrangian sphere and $L$ be a Lagrangian submanifold of $X$ satisfying Condition \ref{condition}. We also fix a 4-chain $K$ in $X$ such that $\partial K=S$ and $L\cap K=\emptyset$, assuming its existence.

In the following construction we need to introduce Hamiltonian perturbations of $L$ and perturbations of the chosen almost complex structure $J$. We first remark that by Gromov compactness theorem, see Lemma 6.4.7 and Lemma 6.4.8 in \cite{MS}, both Condition \ref{condition} and Condition \ref{condition+} are open, except item $(2)$. However there can be more than one connected components of the set of almost complex structures satisfying each condition. So when we say a small perturbation of $J$, we mean another almost complex structure in the same component with $J$, which satisfies the conditions. When we say a small Hamiltonian perturbation $\phi$ of $L$, we mean the Hamiltonian is small such that $\phi(L)$ also satisfies these conditions with respect to another $J'$ which is in the same component with $J$ and satisfies the conditions. Then we can connect $J$ with $J'$ by a family of almost complex structures in this component. And when we use domain-dependent almost complex structures $\lbrace J_{z}\rbrace_{z\in \Sigma}$, we assume that all $J_{z}$ are in the same component. As we mentioned in the beginning of this article, we expect the transversality result needed can be obtained by either a single $J$ via virtual perturbation or by families of $J$'s via classical means.

\subsection{Definition of the complex}
Let $H_{t}$ be a time-dependent Hamiltonian function on $X$ and let $\phi$ be its time-one Hamiltonian diffeomorphism. We first review the Floer complex generated by the Hamiltonian chords with ends on $L$, which is called the \textit{dynamical} version of Floer theory in \cite{FOOO4}.

Consider the path space
$$
\Omega(L)=\lbrace l:[0, 1]\rightarrow X\mid l(0)\in L, l(1)\in L\rbrace.
$$
We fix a base path $l_{a}\in \Omega(L)$ for each component $a\in \pi_{0}(\Omega(L))$. Let $(l, w)$ be a pair such that $l\in \Omega(L)$ and $w: [0, 1]^{2}\rightarrow X$ satisfying
$$
w(s, 0)\in L, w(s, 1)\in L, w(0, t)=l_{a}(t), w(1, t)=l(t).
$$
And we modulo the following $\Gamma$-equivalence. Let $(l,w), (l,w')$ be two such pairs, the concatenation
$$
\bar{w}\# w': [0,1]\times [0,1]\rightarrow X
$$
defines a loop $c:S^{1}\rightarrow \Omega(L)$. One may regard this loop as a map $C:S^{1}\times [0,1]\rightarrow X$ satisfying the boundary condition $C(s,0)\in L, C(s,1)\in L$. Next we consider the symplectic area $I_{\omega}(c):=\int_{C}\omega$ and the Maslov index $I_{\mu}(c):=\mu(C)$. We say two pairs $(l,w), (l,w')$ are $\Gamma$-equivalent if $I_{\omega}(\bar{w}\# w')=0=I_{\mu}(\bar{w}\# w')$. We denote the set of equivalence classes $[l,w]$ by $\tilde{\Omega}(L)$, which is called the Novikov covering space.

Then we define the dynamical action functional, with respect to $H_{t}$, to be
\begin{equation}
\mathcal{A}_{H_{t}, l_{a}}([l, w])=\int w^{*}\omega +\int_{0}^{1} H_{t}(l(t))dt.
\end{equation}
on the space of pairs $[l, w]$. We remark that our convention for Hamilton's equation is
$$
\omega(X_{H}, \cdot)= dH(\cdot ), \quad \textit{where $X_{H}$ is the Hamiltonian vector field of $H$}.
$$
The critical points of this action functional are Hamiltonian chords. We write the set of critical points as
$$
CF(L, H_{t})=\lbrace [l,w]\mid l'(t)=X_{H_{t}}(l(t))\rbrace.
$$
For a critical point $[l,w]$ the path $l$ corresponds to a geometric intersection point in $L\cap \phi(L)$ since $\phi(l(0))=l(1)\in L$. When $H_{t}$ is generic there are only finitely many of them, and they do not intersect the 4-chain $K$. We remark that the set of critical points has a decomposition with respect to the different components $a\in \pi_{0}(\Omega(L))$. We define the action functionals and study their critical points on different components separately.

Now we equip $L$ with local systems. For any group homomorphism
$$
\rho: \pi_{1}(L)\rightarrow \Lambda_{0}-\Lambda_{+}
$$
we choose a flat $\Lambda_{0}$-bundle $(\mathcal{L}, \nabla_{\rho})$ such that its holonomy representation is $\rho$. Then we define the cochain complex as
\begin{equation}
\begin{aligned}
&\widetilde{CF}((L, \rho), H_{t}; \Lambda_{0}):= \bigoplus_{[l,w]\in CF(L, H_{t})} \hom (\mathcal{L}_{l(0)}, \mathcal{L}_{l(1)})\cdot [l,w]\\
& CF((L, \rho), H_{t}; \Lambda_{0}):= \widetilde{CF}((L, \rho), H_{t}; \Lambda_{0})/\sim.
\end{aligned}
\end{equation}
Here $\mathcal{L}_{l(i)}$ is the fiber of the bundle $\mathcal{L}$ over $l(i)$ and $\hom (\mathcal{L}_{l(0)}, \mathcal{L}_{l(1)})$ is defined as
$$
\hom(\mathcal{L}_{l(0)}, \mathcal{L}_{l(1)}):= \hom(\phi_{*}(\mathcal{L}_{l(0)}), \mathcal{L}_{l(1)}),
$$
which is a homomorphism induced by $\phi$. And the notion $\hom (\mathcal{L}_{l(0)}, \mathcal{L}_{l(1)})\cdot [l,w]$ means the free $\Lambda_{0}$-module generated by $[l,w]$. The equivalence $\sim$ here is that for $T^{A}\in \Lambda_{0}$, we define $[l,w]\sim T^{A} [l',w']$ if and only if
$$
l=l', \quad \int_{w}\omega = A +\int_{w'}\omega.
$$
Then we extend the above action functional to $CF((L, \rho), H_{t}; \Lambda_{0})$ as follows
$$
\mathcal{A}_{H_{t}, l_{a}}(\sigma[l, w])= v(\sigma)+ \int w^{*}\omega +\int_{0}^{1} H_{t}(l(t))dt.
$$

Next we consider smooth maps
$$
u(\tau, t): \mathbb{R}\times [0,1]\rightarrow X, \quad u(\tau, 0)\in L, \quad u(\tau, 1)\in L
$$
such that $u(-\infty, t)=l_{0}(t), u(\infty, t)=l_{1}(t)$ for some $l_{0}, l_{1}$ to define the parallel transport maps. Let $B$ be the homotopy class of $u$ and $\sigma\in \hom (\mathcal{L}_{l_{0}(0)}, \mathcal{L}_{l_{0}(1)})$ then we define
$$
Comp_{B}: \hom (\mathcal{L}_{l_{0}(0)}, \mathcal{L}_{l_{0}(1)})\rightarrow \hom (\mathcal{L}_{l_{1}(0)}, \mathcal{L}_{l_{1}(1)})
$$
by
\begin{equation}
Comp_{B}(\sigma)=Pal_{1}\circ \sigma \circ Pal_{0}^{-1}
\end{equation}
where $Pal_{i}$ is the parallel transport along the path $u(\tau, i)\in L$ for $i=0, 1$, see Figure 5. Although we used $u$ to define the parallel transport, the composition map is a homotopy invariant. That is, it only depends on the homotopy class $B$ of $u$.

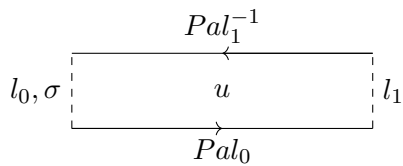
\begin{figure}
  \begin{tikzpicture}
  \draw (-2,1) to (2,1);
  \draw [->] (-2,1) to (0,1);
  \draw (-2,0) to (2,0);
  \draw [->] (2,0) to (0,0);
  \draw [dashed] (-2,0) to (-2,1);
  \draw [dashed] (2,0) to (2,1);
  \node [left] at (-2,0.5) {$l_{0}, \sigma$};
  \node [right] at (2,0.5) {$l_{1}$};
  \node [above] at (0,1) {$Pal_{1}$};
  \node [below] at (0,0) {$Pal_{0}^{-1}$};
  \node at (0,0.5) {$u$};
  \end{tikzpicture}
  \caption{Composition of parallel transport maps.}
\end{figure}

Now we can define the Floer coboundary operator with local systems. Let
$$
\begin{aligned}
\mathcal{M}([l_{0},w_{0}], [l_{1},w_{1}])= &\lbrace u(\tau, t): \mathbb{R}\times [0,1]\rightarrow X \mid \partial_{\tau}u+ J(\partial_{t}u- X_{H_{t}})=0,\\
& u(\tau, 0)\in L, u(\tau, 1)\in L, u(-\infty, t)=l_{0}(t), u(\infty, t)=l_{1}(t)\rbrace
\end{aligned}
$$
be the moduli space of Floer solutions connecting $[l_{0},w_{0}]$ and $[l_{1},w_{1}]$. Then for a fixed $\rho$ we define
$$
\delta^{\rho}: CF((L, \rho), H_{t}; \Lambda_{0})\rightarrow CF((L, \rho), H_{t}; \Lambda_{0})
$$
as
\begin{equation}
\begin{aligned}
& \delta^{\rho}(\sigma [l_{0},w_{0}])\\
=& \sum_{[l_{1},w_{1}]} Comp_{[w_{1}-w_{0}]}(\sigma) \sharp \mathcal{M}([l_{0},w_{0}], [l_{1},w_{1}])[l_{1},w_{1}].
\end{aligned}
\end{equation}
Here the sum is over all $[l_{1},w_{1}]$ such that the corresponding moduli space is zero-dimensional. And the number $ \sharp \mathcal{M}([l_{0},w_{0}], [l_{1},w_{1}])$ is a signed count.

\begin{proposition}
Under Condition \ref{condition}, the coboundary operator is well-defined and satisfies that $(\delta^{\rho})^{2}\equiv 0 \mod T^{E_{+}}$.
\end{proposition}
\begin{proof}
Under the energy bound $E_{+}$, the proof is similar to the case when a Lagrangian torus is monotone, where the self-Floer cohomology is well-defined, see Theorem 16.4.10 in \cite{O2}. Note that Condition \ref{condition} excludes possible disk bubbles, with non-positive Maslov indices, splitting from the Floer strips. For disk bubbles with Maslov index two, they appear in pairs on both sides of the Floer solutions and cancel with each other. We don't need to consider disk bubbles with higher Maslov indices since we are looking at one-dimensional moduli spaces to show the square of $\delta^{\rho}$ is zero. And so far we haven't introduce interior marked points, we don't need to use the energy bound $E_{S}$ in Condition \ref{condition}.
\end{proof}

We call the above cohomology given by $\delta^{\rho}$ the Floer cohomology with local systems. Next we want to deform it further by counting strips with an interior marked point/an interior hole. The aim is to define a new operator
$$
\partial_{K}: CF((L, \rho), H_{t}; \Lambda_{0})\rightarrow CF((L, \rho), H_{t}; \Lambda_{0}).
$$
Here $K$ is the chosen 4-chain with $\partial K=S$ and $K\cap L=\emptyset$. First we describe the domain we will use to count holomorphic maps. Consider the domain
$$
Strip_{\epsilon, r=(r', r'')}=\lbrace (\tau, t)\in \mathbb{R}\times [0, 1]\subset \mathbb{C}\mid (\tau -r')^{2}+ (t- r'')^{2}\geq \epsilon ^{2}\rbrace.
$$
Let $C(\epsilon)$ denote the circle boundary $(\tau -r')^{2}+ (t- r'')^{2}= \epsilon ^{2}$ of $Strip_{\epsilon, r}$. We put the interior hole centered at $(r', r'')$ with radius $\epsilon\in (0, \min\lbrace r'', 1-r''\rbrace)$. The radius $\epsilon$ determines the complex structure on the domain. And we write $Strip=Strip_{0, r}$ as the usual holomorphic strip in $\mathbb{C}$.

Now we consider several moduli spaces. The elements in these moduli spaces satisfy the same Hamiltonian-perturbed holomorphic equation
$$
\partial_{\tau}u+ J(\partial_{t}u- X_{H_{t}})=0.
$$
But they are from different domains and have different boundary conditions.

For a pair $([l_{0},w_{0}], [l_{1},w_{1}])$ let
$$
\widetilde{\mathcal{M}}_{1}(([l_{0},w_{0}], [l_{1},w_{1}]); K)
$$
be the moduli space of Floer strips with one interior marked point at $(r', r'')$, where the interior point is mapped to $K$. More precisely, it contains maps $u: Strip \rightarrow X$ such that
$$
u(\tau, 0)\in L, \quad u(\tau, 1)\in L, \quad u(-\infty, t)=l_{0}, \quad u(\infty, t)=l_{1}
$$
and
$$
u(r', r'')\in K
$$
where the map $u$ represents the class $\beta=w_{1}-w_{0}$. And let
$$
\widetilde{\mathcal{M}}_{1}^{cy}(([l_{0},w_{0}], [l_{1},w_{1}]); S)
$$
be the moduli space of Floer strips with one interior hole, where the hole is mapped to $S$. It contains maps from domain $Strip_{\epsilon, r}$ for all $(\epsilon, r)$. And $u$ satisfies the same Lagrangian boundary condition as above: the line boundaries are mapped to $L$ and two ends converge to given chords $l_{0}, l_{1}$. One extra boundary condition is that the circle boundary is mapped to $S$.

Note that $\partial K=S$ and $S$ is simply connected, the homotopy classes in these two types of moduli spaces can be identified. Now we fix a bulk deformation $\mathfrak{b}=wK$ with $w\in\Lambda_{+}$. Similar to the discussion in Section 4 we want to compactify these moduli spaces and glue them together along a common boundary for the same class $\beta=w_{1}-w_{0}$, with $\omega(\beta)+v(w)< E:=\min\lbrace E_{S}+ v(w), 2v(w), E_{+}\rbrace$.

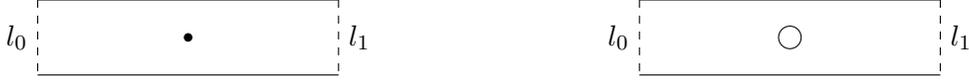
\begin{figure}
  \begin{tikzpicture}
  \draw (-6,1) to  (-2,1);
  \draw (-6,0) to  (-2,0);
  \draw [dashed] (-6,1) to (-6,0);
  \draw [dashed] (-2,1) to (-2,0);
  \draw [fill] (-4,0.5) circle [radius=0.05];
  \node [left] at (-6,0.5) {$l_{0}$};
  \node [right] at (-2,0.5) {$l_{1}$};
  \draw (2,1) to  (6,1);
  \draw (2,0) to  (6,0);
  \draw [dashed] (6,1) to (6,0);
  \draw [dashed] (2,1) to (2,0);
  \draw (4,0.5) circle [radius=0.15];
  \node [left] at (2,0.5) {$l_{0}$};
  \node [right] at (6,0.5) {$l_{1}$};
  \end{tikzpicture}
  \caption{Counting strips with one interior marked point and one hole.}
\end{figure}

\begin{proposition}\label{boundary}
For fixed generators $[l_{0},w_{0}]$ and $[l_{1},w_{1}]$, there are compactification
$$
\mathcal{M}_{1}(([l_{0},w_{0}], [l_{1},w_{1}]); K)\supseteq \widetilde{\mathcal{M}}_{1}(([l_{0},w_{0}], [l_{1},w_{1}]); K)
$$
and compactification
$$
\mathcal{M}_{1}^{cy}(([l_{0},w_{0}], [l_{1},w_{1}]); S)\supseteq \widetilde{\mathcal{M}}_{1}^{cy}(([l_{0},w_{0}], [l_{1},w_{1}]); S).
$$
Each of them has a particular boundary component such that
$$
\partial_{K} \mathcal{M}_{1}(([l_{0},w_{0}], [l_{1},w_{1}]); K)= -\partial_{cy} \mathcal{M}_{1}^{cy}(([l_{0},w_{0}], [l_{1},w_{1}]); S)
$$
and we can glue them on this component to get a compact moduli space
$$
\begin{aligned}
&\mathcal{M}_{1}(([l_{0},w_{0}], [l_{1},w_{1}]); K+S))=\\
&\mathcal{M}_{1}(([l_{0},w_{0}], [l_{1},w_{1}]); K)\sqcup\mathcal{M}_{1}^{cy}(([l_{0},w_{0}], [l_{1},w_{1}]); S)\big/ \sim
\end{aligned}
$$
where the equivalence relation is
$$
\partial_{K} \mathcal{M}_{1}(([l_{0},w_{0}], [l_{1},w_{1}]); K)\sim -\partial_{cy} \mathcal{M}_{1}^{cy}(([l_{0},w_{0}], [l_{1},w_{1}]); S).
$$
\end{proposition}
\begin{proof}
To get the compactification we add several types of degenerations: strip breaking, disk/sphere bubbles and domain degeneration. The cases of the strip breaking and disk/sphere bubbles are more standard in Floer theory. So we mainly care about the domain degeneration involving parameters $\epsilon$ and $r$. The former is the radius of the interior hole and the later is the position of the center of the hole. Note that $r=(r', r'')$, if $r'$ goes to positive or negative infinity, then we will have an interior marked point or an interior hole converging to a chord, which is assume to be away from the chain $K$. Hence such a degeneration does not happen. In the following we focus on the parameters $\epsilon, r''$ and just write $r=r''$ for notational simplicity. Suppose that we have a sequence of parameters $\lbrace(\epsilon_{i}, r_{i})\rbrace_{i=1}^{+\infty}$, we will discuss case by case of possible degenerations.
\begin{enumerate}
\item If $\inf_{i}\lbrace \epsilon_{i}\rbrace>0$ and $\epsilon_{i}+ r_{i}\rightarrow 1$ or $-\epsilon_{i}+ r_{i}\rightarrow 0$. Geometrically the circle boundary approaches to the strip boundary while the radius of the circle is bounded from below. This type of degeneration will not happen since our $S$ and $L$ are disjoint. Without losing generality we assume that $\epsilon_{i}+ r_{i}\rightarrow 1$ with $\epsilon_{i}\equiv \epsilon_{0}> 0$ for some constant $\epsilon_{0}$. Then we can scale a neighborhood of the point $(0, \epsilon_{i}+ r_{i})$ such that locally we have a holomorphic strip $u_{i}$ with one boundary on $L$ and with one curved boundary on $S$, see Figure 7. To compactify such a degeneration we need to add a genuine holomorphic strip $u_{\infty}$ in the moduli space, since in the limit the curved boundary becomes a usual boundary. However, note that such a strip $u_{\infty}$ has finite energy because it splits from a finite energy solution. By an exponential decay estimate we know $\lim_{\tau\rightarrow \pm\infty} u_{\infty}(\tau, t)$ converges to the intersections of $L$ and $S$, which is empty by our assumption. Hence such a degeneration will not appear.
\item If $\epsilon_{i}\rightarrow 0$ and $\lbrace r_{i}\rbrace$ stays the interior of the strip. In the limit we have a holomorphic strip with one interior marked point. Then we can perform the same gluing argument as we did in Section 4. That is, we glue this end with the moduli space of strips with one interior marked point as we did before, to cancel this end of boundary.
\item If $\epsilon_{i}\rightarrow 0$ and $\lbrace r_{i}\rbrace$ goes to one strip boundary. Without losing generality we assume that $\lim_{i}(r_{i})=1$. Then we consider the ratio $\frac{\epsilon_{i}}{1-r_{i}}$ and there are different possibilities.
\begin{enumerate}
\item If $\lim_{i}\frac{\epsilon_{i}}{1-r_{i}}=+\infty$, the case is similar to (1) and we use the fact $L\cap S=\emptyset$ to exclude this degeneration.
\item If $\lim_{i}\frac{\epsilon_{i}}{1-r_{i}}=R$ for some constant $R>0$, after a conformal change this degeneration is equivalent to an annulus bubble on the boundary. So we put this type of limit of solutions into the compactification.
\item If $\lim_{i}\frac{\epsilon_{i}}{1-r_{i}}=R=0$, then after a conformal change it is a disk bubble, with one interior point attaching to $S$. We put this type of limit of solutions into the compactification.
\end{enumerate}
\end{enumerate}

In conclusion, to get the compactification we add broken curves in (2), (3b), (3c) and broken strips. Next we glue the particular boundary component in (2) with the moduli space of holomorphic strips with one interior marked point, as we did in Theorem \ref{compactification}.

We write
$$
\partial_{K} \mathcal{M}_{1}(([l_{0},w_{0}], [l_{1},w_{1}]); K)
$$
as the boundary component containing elements when the interior marked point is mapped to $S=\partial K$. And we write
$$
\partial_{cy} \mathcal{M}_{1}^{cy}(([l_{0},w_{0}], [l_{1},w_{1}]); S)
$$
as the boundary component containing elements in (2). These two boundary components are the same since they contain the same set of elements. Then we glue these two compactified moduli spaces along this common boundary component.
\end{proof}

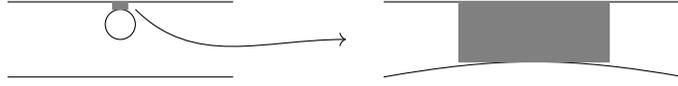
\begin{figure}
  \begin{tikzpicture}
  \draw (0,0)--(3,0);
  \draw (0,1)--(3,1);
  \draw (1.5,0.7) circle [radius=0.2];
  \filldraw [gray] (1.4,0.9)--(1.6,0.9)--(1.6,1)--(1.4,1)--(1.4,0.9);
  \draw [->] (1.7,0.9) to [out=315,in=180] (4.5,0.5);
  \draw (5,1)--(9,1);
  \draw (5,0) to [out=10,in=170] (9,0);
  \filldraw [gray] (6,0.2)--(8,0.2)--(8,1)--(6,1)--(6,0.2);
  \end{tikzpicture}
  \caption{Zoom in on the region where two boundaries meet.}
\end{figure}

We remark that if the class $\beta$ is energy minimal then this boundary component is the only boundary part. So after gluing we will get a closed moduli space.

Now we can define an operator deformed by $\mathfrak{b}=wK$. With the fixed $\rho$ we define
$$
\partial_{K} CF((L, \rho), H_{t}; \Lambda_{0})\rightarrow CF((L, \rho), H_{t}; \Lambda_{0})
$$
as
\begin{equation}
\begin{aligned}
& \partial_{K}(\sigma [l_{0},w_{0}])\\
=& \sum_{[l_{1},w_{1}]} Comp_{[w_{1}-w_{0}]}(\sigma) \sharp \mathcal{M}_{1}(([l_{0},w_{0}], [l_{1},w_{1}]); K+S))[l_{1},w_{1}].
\end{aligned}
\end{equation}
Here the sum is also over all $[l_{1},w_{1}]$ such that the corresponding moduli space is zero-dimensional.

Then we set $d^{\rho}_{K}=\delta^{\rho}+\partial_{K}$ and study when $d^{\rho}_{K}$ gives us a differential.

\begin{proposition}\label{square}
For a bulk deformation $\mathfrak{b}= wK$, the operator $d^{\rho}_{K}$ satisfies that
$$
(d^{\rho}_{K})^{2}=(\partial_{K})^{2}\equiv 0 \mod T^{E},
$$
where $E:=\min\lbrace E_{S}+ v(w), 2v(w), E_{+}\rbrace$.
\end{proposition}
\begin{proof}
By definition we have that
$$
(d^{\rho}_{K})^{2}=(\delta^{\rho})^{2}+ \delta^{\rho}\partial_{K}+ \partial_{K}\delta^{\rho}+ (\partial_{K})^{2}.
$$
Assuming Condition \ref{condition} the operator $\delta^{\rho}$ itself is a differential hence $(\delta^{\rho})^{2}\equiv 0 \mod T^{E_{+}}$. The last term $(\partial_{K})^{2}$ vanishes by the energy reason, since $\partial_{K}$ is weighted by at least $T^{v(w)}$. Then we need to show that $\delta^{\rho}\partial_{K}+ \partial_{K}\delta^{\rho}\equiv 0 \mod T^{E}$. This is obtained by considering one-dimensional moduli spaces of Floer strips with one interior hole and study the breaking of such strips, see Figure 8. By Proposition \ref{boundary} we have a list of possible degenerations. Now we discuss them by cases.

The first type of degeneration, which is strip breaking, corresponds to the sum $\delta^{\rho}\partial_{K}+ \partial_{K}\delta^{\rho}$.

The second type of degeneration corresponds to disk bubbles with Maslov index two. Since we assume Condition \ref{condition} there is no holomorphic disks with non-positive Maslov index. In this case disk bubbles on two components of line boundaries cancel with each other by the invariance of one-pointed open Gromov-Witten invariants.

The third type of degenerations are annulus bubbles. Under Condition \ref{condition}, they have positive Maslov indices and we only need to look at the Maslov two case. In this case, they correspond to the mixed type one-pointed open Gromov-Witten invariants defined in Section 4. Similar to the invariance of one-pointed open Gromov-Witten invariants, annulus bubbles appear on both sides of line boundaries. Hence they cancel in pairs.

In conclusion the codimension one boundaries of the moduli space are listed in Figure 8. Term (2) cannot happen since $L\cap S=\emptyset$. Terms in (3) and (4) cancel with each other. So the only contribution is $\delta^{\rho}\partial_{K}+ \partial_{K}\delta^{\rho}$, which corresponds to (1) and should be zero as a signed count. This completes our proof that $d^{\rho}_{K}$ is a differential modulo $T^{E}$.
\end{proof}

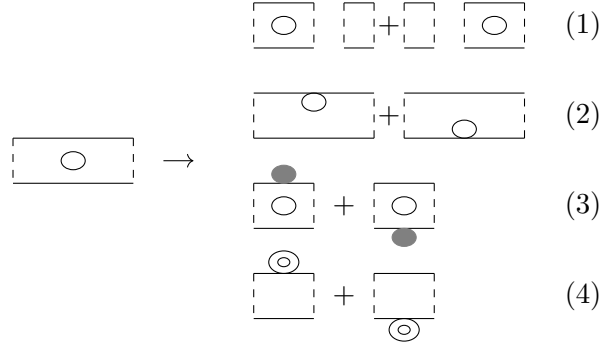
\begin{figure}
  \begin{tikzpicture}[xscale=0.8, yscale=0.6]
  \draw (-2,1) to (0,1);
  \draw (-2,0) to (0,0);
  \draw [dashed] (-2,1) to (-2,0);
  \draw [dashed] (0,1) to (0,0);
  \draw (-1,0.5) circle [radius=0.2];
  \draw [->] (0.5,0.5) to (1,0.5);
  \draw (2,1) to (4,1);
  \draw (2,2) to (4,2);
  \draw [dashed] (2,1) to (2,2);
  \draw [dashed] (4,1) to (4,2);
  \draw (3,1.8) circle [radius=0.2];
  \draw (4.5,1) to (6.5,1);
  \draw (4.5,2) to (6.5,2);
  \draw [dashed] (4.5,1) to (4.5,2);
  \draw [dashed] (6.5,1) to (6.5,2);
  \draw (5.5,1.2) circle [radius=0.2];
  \node at (4.25,1.5) {$+$};
  \draw (2,0) to (3,0);
  \draw (2,-1) to (3,-1);
  \draw (2.5,-0.5) circle [radius=0.2];
  \draw [dashed] (2,0) to (2,-1);
  \draw [dashed] (3,0) to (3,-1);
  \filldraw [gray] (2.5,0.2) circle [radius=0.2];
  \draw (4,0) to (5,0);
  \draw (4,-1) to (5,-1);
  \draw (4.5,-0.5) circle [radius=0.2];
  \draw [dashed] (4,0) to (4,-1);
  \draw [dashed] (5,0) to (5,-1);
  \filldraw [gray] (4.5,-1.2) circle [radius=0.2];
  \node at (3.5,-0.5) {$+$};
  \draw (2,-2) to (3,-2);
  \draw (2,-3) to (3,-3);
  \draw [dashed] (2,-2) to (2,-3);
  \draw [dashed] (3,-2) to (3,-3);
  \draw (2.5,-1.75) circle [radius=0.25];
  \draw (2.5,-1.75) circle [radius=0.1];
  \draw (4,-2) to (5,-2);
  \draw (4,-3) to (5,-3);
  \draw (4.5,-3.25) circle [radius=0.25];
  \draw (4.5,-3.25) circle [radius=0.1];
  \draw [dashed] (4,-2) to (4,-3);
  \draw [dashed] (5,-2) to (5,-3);
  \node at (3.5,-2.5) {$+$};
  \draw (2,3) to (3,3);
  \draw (2,4) to (3,4);
  \draw (3.5,3) to (4,3);
  \draw (3.5,4) to (4,4);
  \draw (4.5,3) to (5,3);
  \draw (5.5,3) to (6.5,3);
  \draw (4.5,4) to (5,4);
  \draw (5.5,4) to (6.5,4);
  \draw [dashed] (2,4) to (2,3);
  \draw [dashed] (3,4) to (3,3);
  \draw [dashed] (3.5,4) to (3.5,3);
  \draw [dashed] (4,4) to (4,3);
  \draw [dashed] (4.5,4) to (4.5,3);
  \draw [dashed] (5,4) to (5,3);
  \draw [dashed] (5.5,4) to (5.5,3);
  \draw [dashed] (6.5,4) to (6.5,3);
  \draw (2.5,3.5) circle [radius=0.2];
  \draw (6,3.5) circle [radius=0.2];
  \node at (4.25,3.5) {$+$};
  \node [right] at (7,3.5) {$(1)$};
  \node [right] at (7,1.5) {$(2)$};
  \node [right] at (7,-0.5) {$(3)$};
  \node [right] at (7,-2.5) {$(4)$};
  \end{tikzpicture}
  \caption{Degenerations of a one-dimensional moduli space.}
\end{figure}

Therefore the operator $d^{\rho}_{K}$ defines a differential modulo $T^{E}$ and we can talk about the cohomology modulo this energy. We write this cohomology as
$$
HF_{cy}((L, \rho), (L, \rho), H_{t}; K).
$$
In the next subsection we will study how this cohomology behaves with respect to the choice of Hamiltonian $H_{t}$. Then we can obtain the desired energy estimate. The key point is that how the energy of a Floer strip with one interior hole change under a Hamiltonian diffeomorphism.

Before dealing with a general Hamiltonian diffeomorphism, we look at the case when $H_{t}$ is $C^{2}$-small. Let $\phi$ be the time-one flow of $H_{t}$. We assume that $L\cap \phi(L)$ is transversal and $S\cap \phi(L)=\emptyset$. Moreover, the new Lagrangian submanifold $\phi(L)$ also satisfies Condition \ref{condition} with the same $J$. Then we can define a similar cohomology theory $HF_{int, cy}((L, \rho), (\phi(L), \rho); K)$ where the underlying complex is generated by intersection points of $L$ and $\phi(L)$. We call it the \textit{intersection model}. The differential is also a sum of two operators, one counts the usual holomorphic strips and the other counts holomorphic strips with one interior hole. Here the pair $(\phi(L), \rho)$ is actually $(\phi(L), (\phi^{-1})^{*}\rho)$ but for notational simplicity we just write it as $(\phi(L), \rho)$.

\begin{proposition}
The intersection model gives a cohomology theory
$$
HF_{int, cy}((L, \rho), (\phi(L), \rho); K)
$$
with coefficients $\Lambda_{0}/ T^{E}\Lambda_{0}$.
\end{proposition}
\begin{proof}
We need to show that the square of the differential is zero. It can be done by the same argument as before, using the assumption that both $S\cap \phi(L)=\emptyset$ and $S\cap L=\emptyset$. Moreover, since $H_{t}$ is $C^{2}$-small such that $L$ and $\phi(L)$ both satisfy Condition \ref{condition} with a common $J$. And the counts of holomorphic disks with Maslov index two, under the energy bound $E_{+}$, are the same. Hence possible disk bubbles on $L$ and $\phi(L)$ cancel with each other. Then the proof in Proposition \ref{square} works for this intersection model.

For a general Hamiltonian perturbation, there may be wall-crossing phenomenon for holomorphic disks with Maslov index two. So this intersection model is only defined with a small perturbation.
\end{proof}

\begin{remark}
With the assumption that $H_{t}$ is $C^{2}$-small we can prove that these two theories are equivalent as filtered cohomology groups. But we do not need this fact in our following context. The intersection model just plays a transition role between the disk model (coming from the potential function) and the chord model. In practice we will use a chord model of which the generators are chords with one end on $L$ and the other end on $\phi(L)$. And the displacement result will be proved by a limit argument since we can take $\phi$ arbitrarily small.

Another possible approach is to use a Morse-Bott model to define the above deformed Floer cohomology. In that case, we don't need to perturb $L$ by a small Hamiltonian.
\end{remark}

\subsection{Change of filtration under Hamiltonian diffeomorphisms}
Let $\phi$ be the time-one flow of $H_{t}$ (not necessarily $C^{2}$-small) such that $L$ and $\phi(L)$ intersect transversally, and the Hamiltonian chords are disjoint from $K$. Then the cohomology $HF_{cy}((L, \rho), (L, \rho), H_{t}; K)$ is well-defined with coefficient $\Lambda_{0}/ T^{E}\Lambda_{0}$. We can view the cohomology group as a $\Lambda_{0}$-module. Now we study how the choice of $H_{t}$ change the cohomology.

This deformed Floer complex is a modification of the Floer complex with bulk deformations and can be regarded as its ``first order approximation''. (Note that the differential is a sum of two operators.) The dependence of $H_{t}$ on the usual differential $\delta^{\rho}$ with local systems is well-studied in \cite{FOOO} and \cite{FOOO4}. So we focus on the part which involves the operator $\partial_{K}$. Actually we will prove a new energy estimate to construct different chain maps and chain homotopies then the rest of the arguments will follow the same proof in Section 6 and 7 in \cite{FOOO4}.

First we recall some relevant backgrounds on the \textit{geometric} version of Floer theory and the \textit{dynamical} one.

Let $L_{0}, L_{1}$ be two transversally intersecting Lagrangian submanifolds. The geometric version of the Floer complex $CF^{*}(L_{0}, L_{1})$ is generated by the intersection points
$$
p\in L_{0}\cap L_{1}
$$
where $p$ can be regarded as a constant element in the path space
$$
\Omega(L_{0}, L_{1})=\lbrace l:[0, 1]\rightarrow X\mid l(0)\in L_{0}, l(1)\in L_{1}\rbrace.
$$
We fix a base path $l_{a}\in \Omega(L_{0}, L_{1})$ for each component $a\in \pi_{0}(\Omega(L_{0}, L_{1}))$. Let $(l, w)$ be a pair such that $l\in \Omega(L_{0}, L_{1})$ and $w: [0, 1]^{2}\rightarrow X$ satisfying
$$
w(s, 0)\in L_{0}, w(s, 1)\in L_{1}, w(0, t)=l_{a}(t), w(1, t)=l(t).
$$
Then we define the geometric action functional
\begin{equation}
\mathcal{A}_{l_{a}}((l, w))=\int w^{*}\omega
\end{equation}
on the space of pairs $(l, w)$. Then as in the previous subsection, we modulo the $\Gamma$-equivalence relation between $(l,w)$ and $(l, w')$ and write $[l, w]$ as the equivalence class.

For two Lagrangian submanifolds $L_{0}, L_{1}$ and a time-dependent Hamiltonian $\tilde{H}_{t}$, the dynamical version of the Floer complex is generated by the solutions of Hamilton's equation
$$
\lbrace x\in \Omega(L_{0}, L_{1})\mid \dot{x}=X_{\tilde{H}_{t}}(x) \rbrace.
$$
For a fixed base path $x_{a}$ and a pair $[x, w]$, the dynamical action functional is defined as
\begin{equation}
\mathcal{A}_{\tilde{H}_{t}, x_{a}}([x, w])=\int w^{*}\omega +\int_{0}^{1} \tilde{H}_{t}(x(t))dt.
\end{equation}

The two versions of Floer complexes can be regarded as filtered complexes with respect to their action functionals. And those two Floer theories are related by a transformation. We refer to Section 4 in \cite{FOOO4} for more details.

Next we introduce the notion of the \textit{perturbed} Cauchy-Riemann equation to study the relation between these two versions of Floer theories. Let $\chi_{+}(\tau):\mathbb{R}\rightarrow \mathbb{R}$ be a smooth function such that
$$
\chi_{+}(\tau)=
\begin{cases}
0 \quad \tau\leq -2,\\
1 \quad \tau\geq -1,
\end{cases}
\chi_{+}'(\tau)\geq 0
$$
and $\chi_{-}(\tau)=1-\chi_{+}(\tau)$. Also we will use a family of smooth bump functions $\chi_{N}(\tau)$ for $N\geq 1$, satisfying
$$
\chi_{N}(\tau)=
\begin{cases}
0 \quad \abs\tau \geq N+1,\\
1 \quad \abs\tau \leq N,
\end{cases}
$$
and
$$
\chi_{N}'(\tau)\geq 0, \forall \tau\in [-N-1, -N], \quad \chi_{N}'(\tau)\leq 0, \forall \tau\in [N, N+1].
$$
In particular, we assume that on $[-N-1, -N]$ ($[N, N+1]$ respectively) the function $\chi_{N}$ is a translation of $\chi_{+}$ ($\chi_{-}$ respectively). For $N\leq 1$ we define $\chi_{N}(\tau)=N\chi_{1}(\tau)$ such that $\chi_{N}(\tau)$ converges to the zero function as $N$ goes to zero.

From now on we assume that our pairs $L_{0}, L_{1}$ intersect transversally, $L_{0}$ satisfies Condition \ref{condition} and $L_{1}$ is a small Hamiltonian perturbation of $L_{0}$. Suppose that $L_{0}$ satisfies Condition \ref{condition} for some $J_{0}$ then we choose $J_{1}$ such that $L_{1}$ satisfies Condition \ref{condition} for $J_{1}$ and $J_{0}, J_{1}$ can be connected by families of $J$'s satisfying Condition \ref{condition} except $(2)$. See Remark \ref{div} and \ref{analytic} for the choices of almost complex structures. The perturbed Cauchy-Riemann equation of $u(\tau, t): \mathbb{R}\times [0, 1]\rightarrow X$ is the following
\begin{equation}
\begin{cases}
\dfrac{\partial u}{\partial \tau}+ J^{\tau}_{t}(\dfrac{\partial u}{\partial t}- \chi(\tau)X_{\tilde{H}_{t}}(u))=0,\\
u(\tau, 0)\in L_{0}, \quad u(\tau, 1)\in L_{1}.
\end{cases}
\end{equation}
Here $J^{\tau}_{t}$ is a family of almost complex structures connecting $J_{0}$ and $J_{1}$. And $\chi(\tau)$ is one of the bump functions we defined before. Similarly we can define the perturbed Cauchy-Riemann equation where the domain is $Strip_{\epsilon, r}$, a strip with one interior hole.

The energy of a solution $u$ is defined as
$$
E_{(J, \chi(\tau), \tilde{H}_{t})}(u)= \int \abs{\dfrac{\partial u}{\partial \tau}}^{2}_{J}
$$
and we will study the moduli space of finite energy solutions. First we review the energy estimate of solutions when the domain is a strip without holes. From now on, we assume that all the Hamiltonian functions are normalized, that is, they satisfy that $\int_{X}\tilde{H}_{t} =0$.

\begin{lemma}(Lemma 5.1, \cite{FOOO4})
Let $u$ be a finite energy solution of the perturbed Cauchy-Riemann equation with domain $Strip$. Then we have that
\begin{equation}
\begin{aligned}
E_{(J, \chi(\tau), \tilde{H}_{t})}(u) &=\int u^{*}\omega +\int_{0}^{1}\tilde{H}_{t}(u(+\infty, t))dt\\
&-\int_{-\infty}^{\infty}\chi'(\tau)\int_{0}^{1}\tilde{H}_{t}(u)dtd\tau.
\end{aligned}
\end{equation}
\end{lemma}

When the domain is a strip with one interior hole we can do the similar computation. As expected, the result has one more term involving the integral on the circle boundary. We will compute the cases where $\chi=\chi_{+}, \chi=\chi_{-}$ and $\chi=\chi_{N}$. First we fix the center of the interior hole at $(0, \frac{1}{2})$ and write
$$
Strip_{\epsilon}:= Strip_{\epsilon, \frac{1}{2}}=\lbrace (\tau, t)\in \mathbb{R}\times [0, 1]\subset \mathbb{C}\mid \tau^{2}+ (t- \frac{1}{2})^{2}\geq \epsilon ^{2}\rbrace
$$
to do the computation.

\begin{lemma}\label{est1}
Let $u$ be a finite energy solution of the perturbed Cauchy-Riemann equation with domain $Strip_{\epsilon}$. Then we have that
\begin{equation}\label{est}
\begin{aligned}
E_{(J, \chi(\tau), \tilde{H}_{t})}(u) &=\int u^{*}\omega +\int_{0}^{1}\tilde{H}_{t}(u(+\infty, t))dt\\
&-\int_{-\infty}^{\infty}\chi'(\tau)\int_{0}^{1}\tilde{H}_{t}(u)dtd\tau -\int_{C(\epsilon)}\tilde{H}_{t}(u)
\end{aligned}
\end{equation}
when $\chi(\tau)=\chi_{+}(\tau)$.
\end{lemma}
\begin{proof}
We prove the lemma by a direct computation.
\begin{equation}
\begin{aligned}
E_{(J, \chi(\tau), H_{t})}(u) &=\int_{Strip_{\epsilon}} \abs{\dfrac{\partial u}{\partial \tau}}^{2}_{J}= \int_{Strip_{\epsilon}} \omega(\dfrac{\partial u}{\partial \tau}, J\dfrac{\partial u}{\partial \tau})\\
&= \int_{Strip_{\epsilon}} \omega(\dfrac{\partial u}{\partial \tau}, \dfrac{\partial u}{\partial t}-\chi(\tau)X_{\tilde{H}_{t}}(u))\\
&= \int_{Strip_{\epsilon}} \omega(\dfrac{\partial u}{\partial \tau}, \dfrac{\partial u}{\partial t})- \int_{Strip_{\epsilon}} \omega(\dfrac{\partial u}{\partial \tau}, \chi(\tau)X_{\tilde{H}_{t}}(u))\\
&= \int_{Strip_{\epsilon}} u^{*}\omega+ \int_{Strip_{\epsilon}} \chi(\tau)\cdot d\tilde{H}_{t}(u)(\dfrac{\partial u}{\partial \tau})\\
&= \int_{Strip_{\epsilon}} u^{*}\omega+ \int_{Strip_{\epsilon}} \chi(\tau)\cdot \dfrac{\partial}{\partial \tau}\tilde{H}_{t}(u).
\end{aligned}
\end{equation}

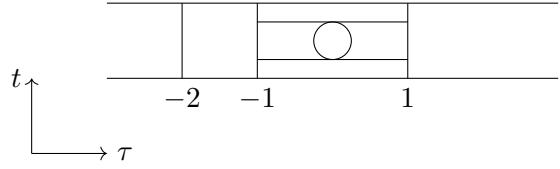
\begin{figure}
  \begin{tikzpicture}
  \draw (-3,0)--(3,0);
  \draw (-3,1)--(3,1);
  \draw (-2,0)--(-2,1);
  \draw (-1,0)--(-1,1);
  \draw (1,0)--(1,1);
  \draw (-1,0.25)--(1,0.25);
  \draw (-1,0.75)--(1,0.75);
  \draw (0,0.5) circle [radius=0.25];
  \draw [->] (-4,-1)--(-3,-1);
  \draw [->] (-4,-1)--(-4,0);
  \node [left] at (-4,0) {$t$};
  \node [right] at (-3,-1) {$\tau$};
  \node [below] at (-2,0) {$-2$};
  \node [below] at (-1,0) {$-1$};
  \node [below] at (1,0) {$1$};
  \end{tikzpicture}
  \caption{Divide $Strip_{\epsilon}$ into regions to do integration.}
\end{figure}

Next we consider the last term.
\begin{equation}\label{cal}
\begin{aligned}
& \int_{Strip_{\epsilon}} \chi(\tau)\cdot \dfrac{\partial}{\partial \tau}\tilde{H}_{t}(u)\\
=& \int_{Strip_{\epsilon}, \tau \leq -2} \chi(\tau)\cdot \dfrac{\partial}{\partial \tau}\tilde{H}_{t}(u)+ \int_{Strip_{\epsilon}, -2\leq \tau \leq -1} \chi(\tau)\cdot \dfrac{\partial}{\partial \tau}\tilde{H}_{t}(u)\\
+& \int_{Strip_{\epsilon}, -1\leq \tau \leq 1} \chi(\tau)\cdot \dfrac{\partial}{\partial \tau}\tilde{H}_{t}(u)+ \int_{Strip_{\epsilon}, 1\leq \tau} \chi(\tau)\cdot \dfrac{\partial}{\partial \tau}\tilde{H}_{t}(u)\\
\end{aligned}
\end{equation}
For $\tau\leq -2$, the integral is zero since $\chi(\tau)$ is zero. For $-2\leq \tau \leq -1$, the integral is
\begin{equation}\label{cal1}
\begin{aligned}
& \int_{-2}^{-1}\int_{0}^{1} \chi(\tau)\cdot \dfrac{\partial}{\partial \tau}\tilde{H}_{t}(u)\\
=& \int_{-2}^{-1} \chi(\tau)\cdot \dfrac{\partial}{\partial \tau}\int_{0}^{1}\tilde{H}_{t}(u) dt d\tau\\
=& (\chi(\tau) \cdot \int_{0}^{1}\tilde{H}_{t}(u) dt)\vert_{-2}^{-1}- \int_{-2}^{-1}\chi'(\tau)\int_{0}^{1}\tilde{H}_{t}(u)dt d\tau\\
=& \int_{0}^{1} \tilde{H}_{t}(u(-1, t)) dt- \int_{-2}^{-1}\chi'(\tau)\int_{0}^{1}\tilde{H}_{t}(u)dt d\tau.
\end{aligned}
\end{equation}
Similarly for $1\leq\tau$, the integral is
\begin{equation}\label{cal2}
\begin{aligned}
& \int_{1}^{+\infty}\int_{0}^{1} \chi(\tau)\cdot \dfrac{\partial}{\partial \tau}\tilde{H}_{t}(u)\\
=& \int_{0}^{1} \tilde{H}_{t}(u(+\infty, t)) dt- \int_{0}^{1} \tilde{H}_{t}(u(1, t)) dt .
\end{aligned}
\end{equation}
Now we consider the terms involving the interior hole. For $-1\leq \tau \leq 1$ we have that $\chi(\tau)\equiv 1$ and the integral can be split as
\begin{equation}\label{cal3}
\begin{aligned}
& \int_{Strip_{\epsilon}, -1\leq \tau \leq 1} \dfrac{\partial}{\partial \tau}\tilde{H}_{t}(u)\\
=& \int_{-1}^{1}\int_{\frac{1}{2}+\epsilon}^{1}\dfrac{\partial}{\partial \tau}\tilde{H}_{t}(u) dt d\tau+ \int_{-1}^{1}\int^{\frac{1}{2}-\epsilon}_{0}\dfrac{\partial}{\partial \tau}\tilde{H}_{t}(u) dt d\tau\\
+& \int_{\frac{1}{2}-\epsilon}^{\frac{1}{2}+\epsilon}\int_{-1}^{-\sqrt{\epsilon^{2}-(t-\frac{1}{2})^{2}}}\dfrac{\partial}{\partial \tau}\tilde{H}_{t}(u) d\tau dt + \int_{\frac{1}{2}-\epsilon}^{\frac{1}{2}+\epsilon}\int^{1}_{\sqrt{\epsilon^{2}-(t-\frac{1}{2})^{2}}}\dfrac{\partial}{\partial \tau}\tilde{H}_{t}(u) d\tau dt.
\end{aligned}
\end{equation}
Direct computation gives that
\begin{equation}\label{cal4}
\begin{aligned}
& \int_{-1}^{1}\int_{\frac{1}{2}+\epsilon}^{1}\dfrac{\partial}{\partial \tau}\tilde{H}_{t}(u) dt d\tau= \int_{\frac{1}{2}+ \epsilon}^{1}\tilde{H}_{t}(u(1, t)) dt- \int_{\frac{1}{2}+ \epsilon}^{1}\tilde{H}_{t}(u(-1, t)) dt\\
& \int_{-1}^{1}\int^{\frac{1}{2}-\epsilon}_{0}\dfrac{\partial}{\partial \tau}\tilde{H}_{t}(u) dt d\tau= \int^{\frac{1}{2}- \epsilon}_{0}\tilde{H}_{t}(u(1, t)) dt- \int^{\frac{1}{2}- \epsilon}_{0}\tilde{H}_{t}(u(-1, t)) dt
\end{aligned}
\end{equation}
and
\begin{equation}\label{cal5}
\begin{aligned}
& \int_{\frac{1}{2}-\epsilon}^{\frac{1}{2}+\epsilon}\int_{-1}^{-\sqrt{\epsilon^{2}-(t-\frac{1}{2})^{2}}}\dfrac{\partial}{\partial \tau}\tilde{H}_{t}(u) d\tau dt\\
=& \int_{\frac{1}{2}- \epsilon}^{\frac{1}{2}+ \epsilon}\tilde{H}_{t}(u(-\sqrt{\epsilon^{2}-(t-\frac{1}{2})^{2}}, t)) dt- \int_{\frac{1}{2}- \epsilon}^{\frac{1}{2}+ \epsilon}\tilde{H}_{t}(u(-1, t)) dt\\
& \int_{\frac{1}{2}-\epsilon}^{\frac{1}{2}+\epsilon}\int^{1}_{\sqrt{\epsilon^{2}-(t-\frac{1}{2})^{2}}}\dfrac{\partial}{\partial \tau}\tilde{H}_{t}(u) d\tau dt\\
=& -\int_{\frac{1}{2}- \epsilon}^{\frac{1}{2}+ \epsilon}\tilde{H}_{t}(u(\sqrt{\epsilon^{2}-(t-\frac{1}{2})^{2}}, t)) dt+ \int_{\frac{1}{2}- \epsilon}^{\frac{1}{2}+ \epsilon}\tilde{H}_{t}(u(1, t)) dt.
\end{aligned}
\end{equation}
Put all $(\ref{cal1})-(\ref{cal5})$ into (\ref{cal}) we get the desired estimate. Here we write
$$
\begin{aligned}
&\int_{C(\epsilon)} \tilde{H}_{t}(u)\\
=& \int_{\frac{1}{2}- \epsilon}^{\frac{1}{2}+ \epsilon}\tilde{H}_{t}(u(\sqrt{\epsilon^{2}-(t-\frac{1}{2})^{2}}, t)) dt- \int_{\frac{1}{2}- \epsilon}^{\frac{1}{2}+ \epsilon}\tilde{H}_{t}(u(-\sqrt{\epsilon^{2}-(t-\frac{1}{2})^{2}}, t)) dt,
\end{aligned}
$$
which corresponds to integrate $\tilde{H}_{t}$ along $C(\epsilon)$ in the counter-clockwise direction.
\end{proof}

\begin{remark}\label{echange}
Note that
$$
\begin{aligned}
\int_{C(\epsilon)} \tilde{H}_{t}(u)= &\int_{\frac{1}{2}- \epsilon}^{\frac{1}{2}+ \epsilon}\tilde{H}_{t}(u(\sqrt{\epsilon^{2}-(t-\frac{1}{2})^{2}}, t)) dt- \int_{\frac{1}{2}- \epsilon}^{\frac{1}{2}+ \epsilon}\tilde{H}_{t}(u(-\sqrt{\epsilon^{2}-(t-\frac{1}{2})^{2}}, t)) dt\\
=& \int_{\frac{1}{2}- \epsilon}^{\frac{1}{2}+ \epsilon}[\tilde{H}_{t}(u(\sqrt{\epsilon^{2}-(t-\frac{1}{2})^{2}}, t))- \tilde{H}_{t}(u(-\sqrt{\epsilon^{2}-(t-\frac{1}{2})^{2}}, t))] dt\\
\leq & \int_{\frac{1}{2}- \epsilon}^{\frac{1}{2}+ \epsilon}[\max_{S}\tilde{H}_{t} -\min_{S}\tilde{H}_{t}]dt\\
\leq & \int_{0}^{1}[\max_{S}\tilde{H}_{t} -\min_{S}\tilde{H}_{t}]dt = \norm{\tilde{H}_{t}}_{S}.
\end{aligned}
$$
Hence we actually find that
$$
-\norm{\tilde{H}_{t}}_{S}\leq  \int_{C(\epsilon)} \tilde{H}_{t}(u)\leq  \norm{\tilde{H}_{t}}_{S}
$$
for all $\epsilon\in (0, \frac{1}{2})$.
\end{remark}

By the same computation when $\chi(\tau)=\chi_{-}$ we have that

\begin{lemma}\label{est2}
Let $u$ be a finite energy solution of the perturbed Cauchy-Riemann equation with domain $Strip_{\epsilon}$. Then we have that
\begin{equation}
\begin{aligned}
E_{(J, \chi(\tau), \tilde{H}_{t})}(u) &=\int u^{*}\omega -\int_{0}^{1}\tilde{H}_{t}(u(-\infty, t))dt\\
&-\int_{-\infty}^{\infty}\chi'(\tau)\int_{0}^{1}\tilde{H}_{t}(u)dtd\tau -\int_{C(\epsilon)}\tilde{H}_{t}(u)
\end{aligned}
\end{equation}
when $\chi(\tau)=\chi_{-}(\tau)$.
\end{lemma}

And when $\chi(\tau)=\chi_{N}$ we have that

\begin{lemma}\label{est3}
Let $u$ be a finite energy solution of the perturbed Cauchy-Riemann equation with domain $Strip_{\epsilon}$. Then we have that
\begin{equation}
E_{(J, \chi(\tau), \tilde{H}_{t})}(u) =\int u^{*}\omega -\int_{-\infty}^{\infty}\chi'(\tau)\int_{0}^{1}\tilde{H}_{t}(u)dtd\tau -\int_{C(\epsilon)}\tilde{H}_{t}(u)
\end{equation}
when $\chi(\tau)=\chi_{N}(\tau)$.
\end{lemma}

The above three lemmas provide necessary energy estimates for us to establish the chain maps and chain homotopies when we change the Hamiltonian functions $H_{t}$. More precisely, they give the estimates of maximal action loss for chain maps. Now we explain how to use them in our situations.

Let $[p,w']$ and $[l,w]$ be the input and output of the strip respectively, with $[w-w']=[u]$. Then the first two terms in (\ref{est}) correspond to the difference between the actions of the input and output. And the last two terms correspond to the ``action loss''. Note that $\chi_{+}(\tau)\geq 0$ and $\chi_{+}(-\infty)=0, \chi_{+}(+\infty)=1$ we have that the maximal action loss is
\begin{equation}\label{el1}
-\int_{0}^{1} \max_{X} H_{t}dt -\int_{C(\epsilon)} H_{t}(u)\geq -\int_{0}^{1} \max_{X} H_{t}dt -\norm{H_{t}}_{S}
\end{equation}
for any solution $u$ in Lemma \ref{est1}. Similarly the maximal action loss is
\begin{equation}\label{el2}
\int_{0}^{1} \min_{X} H_{t}dt -\int_{C(\epsilon)} H_{t}(u)\geq \int_{0}^{1} \min_{X} H_{t}dt -\norm{H_{t}}_{S}
\end{equation}
for any solution $u$ in Lemma \ref{est2}. We remark that both Lemma \ref{est1} and Lemma \ref{est2} estimate the energy of the solution over the domain $Strip_{\epsilon}$ where the interior hole is centered at $(0, \frac{1}{2})$. If we move the center of the hole to $(r', r'')$ then similar estimate can only be weaker. For example, when the hole is contained outside the support of $\chi(\tau)$ then the fourth term in (\ref{est}) will be zero. When the hole is not contained in the region where $\chi(\tau)=1$, the fourth term will only be smaller than the case we did in (\ref{est}) because $\chi(\tau)\leq 1$ and $\chi'(\tau) \geq 0$. In conclusion, the above estimates of maximal action loss work for all the case when we move the center of the interior hole.

Next we construct the chain maps. We fix a $C^{2}$-small perturbation $\varphi$ such that $L\cap \varphi(L)$ transversally and $\varphi(L)\cap S=\emptyset$. Now for a Hamiltonian $G_{t}$, let $\phi$ be its time-one flow. When $L\cap \phi(\varphi(L))$ is transversal we can also define the cohomology
$$
HF_{cy}((L, \rho), (\varphi(L), \rho), G_{t}; K)
$$
where the generators are chords of $G_{t}$ with ends on $L$ and $\varphi(L)$. Here we remark that when $\varphi$ is small $L$ and $\varphi(L)$ have the same one-pointed open Gromov-Witten invariants, under energy $E_{+}$. Hence we can define this cohomology generated by chords with ends on $L$ and $\varphi(L)$, similar to Proposition \ref{square}. For a general Hamiltonian isotopy there may be wall-crossing phenomenon of the one-pointed invariants which cannot be prevented only by Condition \ref{condition}.

Then we use the perturbed Cauchy-Riemann equation to construct chain maps
$$
CF_{int, cy}((L, \rho), (\varphi(L), \rho); K)\rightarrow CF_{cy}((L, \rho), (\varphi(L), \rho), G_{t}; K)
$$
and
$$
CF_{cy}((L, \rho), (\varphi(L), \rho), G_{t}; K)\rightarrow CF_{int, cy}((L, \rho), (\varphi(L), \rho); K).
$$
We remark that the two maps are constructed by using the cut-off functions $\chi_{+}$ and $\chi_{-}$ respectively. Then chain homotopy map is constructed by using the cut-off function $\chi_{N}$.

\begin{figure}
  \begin{tikzpicture}[xscale=1, yscale=0.8]
  \draw (0,0) to [out=0,in=180] (2,-1);
  \draw (2,-1) to (3,-1);
  \draw (0,-1) to [out=0,in=180] (2,0);
  \draw (2,0) to (3,0);
  \draw [dashed] (3,-1) to (3,0);
  \draw [->] (3.5,-0.5) to (4,-0.5);
  \draw (4.5,0) to [out=0,in=180] (5,-1);
  \draw (4.5,-1) to [out=0,in=180] (5,0);
  \draw (5,0) to [out=0,in=180] (5.5,-1);
  \draw (5.5,-1) to (6,-1);
  \draw (5,-1) to [out=0,in=180] (5.5,0);
  \draw (5.5,0) to (6,0);
  \draw [dashed] (6,-1) to (6,0);
  \node at (6.5,-0.5) {$+$};
  \draw (7,0) to [out=0,in=180] (7.5,-1);
  \draw (7,-1) to [out=0,in=180] (7.5,0);
  \draw (7.5,-1) to (8,-1);
  \draw (7.5,0) to (8,0);
  \draw [dashed] (8,-1) to (8,0);
  \draw (8.2,0) to (8.7,0);
  \draw (8.2,-1) to (8.7,-1);
  \draw [dashed] (8.2,-1) to (8.2,0);
  \draw [dashed] (8.7,-1) to (8.7,0);
  \node [right] at (9,-0.5) {$(1)$};
  \draw (0,-3) to [out=0,in=180] (2,-4);
  \draw (2,-4) to (3,-4);
  \draw (0,-4) to [out=0,in=180] (2,-3);
  \draw (2,-3) to (3,-3);
  \draw (2.3,-3.5) circle [radius=0.2];
  \draw [dashed] (3,-4) to (3,-3);
  \draw [->] (3.5,-3.5) to (4,-3.5);
  \draw (4.5,-2) to [out=0,in=180] (5,-3);
  \draw (4.5,-3) to [out=0,in=180] (5,-2);
  \draw (5,-2) to [out=0,in=180] (5.5,-3);
  \draw (5.5,-3) to (6,-3);
  \draw (5,-3) to [out=0,in=180] (5.5,-2);
  \draw (5.5,-2) to (6,-2);
  \draw (5.6,-2.5) circle [radius=0.15];
  \draw [dashed] (6,-3) to (6,-2);
  \node at (6.5,-2.5) {$+$};
  \draw (7,-2) to [out=0,in=180] (7.5,-3);
  \draw (7,-3) to [out=0,in=180] (7.5,-2);
  \draw (7.5,-3) to (8,-3);
  \draw (7.5,-2) to (8,-2);
  \draw (7.6,-2.5) circle [radius=0.15];
  \draw [dashed] (8,-3) to (8,-2);
  \draw (8.2,-2) to (8.7,-2);
  \draw (8.2,-3) to (8.7,-3);
  \draw [dashed] (8.2,-3) to (8.2,-2);
  \draw [dashed] (8.7,-3) to (8.7,-2);
  \node [right] at (9,-2.5) {$(2)$};
  \draw (4.5,-4) to [out=0,in=180] (5,-5);
  \draw (4.5,-5) to [out=0,in=180] (5,-4);
  \draw (5,-4) to [out=0,in=180] (5.5,-5);
  \draw (5.5,-5) to (6,-5);
  \draw (5,-5) to [out=0,in=180] (5.5,-4);
  \draw (5.5,-4) to (6,-4);
  \draw (5,-4.5) circle [radius=0.15];
  \draw [dashed] (6,-5) to (6,-4);
  \node at (6.5,-4.5) {$+$};
  \draw (7,-4) to [out=0,in=180] (7.5,-5);
  \draw (7,-5) to [out=0,in=180] (7.5,-4);
  \draw (7.5,-5) to (8,-5);
  \draw (7.5,-4) to (8,-4);
  \draw [dashed] (8,-5) to (8,-4);
  \draw (8.2,-4) to (8.7,-4);
  \draw (8.2,-5) to (8.7,-5);
  \draw (8.45,-4.5) circle [radius=0.15];
  \draw [dashed] (8.2,-5) to (8.2,-4);
  \draw [dashed] (8.7,-5) to (8.7,-4);
  \node [right] at (9,-4.5) {$(3)$};
  \end{tikzpicture}
  \caption{Degenerations of solutions of the perturbed Cauchy-Riemann equation.}
\end{figure}
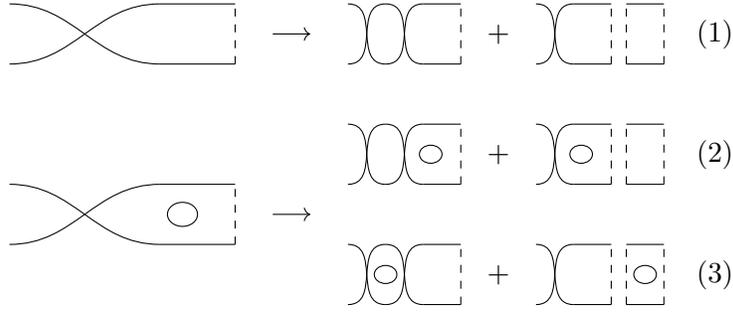

\begin{proposition}
Let $(X, S, L)$ be a symplectic 6-manifold, a Lagrangian 3-sphere and a Lagrangian submanifold satisfying Condition \ref{condition}. Let $(H_{t}, \varphi)$ and $(G_{t}, \phi)$ be the Hamiltonians we chose as above. Then there are two chain maps
$$
\Phi_{+}: CF_{int, cy}((L, \rho), (\varphi(L), \rho); K)\rightarrow CF_{cy}((L, \rho), (\varphi(L), \rho), G_{t}; K)
$$
and
$$
\Phi_{-}: CF_{cy}((L, \rho), (\varphi(L), \rho), G_{t}; K)\rightarrow CF_{int, cy}((L, \rho), (\varphi(L), \rho); K),
$$
modulo $T^{E}$.
\end{proposition}
\begin{proof}
The proof is similar to the proof of Theorem 6.2 in \cite{FOOO4}. The only difference is that we apply our energy estimate of the change of filtration when the domain has an interior hole. So this difference results in the extra term $\norm{H}_{S}$.

First for a fixed cut-off function $\chi_{+}$ we define a chain map
$$
\Phi_{+}: CF_{int, cy}((L, \rho), (\varphi(L), \rho); K)\rightarrow CF_{cy}((L, \rho), (\varphi(L), \rho), G_{t}; K)
$$
by $\Phi_{+}=T^{\tilde{E}_{+}}(\Phi_{+, 0}+ \Phi_{+, 1})$. Here
$$
\Phi_{+, 0}(\sigma[p,w'])= \sum_{[l,w]} Comp_{[w-w']}(\sigma) \sharp \mathcal{M}_{0}([p,w'], [l,w])\cdot [l,w]
$$
and
$$
\Phi_{+, 1}(\sigma[p,w'])= \sum_{[l,w]} Comp_{[w-w']}(\sigma) \sharp \mathcal{M}_{1}([p,w'], [l,w])\cdot [l,w]\cdot T^{v(K)}.
$$
Here the actual bulk deformation is $\mathfrak{b}=wK, w\in \Lambda+$. For notational simplicity, we just write $v(K)$ instead of $v(w)$ since $w$ also represents some role in $[l, w]$. The energy weights $T^{\tilde{E}_{+}}$ is necessary if we require that these maps do not decrease the action filtration. Note that there will be energy loss for the perturbed Cauchy-Riemann equation. And the maximal energy loss is computed in (\ref{el1}) and (\ref{el2}). So if we set
$$
\tilde{E}_{+}=\int_{0}^{1} \max_{X} G_{t}dt +\norm{G_{t}}_{S}
$$
then we get a map which does not decrease the action.

We explain the moduli spaces as follows. The moduli space $\mathcal{M}_{0}([p,w'], [l,w])$ contains solutions of the perturbed Cauchy-Riemann equation when the domain is a genuine strip, representing the class $[w-w']$. The moduli space $\mathcal{M}_{1}([p,w'], [l,w])$ is obtained by gluing two moduli spaces
$$
\mathcal{M}_{1}([p,w'], [l,w])= \mathcal{M}_{1, pt}([p,w'], [l,w])\sqcup \mathcal{M}_{1, hole}([p,w'], [l,w])/\sim
$$
where $\mathcal{M}_{1, pt}([p,w'], [l,w])$ contains solutions of the perturbed Cauchy-Riemann equation when the domain is a strip with one interior marked point and $\mathcal{M}_{1, hole}([p,w'], [l,w])$ contains solutions when the domain is a strip with one interior hole. Both the interior marked point and the center of the interior hole can move freely. And the gluing is understood as we did in defining $\partial_{K}$.

Next we show that $\Phi_{+}$ is a chain map. That is,
$$
\Phi_{+} d^{\rho}_{K, int}+ d^{\rho}_{K}\Phi_{+}\equiv 0 \mod T^{E}.
$$
Note that
\begin{equation}\label{chain map}
\begin{aligned}
& \Phi_{+} d^{\rho}_{K, int}+ d^{\rho}_{K}\Phi_{+}\\
=& T^{\tilde{E}_{+}}(\Phi_{+, 0}+ \Phi_{+, 1})(\delta^{\rho}_{int}+ \partial_{K, int})+ T^{\tilde{E}_{+}}(\delta^{\rho}+ \partial_{K})(\Phi_{+, 0}+ \Phi_{+, 1})
\end{aligned}
\end{equation}
and there are eight terms in the full expansion. After compensating the energy loss by $T^{\tilde{E}_{+}}$, the sum
$$
T^{\tilde{E}_{+}}(\Phi_{+, 1}\partial_{K, int}+ \partial_{K}\Phi_{+, 1})\equiv 0 \mod T^{2v(K)} \quad (\text{hence} \equiv 0 \mod T^{E})
$$
by the energy reason. So we need to check the remaining sum of six terms is zero. The proof is by studying all types of degenerations of one-dimensional moduli spaces. By similar argument in Proposition \ref{square}, we can exclude the sphere bubbles, disk bubbles and annulus bubbles. Then there are six types of degenerations for the moduli spaces $\mathcal{M}_{0}([p,w'], [l,w])$ and $\mathcal{M}_{1}([p,w'], [l,w])$, shown in Figure 10. In particular, the terms in $(1)$ correspond to
$$
\Phi_{+, 0}\delta^{\rho}_{int}+ \delta^{\rho}\Phi_{+, 0}
$$
which are from the boundary components of $\mathcal{M}_{0}([p,w'], [l,w])$. Hence the sum, weighted by $T^{\tilde{E}_{+}}$, vanishes. Similarly the terms in $(2)$ correspond to
$$
\Phi_{+, 1}\delta^{\rho}_{int}+ \delta^{\rho}\Phi_{+, 1}
$$
and the terms in $(3)$ correspond to
$$
\Phi_{+, 0}\partial_{K, int}+ \partial_{K}\Phi_{+, 0}.
$$
Therefore the sum of these four terms, weighted by $T^{\tilde{E}_{+}+v(K)}$, vanishes. In conclusion we have that the sum of these eight terms in (\ref{chain map}) is zero and $\Phi_{+}$ is a chain map. In the same way we can construct
$$
\Phi_{-}=T^{\tilde{E}_{-}}(\Phi_{-, 0}+ \Phi_{-, 1})
$$
as a chain map by a chosen cut-off function $\chi_{-}$. Here
$$
\tilde{E}_{-}=-\int_{0}^{1} \min_{X} G_{t}dt +\norm{G_{t}}_{S}.
$$
Then $\Phi_{\pm}$ induce maps on the cohomology level, which we still write as $\Phi_{\pm}$.
\end{proof}

Next we construct chain homotopy maps such that $\Phi_{-}\circ\Phi_{+}$ is chain homotopic to some inclusion-induced map.

\begin{figure}
  \begin{center}
  \begin{tikzpicture}[xscale=1, yscale=0.8]
  \draw (0,1) to [out=0,in=180] (1,0);
  \draw (0,0) to [out=0,in=180] (1,1);
  \draw (3,1) to [out=0,in=180] (4,0);
  \draw (3,0) to [out=0,in=180] (4,1);
  \draw (1,1)--(3,1);
  \draw (1,0)--(3,0);
  \node at (2,0.5) {$\chi_{N}$};
  \draw [->] (4.5,0.5)--(5,0.5);
  \draw (5.5,1) to [out=0,in=180] (6,0);
  \draw (5.5,0) to [out=0,in=180] (6,1);
  \draw (6,1) to [out=0,in=180] (6.5,0);
  \draw (6,0) to [out=0,in=180] (6.5,1);
  \draw (6.5,1) to [out=0,in=180] (7,0);
  \draw (6.5,0) to [out=0,in=180] (7,1);
  \node at (6.5,0.5) {\small{$\chi_{N}$}};
  \node at (7.5,0.5) {$+$};
  \draw (8,1) to [out=0,in=180] (8.5,0);
  \draw (8,0) to [out=0,in=180] (8.5,1);
  \draw (8.5,1) to [out=0,in=180] (9,0);
  \draw (8.5,0) to [out=0,in=180] (9,1);
  \draw (9,1) to [out=0,in=180] (9.5,0);
  \draw (9,0) to [out=0,in=180] (9.5,1);
  \node at (8.5,0.5) {\small{$\chi_{N}$}};
  \node [right] at (10,0.5) {$(2)=\mathfrak{f}_{0}\delta^{\rho}+ \delta^{\rho}\mathfrak{f}_{0}$};
  \draw (5.5,2.5) to [out=0,in=180] (6.5,1.5);
  \draw (5.5,1.5) to [out=0,in=180] (6.5,2.5);
  \draw (8.5,2.5) to [out=0,in=180] (9.5,1.5);
  \draw (8.5,1.5) to [out=0,in=180] (9.5,2.5);
  \draw (6.5,2.5)--(7.4,2.5);
  \draw (7.6,2.5)--(8.5,2.5);
  \draw (6.5,1.5)--(7.4,1.5);
  \draw (7.6,1.5)--(8.5,1.5);
  \draw [dashed] (7.4,2.5)--(7.4,1.5);
  \draw [dashed] (7.6,2.5)--(7.6,1.5);
  \node at (6.8,2) {$\chi_{+}$};
  \node at (8.2,2) {$\chi_{-}$};
  \node [right] at (10,2) {$(1)=\Phi_{-,0}\Phi_{+,0}$};
  \draw (5.5,-0.5) to [out=0,in=180] (6.5,-1.5);
  \draw (5.5,-1.5) to [out=0,in=180] (6.5,-0.5);
  \draw (8.5,-0.5) to [out=0,in=180] (9.5,-1.5);
  \draw (8.5,-1.5) to [out=0,in=180] (9.5,-0.5);
  \draw (6.5,-0.5)--(8.5,-0.5);
  \draw (6.5,-1.5)--(8.5,-1.5);
  \node at (7.5,-1) {$\chi_{0}$};
  \node [right] at (10,-1) {$(3)=\mathfrak{i}_{0}$};
  \end{tikzpicture}
  \end{center}
  \caption{Degenerations in $\mathcal{M}_{0}^{para}$.}
\end{figure}
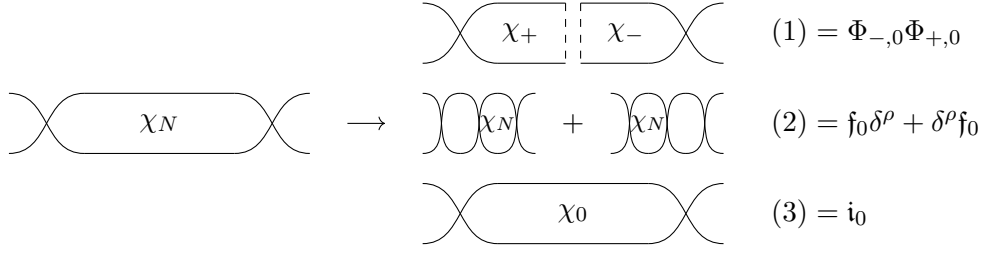

\begin{proposition}\label{chain homotopy}
With the same notations in the previous proposition, the composition
$$
\Phi_{-}\circ\Phi_{+}: HF_{int, cy}((L, \rho), (\varphi(L), \rho); K)\rightarrow HF_{int, cy}((L, \rho), (\varphi(L), \rho); K)
$$
equals the inclusion-induced map
$$
\mathfrak{i}=T^{E'}(\mathfrak{i}_{0}+ \mathfrak{i}_{1}): HF_{int, cy}((L, \rho), (\varphi(L), \rho); K)\rightarrow HF_{int, cy}((L, \rho), (\varphi(L), \rho); K).
$$
Here $E'=\tilde{E}_{+}+\tilde{E}_{-}=\norm{G_{t}}_{X}+2\norm{G_{t}}_{S}$.
\end{proposition}
\begin{proof}
The chain homotopy maps are constructed by using the perturbed Cauchy-Riemann equation with cut-off function $\chi_{N}$. Consider the one-parameter moduli spaces
$$
\widetilde{\mathcal{M}}_{0}^{para}=\bigcup_{N\in [0, +\infty)} \lbrace N\rbrace \times \mathcal{M}_{0}^{N}(p, q)
$$
and
$$
\widetilde{\mathcal{M}}_{1}^{para}=\bigcup_{N\in [0, +\infty)} \lbrace N\rbrace \times \mathcal{M}_{1}^{N}(p, q)
$$
parameterized by $N$. Here the moduli space $\mathcal{M}_{0}^{N}(p, q)$ contains solutions of the perturbed Cauchy-Riemann equation with cut-off function $\chi_{N}$ where the domain is a genuine strip. The moduli space $\mathcal{M}_{1}^{N}(p, q)$ contains solutions of the perturbed Cauchy-Riemann equation with cut-off function $\chi_{N}$ where the domain is a strip with one interior hole. We only consider the case that the above solutions represents the class of the constant map. The energy estimate in Lemma \ref{est3} tells that for a solution $u$ in $\mathcal{M}_{0}^{N}(p, q)$ or $\mathcal{M}_{1}^{N}(p, q)$, we always have that
$$
\begin{aligned}
E_{(J, \chi_{N}(\tau), G_{t})}(u) = &\int u^{*}\omega -\int_{-\infty}^{\infty}\chi_{N}'(\tau)\int_{0}^{1}G_{t}(u)dtd\tau -\int_{C(\epsilon)}G_{t}(u)\\
\leq &\int u^{*}\omega+ \norm{G_{t}}_{X}+ \norm{G_{t}}_{S}
\end{aligned}
$$
which is uniformly bounded from above, independent of $N$. Then we can compactify $\widetilde{\mathcal{M}}_{0}^{para}$ and $\widetilde{\mathcal{M}}_{1}^{para}$ to obtain $\mathcal{M}_{0}^{para}$ and $\mathcal{M}_{1}^{para}$, by adding possible broken curves. In particular, we deal with the codimension one boundary from domain degenerations in $\mathcal{M}_{1}^{N}(p, q)$ by gluing it with the moduli space where the domain is a strip with one interior marked point, as we did before.

Now we only consider the case where $p=q$ and the solutions represent the class of a constant map. Under transversality assumptions, both of the moduli spaces $\mathcal{M}_{0}^{para}(p, q)$ and $\mathcal{M}_{1}^{para}(p, q)$ have dimension one in this case. Next we study the boundary of the these two moduli spaces. By similar argument before, we exclude disk bubbles, sphere bubbles and annulus bubbles. Then the boundary components of $\mathcal{M}_{0}^{para}(p, p)$ have four types of degenerations (listed in Figure 11) and the boundary components of $\mathcal{M}_{1}^{para}(p, p)$ have seven types of degenerations (listed in Figure 12). There is another type of degenerations in $\mathcal{M}_{1}^{para}(p, p)$, where the interior circle shrinks to a point. We deal with it by using the same strategy as before, gluing this boundary component with the boundary of moduli space with one interior marked point. Hence we omit it in Figure 12.

Now we look at the chain homotopy equation
\begin{equation}
\Phi_{-}\circ\Phi_{+}- \mathfrak{i}= d^{\rho}_{K, int}\mathfrak{f}+ \mathfrak{f}d^{\rho}_{K, int}
\end{equation}
where
$$
\begin{aligned}
& \Phi_{+}=T^{\tilde{E}_{+}}(\Phi_{+,0}+ \Phi_{+,1});\\
& \Phi_{-}=T^{\tilde{E}_{-}}(\Phi_{-,0}+ \Phi_{-,1});\\
& \mathfrak{i}=T^{E'}(\mathfrak{i}_{0}+ \mathfrak{i}_{1});\\
& \mathfrak{f}=\mathfrak{f}_{0}+ \mathfrak{f}_{1};\\
& d^{\rho}_{K, int}=\delta^{\rho}+ \partial_{K, int}.
\end{aligned}
$$
We explain corresponding moduli spaces to construct the operators as follows. These operators $\Phi_{+,0}, \Phi_{+,1}, \Phi_{-,0}, \Phi_{-,1}$ are chain maps defined in the previous proposition. The operator $d^{\rho}_{K, int}=\delta^{\rho}+ \partial_{K, int}$ is the differential to define the cohomology. Operators $\mathfrak{f}_{0}, \mathfrak{f}_{1}$ will be defined as chain homotopy maps between $\Phi_{-}\circ\Phi_{+}$ and $\mathfrak{i}$.

All four operators $\mathfrak{f}_{0}, \mathfrak{f}_{1}, \mathfrak{i}_{0}, \mathfrak{i}_{1}$ are defined from $CF_{int, cy}((L, \rho), (\varphi(L), \rho); K)$ to itself, by using the perturbed Cauchy-Riemann equation with a bump function $\chi_{N}$. Their difference is from different domains and different bump functions. For $\mathfrak{i}_{0}$, the domain is a strip and the bump function is $\chi_{N}=\chi_{0}\equiv 0$. For $\mathfrak{i}_{1}$, the domain is a strip with one interior hole and the bump function is $\chi_{N}=\chi_{0}\equiv 0$. For $\mathfrak{f}_{0}$, the domain is a strip and the bump function is $\chi_{N}$. For $\mathfrak{f}_{1}$, the domain is a strip with one interior hole and the bump function is $\chi_{N}$.

Since our asymptotic conditions are the same element $p$, the operator $\mathfrak{i}_{0}$ is the identity map, which comes from the ``zero end'' moduli space $\mathcal{M}_{0}^{0}(p, p)$ as a boundary of $\mathcal{M}_{0}^{para}$. Note that when $p=q$ and $\chi_{N}=\chi_{0}\equiv 0$ the only element in $\mathcal{M}_{0}^{0}(p, p)$ is the constant map. Similarly, the only element in $\mathcal{M}_{1}^{0}(p, p)$ is the constant map. But we have the boundary condition that the interior hole is mapped to $S$, which does not intersect $L$. Hence such a degeneration will not happen. We formally define the operator $\mathfrak{i}_{1}$ as zero.

So in the full expansion of the chain homotopy equation there are 14 terms. The following three terms
$$
\partial_{K, int}\mathfrak{f}_{1}, \quad \mathfrak{f}_{1}\partial_{K, int}, \quad T^{E'}\Phi_{-,1}\Phi_{+,1} \equiv 0\mod T^{2v(K)} \quad (\text{hence} \equiv 0\mod T^{E})
$$
by energy reason. And the remaining 11 terms correspond to the 11 types of degenerations in the moduli spaces $\mathcal{M}_{0}^{para}(p, p)$ and $\mathcal{M}_{1}^{para}(p, p)$, which form the boundary components of two compact one-dimensional manifolds. Therefore we proved the chain homotopy property.
\end{proof}

\begin{figure}
  \begin{center}
  \begin{tikzpicture}[xscale=1, yscale=0.8]
  \draw (0,1) to [out=0,in=180] (1,0);
  \draw (0,0) to [out=0,in=180] (1,1);
  \draw (3,1) to [out=0,in=180] (4,0);
  \draw (3,0) to [out=0,in=180] (4,1);
  \draw (1,1)--(3,1);
  \draw (1,0)--(3,0);
  \node at (1,0.5) {$\chi_{N}$};
  \draw (2,0.5) circle [radius=0.2];
  \draw [->] (4.5,0.5)--(5,0.5);
  \draw (5.5,2.5) to [out=0,in=180] (6.5,1.5);
  \draw (5.5,1.5) to [out=0,in=180] (6.5,2.5);
  \draw (8.5,2.5) to [out=0,in=180] (9.5,1.5);
  \draw (8.5,1.5) to [out=0,in=180] (9.5,2.5);
  \draw (6.5,2.5)--(7.4,2.5);
  \draw (7.6,2.5)--(8.5,2.5);
  \draw (6.5,1.5)--(7.4,1.5);
  \draw (7.6,1.5)--(8.5,1.5);
  \draw [dashed] (7.4,2.5)--(7.4,1.5);
  \draw [dashed] (7.6,2.5)--(7.6,1.5);
  \node at (6.5,2) {\small{$\chi_{+}$}};
  \draw (6.9,2) circle [radius=0.15];
  \node at (8.2,2) {$\chi_{-}$};
  \node [right] at (10,2) {$(4')=\Phi_{-,0}\Phi_{+,1}$};
  \draw (5.5,4) to [out=0,in=180] (6.5,3);
  \draw (5.5,3) to [out=0,in=180] (6.5,4);
  \draw (8.5,4) to [out=0,in=180] (9.5,3);
  \draw (8.5,3) to [out=0,in=180] (9.5,4);
  \draw (6.5,4)--(7.4,4);
  \draw (7.6,4)--(8.5,4);
  \draw (6.5,3)--(7.4,3);
  \draw (7.6,3)--(8.5,3);
  \draw [dashed] (7.4,4)--(7.4,3);
  \draw [dashed] (7.6,4)--(7.6,3);
  \node at (8.5,3.5) {\small{$\chi_{-}$}};
  \draw (8,3.5) circle [radius=0.15];
  \node at (6.8,3.5) {$\chi_{+}$};
  \node [right] at (10,3.5) {$(4)=\Phi_{-,1}\Phi_{+,0}$};
  \draw (5.5,1) to [out=0,in=180] (6,0);
  \draw (5.5,0) to [out=0,in=180] (6,1);
  \draw (6,1) to [out=0,in=180] (6.5,0);
  \draw (6,0) to [out=0,in=180] (6.5,1);
  \draw (6.5,1) to [out=0,in=180] (7,0);
  \draw (6.5,0) to [out=0,in=180] (7,1);
  \draw (6,0.5) circle [radius=0.1];
  \node at (6.5,0.5) {\small{$\chi_{N}$}};
  \node at (7.5,0.5) {$+$};
  \draw (8,1) to [out=0,in=180] (8.5,0);
  \draw (8,0) to [out=0,in=180] (8.5,1);
  \draw (8.5,1) to [out=0,in=180] (9,0);
  \draw (8.5,0) to [out=0,in=180] (9,1);
  \draw (9,1) to [out=0,in=180] (9.5,0);
  \draw (9,0) to [out=0,in=180] (9.5,1);
  \node at (8.5,0.5) {\small{$\chi_{N}$}};
  \draw (9,0.5) circle [radius=0.1];
  \node [right] at (10,0.5) {$(5)=\mathfrak{f}_{0}\partial_{K, int}+ \partial_{K, int}\mathfrak{f}_{0}$};
  \draw (5.5,-0.5) to [out=0,in=180] (6,-1.5);
  \draw (5.5,-1.5) to [out=0,in=180] (6,-0.5);
  \draw (6,-0.5) to [out=0,in=180] (6.5,-1.5);
  \draw (6,-1.5) to [out=0,in=180] (6.5,-0.5);
  \draw (6.5,-0.5) to [out=0,in=180] (7,-1.5);
  \draw (6.5,-1.5) to [out=0,in=180] (7,-0.5);
  \draw (6,-0.8) circle [radius=0.1];
  \node at (6,-1.2) {\tiny{$\chi_{N}$}};
  \node at (7.5,-1) {$+$};
  \draw (8,-0.5) to [out=0,in=180] (8.5,-1.5);
  \draw (8,-1.5) to [out=0,in=180] (8.5,-0.5);
  \draw (8.5,-0.5) to [out=0,in=180] (9,-1.5);
  \draw (8.5,-1.5) to [out=0,in=180] (9,-0.5);
  \draw (9,-0.5) to [out=0,in=180] (9.5,-1.5);
  \draw (9,-1.5) to [out=0,in=180] (9.5,-0.5);
  \draw (9,-0.8) circle [radius=0.1];
  \node at (9,-1.2) {\tiny{$\chi_{N}$}};
  \node [right] at (10,-1) {$(5')=\delta^{\rho}\mathfrak{f}_{1}+ \mathfrak{f}_{1}\delta^{\rho}$};
  \draw (5.5,-2) to [out=0,in=180] (6.5,-3);
  \draw (5.5,-3) to [out=0,in=180] (6.5,-2);
  \draw (8.5,-2) to [out=0,in=180] (9.5,-3);
  \draw (8.5,-3) to [out=0,in=180] (9.5,-2);
  \draw (6.5,-2)--(8.5,-2);
  \draw (6.5,-3)--(8.5,-3);
  \node at (7,-2.5) {$\chi_{0}$};
  \draw (8,-2.5) circle [radius=0.2];
  \node [right] at (10,-2.5) {$(6)=\mathfrak{i}_{1}$};
  \end{tikzpicture}
  \end{center}
  \caption{Degenerations in $\mathcal{M}_{1}^{para}$.}
\end{figure}

\begin{remark}\label{analytic}
The above two propositions are proved assuming some analytic results. First, Condition \ref{condition} is necessarily used to exclude disk and annulus bubbles. Moreover, the regularity of the moduli spaces of perturbed Cauchy-Riemann equations is assumed. When the domain is a genuine strip this moduli space is discussed in \cite{FOOO4}. And we expect the same analytic argument therein can be applied here when the domain has one interior hole. In particular, we further expect that, under Condition \ref{condition}, transversality of related moduli spaces can be achieved by using domain-dependent almost complex structures or by using fixed almost complex structures via virtual perturbation.
\end{remark}

\subsection{Relations among three deformed Floer complexes}
So far we defined three complexes to describe a new version of deformed Floer cohomology, with a bulk deformation $\mathfrak{b}=wK$. For the first one, the disk model,
$$
HF_{cy}(L; (\mathfrak{b}, \rho))
$$
the underlying complex is the singular cohomology of $L$ and the differential counts holomorphic disks and holomorphic annuli, twisted by a local system $\rho$. The second one, the intersection model,
$$
HF_{int, cy}((L, \rho), (\varphi(L), \rho); K)
$$
and the third one, the chord model,
$$
HF_{cy}((L, \rho), (\varphi(L), \rho), G_{t}; K)
$$
are defined by first choosing suitable $(H_{t}, \varphi)$ and $(G_{t}, \phi)$ then counting holomorphic strips with a possible interior hole. For the genuine Floer cohomology with bulk deformations, it is known that these three cohomology theories are equivalent over the Novikov field $\Lambda$ (Proposition 8.24 \cite{FOOO2}) and have a good Lipschitz property over the Novikov ring $\Lambda_{0}$ (Theorem 6.2 \cite{FOOO4}). Now we will discuss the relations among these three models in our setting.

The disk model, of which the cohomology is determined by the potential function, is used for concrete computation once we know the potential function. The displacement results are given by the change of torsion exponents of the chord model, where large Hamiltonian perturbation is allowed. And to connect these two models we need the intersection model, where only small Hamiltonian perturbation is considered.

\begin{proposition}\label{cri}
Suppose that the potential function $\mathfrak{PO}^{cy, \mathfrak{b}}(\rho)$ for $L$ has a critical point for some $(\mathfrak{b}, \rho)$ modulo $T^{E'}$, $E'\leq E$. If there is a Hamiltonian $G_{t}$ with time-one flow $\phi$ such that $L\cap \phi(L)=\emptyset$ then it satisfies that $\norm{G_{t}}_{X}+ 2\norm{G_{t}}_{S}\geq E'$.
\end{proposition}
\begin{proof}
First the existence of the critical point shows that
$$
HF_{cy}(L; (\mathfrak{b}, \rho))\cong H^{*}(L; \dfrac{\Lambda_{0}}{T^{E'}\Lambda_{0}})\cong (\dfrac{\Lambda_{0}}{T^{E'}\Lambda_{0}})^{\oplus 8}\neq \lbrace 0\rbrace
$$
by Proposition \ref{decom}.

Next we choose a $C^{2}$-small $(H_{t}, \varphi)$ such that $L\cap \varphi(L)$ is transversal. Then the cohomology
$$
HF_{int, cy}((L, \rho), (\varphi(L), \rho); K)
$$
is well-defined for $(\mathfrak{b}=wK, \rho)$. We can construct chain maps between the two theories $HF_{cy}(L; (\mathfrak{b}, \rho))$ and $HF_{int, cy}((L, \rho), (\varphi(L), \rho); K)$. In the case of genuine Floer cohomology with bulk deformations, the chain maps are constructed in Section 8 \cite{FOOO2}. So we combine the proof therein with the special case when the domain has one interior hole in the previous subsection, to get the chain maps and chain homotopies with new energy estimates. Note that $H_{t}$ is $C^{2}$-small, there exists a $J'$ for $\varphi(L)$ to satisfy Condition \ref{condition} and $J'$ is in the same component with $J$ which makes $L$ satisfy Condition \ref{condition}. Moreover $\varphi(L)\cap S$ is empty. Hence the discussion in previous subsections all works. Then we obtain that
$$
HF_{int, cy}((L, \rho), (\varphi(L), \rho); K)\cong \bigoplus_{i=1}^{8} (\dfrac{\Lambda_{0}}{T^{E_{i}}\Lambda_{0}})
$$
where $\abs{E'-E_{i}}< \norm{H_{t}}_{X}+2\norm{H_{t}}_{S}$ for all $i$. That is, under the small perturbation $H_{t}$ the torsion exponents are also slightly perturbed, by the amount of some Hofer norms.

Therefore we have transited from the disk model to the intersection model. Next the estimates in previous subsection help us to transit from the intersection model to the chord model, where large Hamiltonian perturbation is allowed. Suppose that there is a Hamiltonian $G_{t}$ with time-one flow $\phi$ such that $L\cap \phi(\varphi(L))=\emptyset$. From the definition we know that
$$
HF_{cy}((L, \rho), (\varphi(L), \rho), G_{t}; K)=\lbrace 0\rbrace
$$
and $\Phi_{+}=\Phi_{-}=0$. Proposition \ref{chain homotopy} tells that
$$
\Phi_{-}\circ\Phi_{+}: HF_{int, cy}((L, \rho), (\varphi(L), \rho); K)\rightarrow HF_{int, cy}((L, \rho), (\varphi(L), \rho); K)
$$
equals the inclusion-induced map
$$
T^{E_{0}}(\mathfrak{i}_{0}+\mathfrak{i}_{1}): HF_{int, cy}((L, \rho), (\varphi(L), \rho); K)\rightarrow HF_{int, cy}((L, \rho), (\varphi(L), \rho); K)
$$
where $E_{0}=\norm{G_{t}}_{X}+2\norm{G_{t}}_{S}$. Therefore we have that
$$
0=T^{E_{0}}(\mathfrak{i}_{0}+ \mathfrak{i}_{1}): HF_{int, cy}((L, \rho), (\varphi(L), \rho); K)\rightarrow HF_{int, cy}((L, \rho), (\varphi(L), \rho); K).
$$
So $E_{0}> \max_{i}\lbrace E_{i}\rbrace$ for all $i$. Let $\norm{H_{t}}\rightarrow 0$ we obtain that $\norm{G_{t}}_{X}+ 2\norm{G_{t}}_{S}\geq E'$.

In conclusion, for any Hamiltonian diffeomorphism $\psi$ which displaces $L$ there is a small amount $\epsilon_{\psi} >0$ such that for any pair $(H_{t}, \varphi)$ with $\norm{H_{t}}<\epsilon_{\psi}$, the diffeomorphism $\psi$ also displaces $\varphi(L)$ from $L$. Hence we can use those small $(H_{t}, \varphi)$ to do the above energy estimate for $\psi$, which completes the proof.
\end{proof}

The above theorem is parallel to Theorem 5.11 in \cite{FOOO1} for potential functions without bulk deformation and Theorem 7.7 in \cite{FOOO4} for potential functions with bulk deformation. We just adapt the proof therein by using our energy estimates in this section.

\section{Estimates of displacement energy}
In this section we estimate the displacement energy of a local torus. We fix a triple $(X, S, U)$ and a local torus $L$ inside $U$ as in Theorem \ref{dis}, such that $L$ satisfies Condition \ref{condition}. See Subsection 4.1 for the fibration structure on $U$ of which $L$ is a fiber. We assume that $S$ is integrally homologically trivial, and fix a 4-chain $K$ such that $\partial K=S$.

\subsection{First estimate for $\mathcal{E}_{L}$}
Let $L$ be a local torus, we will first show its displacement energy is greater than or equal to $E_{5}$. This is directly from the decomposition formula of the Floer cohomology, which do not need the bulk deformation by the chain $K$.

Assuming $(1)-(3)$ in Condition \ref{condition}, the one-pointed open Gromov-Witten invariant $n_{\beta}$ is defined, for any Maslov two disk class $\beta\in\pi_{2}(X, L)$ with energy less than $E_{+}$. We consider the sequence
$$
\lbrace \beta_{k}\mid n_{\beta}\neq 0, E(\beta_{k})\leq E(\beta_{k+1})\rbrace_{k=1}^{\infty}
$$
of disk classes with Maslov index two, enumerated by their symplectic energy. From the local study we know that $L_{\lambda}$ bounds four $J$-holomorphic disks with Maslov index two inside $U$, with same energy $E_{1}$. Those are the first four elements in the above sequence if $L$ is near $S$. Let $E_{5}=E(\beta_{5})$ be the least energy of outside disk contribution.

Without bulk deformation, the disk potential function is
$$
\mathfrak{PO}(\rho)=(x+y^{-1}+xz^{-1}+y^{-1}z) T^{E_{1}} \mod T^{E_{5}},
$$
which has a critical point at $\rho_{0}=(x=1, y=1, z=-1)$. Hence we have
$$
HF(L, \rho_{0}; \Lambda_{0})\cong (\bigoplus_{i=1}^{8} \dfrac{\Lambda_{0}}{T^{E}\Lambda_{0}}) \mod T^{E_{5}}
$$
by the decomposition formula. And Theorem J in \cite{FOOO} gives that $\mathcal{E}_{L}\geq E_{5}$.

\subsection{Second estimate for $\norm{G_{t}}_{X}+ 2\norm{G_{t}}_{S}$}
For the second estimate we will use the deformed Floer cohomology of a local torus, which was constructed in previous sections.

Now we assume that $L$ further satisfies Condition \ref{condition+}. From equation (\ref{po}), we have the following deformed potential function
\begin{equation}
\begin{aligned}
&\mathfrak{PO}^{cy, \mathfrak{b}}(\rho)\\
=&[(1+(1+n_{\beta_{1}}^{cy})w)x+ (1+n_{\beta_{2}}^{cy}w)y^{-1}+ (1+n_{\beta_{3}}^{cy}w)xz^{-1}+ (1+n_{\beta_{4}}^{cy}w)y^{-1}z]T^{E_{1}}+\\
+& \sum_{\mu(\beta)=2, \omega(\beta)=E_{1}, \beta\neq\beta_{i}}n_{\beta}^{cy}f_{\beta}(x,y,z)T^{E_{1}}+ H(w, x, y, z, T) \mod T^{E}
\end{aligned}
\end{equation}
where $H(w, x, y, z, T) $ are higher energy terms. Recall that $E:=\min\lbrace E_{S}+ v(w), 2v(w), E_{+}\rbrace$ for a Lagrangian torus satisfying Condition \ref{condition}.

By Proposition \ref{cri}, if the above function has a critical point modulo $T^{E'}$ with $E'<E$, then we have the estimate
$$
\norm{G_{t}}_{X}+ 2\norm{G_{t}}_{S}\geq E'.
$$
Hence it suffices to analyze the critical points of this deformed potential function. The idea is to find some critical points for the low energy terms, then apply an implicit function theorem to deduce the existence of critical points globally.

\begin{lemma}
Consider a vector-valued Laurent polynomial function
$$
F=(f_{1}, \cdots, f_{n}): \Lambda_{0}^{n}\rightarrow \Lambda_{0}^{n}; \quad f_{i}\in \Lambda_{0}[x_{1}^{\pm 1}, \cdots, x_{n}^{\pm 1}], \quad \forall 1\leq i\leq n.
$$
We assume that $F$ has a decomposition by the valuation on $\Lambda_{0}$
$$
F=F_{0}+ H, \quad v(F_{0})=0, \quad v(H)>\epsilon
$$
for some $\epsilon>0$. If $F_{0}=0$ has a nondegenerate solution at some point $(x_{1}, \cdots, x_{n})\in \mathbb{C}^{n}$ then $F=0$ has a solution at $(x_{1}', \cdots, x_{n}')\in \Lambda_{0}^{n}$. Moreover $(x_{1}, \cdots, x_{n})\equiv (x_{1}', \cdots, x_{n}') \mod T^{\epsilon}$.
\end{lemma}

This lemma is an implicit function theorem in the setting of the Novikov ring, see Section 10 in \cite{FOOO1} for a proof. Now we look at the critical points equation of the deformed potential function. Note that even in the low energy level, there are some unknown terms due to cylinder contributions. And we will use $(7), (8)$ in Condition \ref{condition+} to control these cylinder contributions. First, assuming Condition \ref{condition+} $(8)$ gives
$$
n_{\beta}^{cy}=0, \quad \text{if} \quad \mu(\beta)=2, \omega(\beta)<E_{5}, \beta\neq\beta_{i}.
$$
Then the deformed potential function becomes
$$
\begin{aligned}
&\mathfrak{PO}^{cy, \mathfrak{b}}(\rho)\\
=&[(1+(1+n_{\beta_{1}}^{cy})w)x+ (1+n_{\beta_{2}}^{cy}w)y^{-1}+ (1+n_{\beta_{3}}^{cy}w)xz^{-1}+ (1+n_{\beta_{4}}^{cy}w)y^{-1}z]T^{E_{1}}+\\
+&  H(w, x, y, z, T) \mod T^{E}.
\end{aligned}
$$
And the critical points equation will be
\begin{equation}\label{crieq}
\begin{aligned}
&0=\partial_{x}\mathfrak{PO}^{cy, \mathfrak{b}}(\rho)=[(1+(1+n_{\beta_{1}}^{cy})w)+ (1+n_{\beta_{3}}^{cy}w)z^{-1}]T^{E_{1}}+ \dfrac{\partial H}{\partial x} \mod T^{E}\\
&0=\partial_{y}\mathfrak{PO}^{cy, \mathfrak{b}}(\rho)=[-(1+n_{\beta_{2}}^{cy}w)y^{-2}- (1+n_{\beta_{4}}^{cy}w)y^{-2}z]T^{E_{1}} +\dfrac{\partial H}{\partial y} \mod T^{E}\\
&0=\partial_{z}\mathfrak{PO}^{cy, \mathfrak{b}}(\rho)=[-(1+n_{\beta_{3}}^{cy}w)xz^{-2}+ (1+n_{\beta_{4}}^{cy}w)y^{-1}]T^{E_{1}}+ \dfrac{\partial H}{\partial z} \mod T^{E}.
\end{aligned}
\end{equation}
We view (\ref{crieq}) as a system of three equations with four variables $(w, x, y, z)$ hence we have freedom to prescribe the value of one of the variables. So we set $x=1$ to these equations and view $w, y, z$ as variables. Then the low energy part of (\ref{crieq}) is equivalent to
\begin{equation}\label{lead}
\begin{aligned}
&0=(1+(1+n_{\beta_{1}}^{cy})w)+ (1+n_{\beta_{3}}^{cy}w)z^{-1}\\
&0=(1+n_{\beta_{2}}^{cy}w)+ (1+n_{\beta_{4}}^{cy}w)z\\
&0=-(1+n_{\beta_{3}}^{cy}w)z^{-2}+ (1+n_{\beta_{4}}^{cy}w)y^{-1}.
\end{aligned}
\end{equation}
Next we check that $w=0, y=1, z=-1$ is a solution of (\ref{lead}) with Jacobian determinant
$$
d= 1+n_{\beta_{1}}^{cy}+n_{\beta_{2}}^{cy}-n_{\beta_{3}}^{cy}-n_{\beta_{4}}^{cy}.
$$
Hence by assuming Condition \ref{condition+} $(7)$ that $d\neq 0$ then $w=0, y=1, z=-1$ becomes a non-degenerate solution. By the Gromov compactness theorem the higher energy part $H$ in the potential function is a Laurent polynomial since we work modulo $T^{E}$. (In general the potential function could be a Laurent series with energy going to infinity.) Hence our system of equations fits in the above lemma and the whole system of critical point equation has a suitable solution modulo $T^{E}$.

The solution $w=0, y=1, z=-1$ is just a solution to the low energy part (\ref{lead}). The implicit function theorem assures that there is a solution to (\ref{crieq}). In particular, this global solution looks like $w=0+w_{+}, y=1+y_{+}, z=-1+z_{+}$ with $w_{+}, y_{+}, z_{+} \in \Lambda_{+}$. Now we estimate $v(w)=v(w_{+})$ in order to estimate $E:=\min\lbrace E_{S}+ v(w), 2v(w), E_{+}\rbrace$.

The higher energy terms $H(w, x, y, z, T)$ contains contributions from disks and cylinders. By our assumption, there are no disk or cylinder contributions between energy $E_{1}$ and $E_{5}$. So we only need to perturb the higher energy terms with energy larger than or equal to $E_{5}$. The proof of the implicit function theorem is a step-by-step induction. In particular, a careful study of the construction of $w_{+}$ shows that
$$
w_{+}= c_{1}w_{1,+}+ c_{2}w_{2,+}+\cdots +c_{j}w_{j,+}, \quad c_{i}\in\mathbb{C}, w_{i,+}\in \Lambda_{+}, v(w_{i,+})<v(w_{i+1,+}).
$$
Since the first possible high energy term has energy $E_{5}$, the valuation $v(w_{1,+})=E_{5}-E_{1}$. If $c_{1}\neq 0$ we have that $v(w)=E_{5}-E_{1}$. If $c_{1}= 0$ we consider the next term $c_{2}w_{2,+}$. The extremal case is that all $c_{i}=0$. This means that we don't need to deform the potential function by the chain $wK$. Hence we can directly use the usual Floer cohomology with local systems to get the energy estimate. If some $c_{i}\neq 0$ we have $v(w)\geq E_{5}-E_{1}$.

Therefore we know that the critical points equation (\ref{crieq}) has a solution modulo $T^{E}$. By Proposition \ref{cri} we know that if $\phi$ displaces $L$ then its corresponding Hamiltonian functions $G_{t}$ satisfy that
$$
\norm{G_{t}}_{X}+ 2\norm{G_{t}}_{S}\geq E,
$$
which completes the proof of Theorem \ref{dis}.

Now we explain the proof of Corollary \ref{sphere}. Let $G_{t}$ be a time-dependent Hamiltonian function and $\phi$ be its time-one map such that $S\cap \phi(S)=\emptyset$. Then there is a small neighborhood $U$ which is also displaced by $\phi$. Note that for a small number $\lambda'$, all local tori $L_{\lambda}$ are contained in $U$ if $\lambda\in (0, \lambda')$ and are displaced by $\phi$. Therefore we know that
$$
\norm{G_{t}}_{X} \geq  E_{5,\lambda}, \quad \norm{G_{t}}_{X} +2\norm{G_{t}}_{S} \geq  2 (E_{5,\lambda}- E_{1,\lambda})
$$
for all $\lambda\in (0, \lambda')$. As $\lambda$ goes to zero, the energy $E_{1,\lambda}$ decreases and $E_{5,\lambda}$ increases hence we complete the proof of Corollary \ref{sphere}.

\subsection{Examples of displaceable Lagrangian spheres}
Now we briefly review Pabiniak's construction \cite{P} of displaceable Lagrangian 3-spheres and formally show that our theoretical estimate is almost optimal in this case. We remark that Condition \ref{condition+} is not verified for these examples, although we expect they satisfy it.

Consider the Lie group $SU(3)$. We identify the dual of its Lie algebra $\mathfrak{su}^{*}(3)$ with the vector space of $3\times 3$ traceless Hermitian matrices. Then the group $SU(3)$ acts on $\mathfrak{su}^{*}(3)$ by conjugation. Through a regular point $diag(a, b, -a-b)$, the action orbit $M$ is a smooth 6-dimensional symplectic manifold with the Kostant-Kirillov symplectic form.

We fix a regular point $diag(a, b, -a-b)$ with $a>b\geq 0$ and write the orbit as $M$. The symplectic form on $M$ is monotone if and only if $b=0$. There is a Gelfand-Tsetlin fibration $\Gamma: M\rightarrow \mathbb{R}^{3}$. For a matrix $A\in M$ let $a_{1}(A)\geq a_{2}(A)$ denote the two eigenvalues of the $2\times 2$ top left minor of $A$, and let $a_{3}(A)=a_{11}$ be the $(1,1)$ entry of $A$. Then the system $\Gamma(A)= (a_{1}(A), a_{2}(A), a_{3}(A))$ gives the fibration map. Let $(x,y,z)$ be the coordinates of $\mathbb{R}^{3}$. The image polytope (see Figure 13) of $\Gamma$ is given by affine functions
\begin{enumerate}
\item[] $a\geq x\geq b$;
\item[] $b\geq y\geq -a-b$;
\item[] $x\geq z\geq y$.
\end{enumerate}
This Gelfand-Tsetlin fibration $\Gamma$ can be viewed as a smooth torus fibration away from the fiber $\Gamma^{-1}(b,b,b)$ since the three functions $(a_{1}, a_{2}, a_{3})$ integrate to a 3-torus action. There is a unique non-smooth point $(b,b,b)$ in the polytope, of which the fiber $S=\Gamma^{-1}(b,b,b)$ is a smooth Lagrangian 3-sphere. So this fibration is a compactification of the fibration on $T^{*}S^{3}$ by putting divisors at infinity, see Section 4.1. And the parameters $a$ and $b$ measure the symplectic form on this compactification.

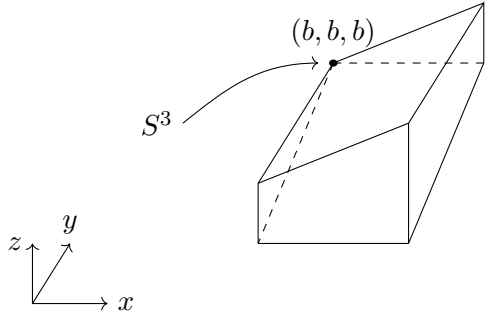
\begin{figure}
  \begin{tikzpicture}[xscale=1, yscale=0.8]
  \draw [dashed] (0,0)--(2,0);
  \draw (2,0)--(1,-3);
  \draw [dashed] (0,0)--(-1,-3);
  \draw (-1,-3)--(1,-3);
  \draw (0,0)--(-1,-2)--(-1,-3);
  \draw (2,0)--(2,1)--(1,-1)--(1,-3);
  \draw (1,-1)--(-1,-2);
  \draw (0,0)--(2,1);
  \draw [->] (-4,-4)--(-3.5,-3);
  \draw [->] (-4,-4)--(-3,-4);
  \draw [->] (-4,-4)--(-4,-3);
  \node [left] at (-4,-3) {$z$};
  \node [above] at (-3.5,-3) {$y$};
  \node [right] at (-3,-4) {$x$};
  \filldraw (0,0) circle [radius=0.05];
  \node at (0,0.5) {$(b,b,b)$};
  \draw [->] (-2,-1) to [in=180,out=45] (-0.2,0);
  \node [left] at (-2,-1) {$S^{3}$};
  \end{tikzpicture}
  \caption{Moment polytope for the fibration $\Gamma$.}
\end{figure}

Moreover we can consider the standard action of the maximal torus of $SU(3)$, which gives us a subaction of the Gelfand-Tsetlin action. This 2-torus action has a moment map $\mu: M\rightarrow \mathbb{R}^{2}$. We have the following commutative diagram
$$
\begin{tikzcd}
M \arrow{dr}{\mu} \arrow{r}{\Gamma} & \mathbb{R}^{3} \arrow{d}{pr}\\
                                    & \mathbb{R}^{2}
\end{tikzcd}
$$
where we view $\mathbb{R}^{2}=\lbrace x+y+z=0\rbrace\subset \mathbb{R}^{3}$. The projection map is given by
$$
pr(x, y, z)=(z, x+y-z, -x-y).
$$

Consider the permutation matrix
\[
P=\begin{bmatrix}
    -1       & 0 & 0 \\
    0       & 0 & 1 \\
    0       & 1 & 0
\end{bmatrix}
\]
which is an element of $SU(3)$. Then the conjugation with $P$ is a Hamiltonian action on $M$. Note that for $A=[a_{ij}]\in M$
$$
\mu(PAP^{-1})=(a_{11}, a_{33}, a_{22}).
$$
So we have that
$$
\mu(S)=\mu(\Gamma^{-1}(b,b,b))=pr(b,b,b)=(b,b,-2b)
$$
and
$$
\mu(PSP^{-1})=(b,-2b,b).
$$
In particular if $b\neq 0$ then the Lagrangian 3-sphere $S$ will be displaced by this group action. We also remark that when $b=0$ the Lagrangian 3-sphere $S$ is monotone and is proved to be nondisplaceable by Cho-Kim-Oh \cite{CKO}.

In \cite{NNU} it is calculated that $S$ bounds two holomorphic disks with energy $2\pi(a+2b)$ and $2\pi(a-b)$. Moreover the Floer cohomology $HF(S,S; \Lambda)$ vanishes. Next we assume that $b>0$ so that $2\pi(a+2b)> 2\pi(a-b)$. By Chekanov's theorem (see Main Theorem in \cite{Ch2}) the displacement energy $\mathcal{E}_{S}$ of $S$ is greater than $2\pi(a-b)$. For the Hamiltonian action by $P$, its corresponding Hamiltonian function is the inner product with the vector $diag(0,\pi,-\pi)$. That is, for a fiber $\Gamma^{-1}(x,y,z)$ over the point $(x,y,z)$ the Hamiltonian function is constant on the fiber and can be written as
$$
H(x,y,z)=(0,\pi,-\pi)\cdot pr(x,y,z)=\pi(2x+2y-z).
$$
From the polytope we can check that
$$
\max_{M} H=H(a,b,b)=\pi(2a+b), \quad \min_{M} H=H(b,-a-b,b)=\pi(-2a-b).
$$
Hence we have that
$$
\int_{0}^{1} (\max_{M}H- \min_{M}H) dt= 2\pi(2a+b).
$$
In particular $H\mid_{S}\equiv H(b,b,b)=3b$. So for this Hamiltonian we have that
$$
\norm{H}_{M}=2\pi(2a+b), \quad \norm{H}_{S}=0
$$
and
$$
\norm{H}_{M}+ 2\norm{H}_{S}=\norm{H}_{M}=2\pi(2a+b) \geq 2E_{5}:= \lim_{\lambda\rightarrow 0}2E_{5,\lambda}= 4\pi(a-b).
$$
This matches our theoretical prediction in Theorem \ref{dis}. And when $a\gg b\geq 0$ we have that $2\pi(2a+b)$ is close to $4\pi(a-b)$, which shows that the estimate is almost optimal in this case.

One can also check the case of the displaceable Lagrangian $S^{3}\subset \mathbb{C}^{2}\times \mathbb{C}P^{1}$. Consider the following Lagrangian embedding
$$
S^{3}\rightarrow \mathbb{C}^{2}\times \mathbb{C}P^{1}, \quad x \mapsto (i(x), -h(x))
$$
where $i$ is the inclusion of the unit sphere and $h$ is the Hopf map. The symplectic form on $\mathbb{C}^{2}\times \mathbb{C}P^{1}$ is the standard one times the Fubini-Study form. Let $H$ be a Hamiltonian on $\mathbb{C}^{2}$ which displaces the unit sphere and $G(z_{1}, z_{2}):=H(z_{1})$ be a Hamiltonian on $\mathbb{C}^{2}\times \mathbb{C}P^{1}$. Then $G$ displaces the Lagrangian sphere and $\norm{G}_{\mathbb{C}^{2}\times \mathbb{C}P^{1}}=\norm{H}_{\mathbb{C}^{2}}$. Moreover, it is known that $\norm{H}_{\mathbb{C}^{2}}$ can be chosen to be arbitrarily close to $\pi$.

However, the Hamiltonian $H$ takes maximal and minimal values on the unit sphere hence $G$ takes maximal and minimal values on the Lagrangian sphere $S$. So we have $\norm{G}_{\mathbb{C}^{2}\times \mathbb{C}P^{1}}= \norm{G}_{S}$. Note that
$$
H_{2}(\mathbb{C}^{2}\times \mathbb{C}P^{1}, S)\cong H_{2}(\mathbb{C}^{2}\times \mathbb{C}P^{1})\cong H_{2}(\mathbb{C}P^{1})
$$
hence the minimal energy of a holomorphic disk bounding $S$ is $\pi=\int_{\mathbb{C}P^{1}}\omega_{FS}$. And our estimate gives that $3\norm{G}_{\mathbb{C}^{2}\times \mathbb{C}P^{1}}\geq 2\pi$, which is not a contradiction but not very powerful for this example.

\bibliographystyle{amsplain}

\end{document}